\documentclass[letterpaper,twoside,10pt]{article}

\pdfoutput=1

\usepackage[letterpaper,left=1in,right=1in,top=1in,bottom=1in]{geometry}

\newcommand{\myauthor}{Benjamin Antieau, Lennart Meier, and Vesna Stojanoska}
\newcommand{\mytitle}{Picard sheaves, local Brauer groups, and topological
modular forms}

\author{\myauthor}
\title{\mytitle}
\date{}

\usepackage{mathrsfs}
\usepackage{graphicx}
\usepackage{rotating}
\usepackage{float}
\usepackage{caption}
\usepackage{pdflscape}
\usepackage{url}
\usepackage[all,arc,2cell]{xy}
\UseAllTwocells
\usepackage{enumerate}
 \usepackage{lineno}
\usepackage{stix}
\usepackage{accents}
\usepackage{makecell}
\usepackage{relsize}
\usepackage{tikz}
\usepackage{amsmath,amsthm}

\usepackage{microtype}
\usepackage{amsthm}
\usepackage{stmaryrd}

\usepackage{marginnote}

\usepackage{spectralsequences}
\usepackage[nomessages]{fp}
\usepackage{ifthen}

\usepackage[pdfstartview=FitH,
            pdfauthor={\myauthor},
            pdftitle={\mytitle},
            colorlinks,
            linkcolor=reference,
            citecolor=citation,
            urlcolor=e-mail,
            backref]{hyperref}

\usepackage[nameinlink, capitalise]{cleveref}

\usepackage{fancyhdr}
\usepackage{color}
\definecolor{todo}{rgb}{1,0,0}
\definecolor{conditional}{rgb}{0,1,0}
\definecolor{e-mail}{rgb}{0,.40,.80}
\definecolor{reference}{rgb}{.20,.60,.22}
\definecolor{mrnumber}{rgb}{.80,.40,0}
\definecolor{citation}{rgb}{0,.40,.80}

\let\oldmarginpar\marginpar
\renewcommand\marginpar[1]{\-\oldmarginpar[\raggedleft\footnotesize #1]%
{\raggedright\footnotesize #1}}

\renewcommand{\mathcal}[1]{\mathscr{#1}}
\newcommand{\Ascr}{\mathcal{A}}
\newcommand{\Bscr}{\mathcal{B}}
\newcommand{\Cscr}{\mathcal{C}}
\newcommand{\Dscr}{\mathcal{D}}
\newcommand{\Escr}{\mathcal{E}}
\newcommand{\Fscr}{\mathcal{F}}
\newcommand{\Gscr}{\mathcal{G}}
\newcommand{\Hscr}{\mathcal{H}}
\newcommand{\Iscr}{\mathcal{I}}
\newcommand{\Kscr}{\mathcal{K}}
\newcommand{\Lscr}{\mathcal{L}}
\newcommand{\Mscr}{\mathcal{M}}
\newcommand{\MM}{\mathcal{M}}
\newcommand{\Oscr}{\mathcal{O}}
\newcommand{\Pscr}{\mathcal{P}}
\newcommand{\Qscr}{\mathcal{Q}}

\newcommand{\Uscr}{\mathcal{U}}
\newcommand{\Wscr}{\mathcal{W}}
\newcommand{\Xscr}{\mathcal{X}}
\newcommand{\Yscr}{\mathcal{Y}}

\newcommand{\B}{\mathrm{B}}

\newcommand{\E}{\mathrm{E}}
\newcommand{\F}{\mathrm{F}}
\renewcommand{\H}{\mathrm{H}}
\newcommand{\K}{\mathrm{K}}
\renewcommand{\L}{\mathrm{L}}

\newcommand{\R}{\mathrm{R}}

\renewcommand{\AA}{\mathbb{A}}
\newcommand{\CC}{\mathbb{C}}
\newcommand{\EE}{\mathbb{E}}
\newcommand{\FF}{\mathbb{F}}

\newcommand{\GG}{\mathbb{G}}

\newcommand{\QQ}{\mathbb{Q}}
\newcommand{\Q}{\mathbb{Q}}
\newcommand{\RR}{\mathbb{R}}
\renewcommand{\SS}{\mathbb{S}}
\newcommand{\ZZ}{\mathbb{Z}}
\newcommand{\Z}{\mathbb{Z}}

\newcommand{\WW}{\mathbb{W}}

\newcommand{\gr}{\mathrm{gr}}

\newcommand{\op}{\mathrm{op}}
\DeclareMathOperator{\Gal}{Gal}
\DeclareMathOperator{\tors}{tors}

\DeclareMathOperator{\res}{res}

\newcommand{\OO}{\mathcal{O}}

\newcommand{\Ext}{\mathrm{Ext}}

\newcommand{\Pic}{\mathrm{Pic}}
\newcommand{\Br}{\mathrm{Br}}
\newcommand{\LBr}{\mathrm{LBr}}

\newcommand{\ShLBrW}{\mathbf{BrW}}
\newcommand{\LBrW}{\mathrm{LBrW}}
\newcommand{\ShBr}{\mathbf{Br}}
\newcommand{\Shpic}{\mathbf{pic}}
\newcommand{\ShLBr}{\mathbf{LBr}}

\newcommand{\ShPic}{\mathbf{Pic}}
\newcommand{\sPic}[1]{\ShPic_{\Oscr_{#1}}}
\newcommand{\ShB}{\mathbf{B}}
\newcommand{\ShBPic}{\mathbf{BPic}}
\newcommand{\Shbr}{\mathbf{br}}
\newcommand{\Shlbr}{\mathbf{lbr}}

\newcommand{\Aut}{\mathrm{Aut}}

\newcommand{\End}{\mathrm{End}}
\newcommand{\Hom}{\mathrm{Hom}}
\newcommand{\Map}{\mathrm{Map}}
\newcommand{\ShMap}{\mathbf{Map}}
\newcommand{\ShEnd}{\mathbf{End}}

\newcommand{\QCoh}{\mathrm{QCoh}}
\newcommand{\Mod}{\mathrm{Mod}}
\newcommand{\Shv}{\mathrm{Shv}}
\newcommand{\Sp}{\mathrm{Sp}}

\newcommand{\Cat}{\mathrm{Cat}}

\newcommand{\CAlg}{\mathrm{CAlg}}
\newcommand{\PrL}{\mathrm{Pr}^\L}

\newcommand{\GL}{\mathrm{GL}}

\newcommand{\Gm}{\mathbb{G}_{m}}

\newcommand{\Tmf}{\mathrm{Tmf}}
\newcommand{\TMF}{\mathrm{TMF}}
\newcommand{\tmf}{\mathrm{tmf}}
\newcommand{\KU}{\mathrm{KU}}
\newcommand{\KO}{\mathrm{KO}}
\newcommand{\MU}{\mathrm{MU}}

\newcommand{\et}{\text{\'et}}

\newcommand{\coker}{\mathrm{coker}}
\newcommand{\tensor}{\otimes}

\DeclareMathOperator{\SpetLurie}{Sp\acute{e}t}
\DeclareMathOperator{\Spet}{Spec}
\DeclareMathOperator{\Spec}{Spec}
\DeclareMathOperator{\Spf}{Spf}

\newcommand{\we}{\simeq}
\newcommand{\iso}{\cong}

\numberwithin{equation}{section}
\theoremstyle{plain}
\newtheorem{theorem}[equation]{Theorem}
\newtheorem*{theorem*}{Theorem}
\newtheorem{lemma}[equation]{Lemma}

\newtheorem{proposition}[equation]{Proposition}

\newtheorem{corollary}[equation]{Corollary}

\newtheoremstyle{named}{}{}{\itshape}{}{\bfseries}{.}{.5em}{#1 \thmnote{#3}}
\theoremstyle{named}

\theoremstyle{definition}
\newtheorem{definition}[equation]{Definition}

\newtheorem{example}[equation]{Example}
\newtheorem{question}[equation]{Question}
\newtheorem{construction}[equation]{Construction}
\newtheorem{warning}[equation]{Warning}
\newtheorem{remark}[equation]{Remark}

\begin{document}

\maketitle

\begin{abstract}
    \noindent
    We prove that the Brauer group of $\TMF$ is isomorphic to the Brauer group
    of the derived moduli stack of elliptic curves. Then, we compute the local
    Brauer group, i.e., the subgroup of the Brauer group of elements
    trivialized by some \'etale cover of the moduli stack, up to a finite $2$-torsion group.


    \paragraph{Mathematics Subject Classification 2010:}
    \href{http://www.ams.org/mathscinet/msc/msc2010.html?t=14Fxx&btn=Current}{14F22},
    \href{http://www.ams.org/mathscinet/msc/msc2010.html?t=14Hxx&btn=Current}{14H52},
    \href{http://www.ams.org/mathscinet/msc/msc2010.html?t=14Kxx&btn=Current}{14K10},
    \href{http://www.ams.org/mathscinet/msc/msc2010.html?t=55Nxx&btn=Current}{55N34}.
\end{abstract}

\tableofcontents

\section{Introduction}

The Brauer group $\Br(R)$ of an $\EE_\infty$-ring spectrum $R$ was introduced by Baker--Richter--Szymik~\cite{baker-richter-szymik} following
previous work of Baker--Lazarev~\cite{baker-lazarev} and To\"en~\cite{toen-derived}.
The group classifies Azumaya algebras over $R$ up to Morita equivalence; 
equivalently it classifies invertible $R$-linear stable $\infty$-categories. These can be seen as twisted versions of $R$-modules and thus $\Br(R)$ classifies all possible twists of $\Mod_R$. One can actually replace $\Mod_R$ here by any symmetric monoidal $\infty$-category, like quasi-coherent sheaves on a scheme or stack. In the most classical case of vector spaces over a field $k$, Azumaya algebras are just central simple algebras (i.e.\ matrix algebras over a central division algebra) and the corresponding Brauer group was introduced by Brauer around 1930. 

Classically, Brauer groups can often be computed as \'etale cohomology groups. They thus allow
cohomological control of natural occurrences of Azumaya algebras (e.g.\ as endomorphism algebras of
representations \cite[Section 12.2]{serre-representations}) or twisted sheaves (like in the theory
of moduli of stable sheaves \cite{caldararu}). Another $\infty$-categorical example is given by the
relevance of twists of parametrized spectra in Seiberg--Witten Floer homotopy theory
\cite{douglas-twisted}. On the other hand, Brauer groups also allow algebraic or geometric
interpretations of cohomology classes, as utilized e.g.\ in the classic Artin--Mumford example of a
non-rational unirational variety \cite{artin-mumford} or the Merkurjev--Suslin theorem
\cite{gille-szamuely}. Brauer groups give also one of the approaches to class field theory
\cite{roquette,weil-basic-number-theory, milne-class} and form the basis of the Brauer--Manin
obstruction for rational points \cite{brauer-grothendieck-group}. Thus, the study of Brauer groups
of ring spectra might be interesting for possible theories of \'etale cohomology on
$\EE_{\infty}$-ring spectra, and can be seen as a contribution to the nascent subject of arithmetic
of $\EE_{\infty}$-ring spectra. Moreover, the Brauer space provides a natural delooping of the
Picard space, like the Picard space is a natural delooping of the space of units of an
$\EE_{\infty}$-ring spectrum.

When $R$ is a connective $\EE_{\infty}$-ring spectrum, $\Br(R)$ depends only on $\pi_0R$ and
\begin{equation}\label{eq:connective}\Br(R)\iso\H^1(\Spec\pi_0R,\ZZ)\times\H^2(\Spec\pi_0R,\Gm),\end{equation}
    where all cohomology is \'etale unless otherwise specified;
see~\cite{antieau-gepner,toen-derived}. For example, for a prime $p$, we have $\Br(\SS[1/p])\iso\ZZ/2$, so there is a
``twisted form'' of finite spectra after inverting $p$. These twisted forms are
$\infty$-categories of modules for spherical quaternion algebras.
In case that $R$ is a classical ring, $\Br(R)$ might actually be larger than the classical Brauer
group of $R$ because of the presence of derived, non-classical Azumaya algebras.

The role of connectivity is to ensure (by an argument of To\"en) that Brauer classes on connective $\EE_\infty$-rings are
\'etale-locally trivial. This fact enables the cohomological calculation of
the Brauer group as in~\eqref{eq:connective}. We will show in \cref{ex:quasi-affine} that this
\emph{fails} in general for nonconnective ring spectra. Thus we will differentiate between $\Br(R)$
and its subgroup $\LBr(R)$ of Brauer classes that are \emph{\'etale-locally trivial}, i.e.\ become
trivial after some faithful \'etale extension in the sense of \cite[Definition 7.5.0.4]{ha}. Two of
our main themes are that $\LBr(R)$ is quite computable (up to the general difficulty of computing
differentials), and that sometimes we may enlarge $\LBr(R)$ by allowing more general extensions to
kill Brauer classes. We can say something about the resulting subgroups of $\Br(R)$, which may or
may not coincide with $\LBr(R)$.

Our main examples are real K-theory and topological modular forms. Let us begin with the former. 

\begin{theorem}[\cref{ex:ko}]
    There is an isomorphism $\LBr(\KO)\iso\ZZ/2$.
\end{theorem}

The nontriviality of $\Br(\KO)$ goes back to~\cite{gepner-lawson}, where Gepner and Lawson compute the subgroup
$\Br(\KU|\KO)\subseteq\Br(\KO)$ of classes split by the faithful $\ZZ/2$-Galois extension
$\KO\rightarrow\KU$ to be $\ZZ/2$. It is not
hard to check that $\LBr(\KU)=0$ and thus we find in fact that
$\LBr(\KO)=\Br(\KU|\KO)$ as subgroups of $\Br(\KO)$ although one a priori might expect $\Br(\KU|\KO)$ to be bigger. In particular, we show that the non-trivial
class $\alpha\in\Br(\KU|\KO)$ is split by the faithful \'etale extension
$\KO\rightarrow\KO[\tfrac12,\zeta_4]\times\KO[\tfrac13,\zeta_3]$. 

Regarding the spectrum $\TMF$ of topological modular forms, we recall that Goerss, Hopkins and
Miller have defined a sheaf of $\EE_{\infty}$-ring spectra $\Oscr$ on the moduli stack $\Mscr$ of
elliptic curves \cite{TMF}. The pair $(\Mscr, \Oscr)$ defines a nonconnective spectral
Deligne--Mumford stack in the sense of \cite{sag} and
$\TMF$ is the spectrum of global sections of $\Oscr$. We may define $\Br(\Mscr, \Oscr)$ as the
Brauer group of $\QCoh(\Mscr, \Oscr)$,\footnote{Our actual definition of the Brauer group of a
nonconnective spectral DM stack in \cref{def:BrXO} is slightly different, but coincides in this
case.} which coincides with $\Mod_{\TMF}$ as $(\Mscr, \Oscr)$ is $0$-affine by \cite{mathew-meier}.
On the other hand, we may define $\LBr(\Mscr, \Oscr)$ as the subgroup of Brauer classes that become
trivial after pulling back to an \'etale cover of $\Mscr$, and this group is potentially bigger than
$\LBr(\TMF)$. As by \cite[Theorem 10.4]{mathew-galois}, all faithful Galois extensions of
localizations of $\TMF$ arise from \'etale covers of $\Mscr$, this local Brauer group $\LBr(\Mscr,
\Oscr)$ is a natural analogue of $\Br(\KU|\KO)$ above.

\begin{theorem}[\cref{thm:LBrMO1}, \cref{thm:LBrTMF},\cref{thm:LBrMO}]
After inverting $2$, the inclusion $\LBr(\TMF) \subset \LBr(\Mscr, \Oscr)$ becomes an equality and both groups are isomorphic to $\ZZ/3$. 

After localizing at $2$, the inclusion $\LBr(\TMF) \subset \LBr(\Mscr, \Oscr)$ has finite cokernel
    and both groups admit surjections to $(\ZZ/2)^{\infty}$ with kernel of order at most $8$. In
    particular, $\Br(\TMF)$ is an infinitely generated torsion abelian group. 
\end{theorem}

For the (partial) determination of $\LBr(\KO)$ and $\LBr(\TMF)$ our most important tool is an exact sequence for $\LBr(R)$ (with mild assumptions on $\pi_0R$) of the form 
\[\Br(\pi_0R) \to \LBr(R) \to \H^1(\Spec \pi_0R, \pi_0\ShPic_R),\]
which will be proven in a more precise form in \cref{prop:omni}. Here, $\pi_0\ShPic_R$ is the \emph{Picard sheaf} of $R$. It arises as the \'etale sheafification of the presheaf sending each \'etale extension $A$ of $\pi_0R$ to $\Pic(R_A)$, where $R \to R_A$ is the unique \'etale extension realizing $\pi_0R \to A$. We determine the Picard sheaf of $\KO$ in \cref{prop:PicKO} and give a partial determination of the Picard sheaf of $\TMF$ in \cref{thm:PicTMF}. The main method is a sheafy version of the Picard spectral sequence of \cite{mathew-stojanoska}. The remaining uncertainties lie in our inability to compute long differentials in the sheafy Picard spectral sequence or, essentially equivalently, in our inability to compute $\Pic(\TMF_{(2)}[\zeta_{2^n-1}])$ for $n\geq 2$. For possible subleties arising in such computations, we refer to \cref{rem:17}.

For the (partial) determination of $\LBr(\Mscr, \Oscr)$ a crucial point is to compare the (local)
Brauer groups of a nonconnective spectral Deligne--Mumford stack $(\Xscr,\Oscr)$ with the following
variant: The \emph{cohomological} Brauer group $\Br'(\Xscr,\Oscr)$ is defined
using descent from the affine case and we also obtain a subgroup $\LBr'(\Xscr, \Oscr)$. The Brauer
group of $(\Xscr,\Oscr)$ is the subgroup $\Br(\Xscr,\Oscr)\subseteq\Br'(\Xscr,\Oscr)$ of Brauer
classes representable by Azumaya algebras.
If $(\Xscr,\Oscr)\we\Spet R$ is affine, then $\Br(R)\iso\Br(\Spet R)\iso\Br'(\Spet R)$ and likewise
for $\LBr$, but $\Br$ and $\Br'$ might be different in general. In many cases of interest we show
however (extending work of To\"en~\cite{toen-derived} and
Hall--Rydh~\cite{hall-rydh}) that the cohomological Brauer agrees with the usual Brauer group. This applies in particular to $(\Mscr, \Oscr)$.

\begin{theorem}[$\Br=\Br'$, \cref{thm:spectraldm}]
    If $(\Xscr,\Oscr_\Xscr)$ is a nonconnective spectral DM stack satisfying some
    mild conditions stated in the body of the paper, then
    $\Br(\Xscr,\Oscr_\Xscr)\we\Br'(\Xscr,\Oscr_\Xscr)$ and
    $\LBr(\Xscr,\Oscr_\Xscr)\we\LBr'(\Xscr,\Oscr_\Xscr)$.
\end{theorem}

Since by definition, $\Br'$ and $\LBr'$ are approachable via descent, this result allows us to calculate $\LBr(\Mscr, \Oscr)$ via the Picard spectral sequence of \cite{mathew-stojanoska}, where it will be visible in the $(-1)$-column. Up to two differentials, we can analyze this column using \cite{mathew-stojanoska} and the vanishing of the classical Brauer group of $\Mscr$ from \cite{antieau-meier}.

\begin{question}\label{conj:a}
    Is the inclusion $\LBr(\Mscr,\Oscr)\subset \Br(\TMF)$ an equality?
\end{question} 

A similar question can be asked for every $0$-affine nonconnective spectral DM stack $(\Xscr,
\Oscr)$ where $\Oscr$ is even-periodic and the underlying stack of $\Xscr$ is regular noetherian.
Here we would replace $\LBr$ by a Brauer--Wall type extension as in \cref{rem:brw} (which
makes no difference in the case of $(\Mscr, \Oscr)$). Inspired by the case of $\KO$, we also want to
pose the question:

\begin{question}
For $(\Xscr, \Oscr)$ as above, is $\LBr(\Oscr(\Xscr))\subset \LBr(\Xscr,  \Oscr)$ always an equality?
\end{question}

\paragraph{Acknowledgments.} We would like to thank Elden Elmanto, David Gepner, Tyler Lawson, Akhil Mathew and Bertrand To{\"e}n for helpful conversations about the subject matter of this
paper over the years. 

This material is based upon work supported by the National Science Foundation
under Grant No. DMS-1440140 while the first and third authors were in residence at the
Mathematical Sciences Research Institute in Berkeley, California, during the
Spring 2019 semester. The authors would also like to thank the Isaac Newton Institute for
Mathematical Sciences for support and hospitality during the programme ``Homotopy harnessing higher
structures'' when work on this paper was undertaken. This work was supported by EPSRC grant number
EP/R014604/1. We moreover thank the Hausdorff Research Institute for Mathematics for its
hospitality.

The first author was supported by NSF Grants
DMS-2120005 and DMS-2102010 and by a Simons Fellowship. The second author was supported by the NWO grant VI.Vidi.193.111. The third author was supported by NSF Grant DMS-1812122.

\paragraph{Conventions.}
We will always use the \'etale topology. Thus, $\H^s$ means \emph{\'etale} cohomology if applied to
a scheme or Deligne--Mumford stack, and $\R^sf_*$ will like-wise refer to the $s$-th higher direct
image with respect to the \'etale topology. The notation $\Gamma$ will always refer to the global
sections of some \'etale sheaf with values in an appropriate $\infty$-category; in particular, when
applied to a sheaf of abelian groups $\Fscr$ on a site, we will view $\Fscr$ as a sheaf of spaces so
that $\pi_{-i}\Gamma(\Fscr)$ is the $i$-th cohomology group of $\Fscr$. Moreover, if $\Fscr$ is a
sheaf of spaces or spectra, $\pi_i\Fscr$ will always refer to the \'etale sheafification of the
presheaf of homotopy groups.

Generally, we will work in an $\infty$-categorical context. In particular, a \emph{commutative ring
spectrum} will mean for us a commutative algebra in the $\infty$-category of spectra, i.e.\ what is
also called an $\EE_{\infty}$-ring (spectrum). If $R$ is a commutative ring spectrum, we will always
equip $\Spec R$ with the \'etale topology. In an $\infty$-category $\Cscr$, we will write
$\Map_\Cscr(x,y)$ for the mapping space between $x,y\in\Cscr$; if $\Cscr$ is a stable
$\infty$-category or an $\infty$-category of quasi-coherent sheaves, then we will write
$\ShMap_\Cscr(x,y)$ or simply $\ShMap(x,y)$ for the mapping spectrum or the internal mapping spectrum.

The following infinity categories will be used in some of the theoretical results:
\begin{itemize}
    \item $\PrL$, the infinity category of presentable $\infty$-categories and left adjoint morphisms;
    \item $\widehat{\Cat}_\infty$, the infinity category of possibly large $\infty$-categories.
\end{itemize}
See~\cite{htt} for details.

\section{The local Brauer group in the affine case}\label{sec:affine}

After reminding the reader about the classical Brauer group of a commutative ring, we recall in this
section the definition of the Brauer group and Brauer space of a commutative ring spectrum and
introduce the notion of the local Brauer group. We will prove several basic properties (in
particular that Brauer spaces define an \'etale hypersheaf) and provide basic tools for the
computation of local Brauer groups.

\subsection{The classical Brauer group}
In this subsection, we will give a short introduction to the classical Brauer group. For more
background we refer for example to \cite{gille-szamuely}, \cite{brauer-grothendieck-group} and the
series of articles starting with \cite{grothendieck-brauer-1}.

Let $R$ be a commutative ring. An $R$-algebra $A$ is called \textbf{Azumaya} if one of the following equivalent conditions holds:
\begin{enumerate}
    \item $A$ is finitely generated, faithful, and projective as an $R$-module and the map 
    \[A \tensor_R A^{\op} \to \End_R(A), \qquad a \tensor b \,\mapsto\, (x \mapsto axb) \]
    is an isomorphism. 
    \item \'etale locally, $A$ is isomorphic to the matrix algebra $\mathrm{Mat}_{n}(R)$. 
\end{enumerate}
Two Azumaya algebras $A$ and $B$ are called \textbf{Morita equivalent} if their module categories are equivalent. 

\begin{definition}
    The \textbf{classical Brauer group} $\Br^{\mathrm{cl}}(R)$ of $R$ is the set of Azumaya algebras over $R$ up to Morita equivalence. 
\end{definition}

\begin{remark}
    Instead of working with Morita equivalence classes of Azumaya algebras, one can also directly
    define the Brauer groups via the module categories. This is the approach we will take in
    \cref{def:brauer}.
\end{remark}

In the case that $R$ is regular noetherian, $\Br^{\mathrm{cl}}(R)$ coincides with what we later
introduce as $\Br(R)$; thus we will drop the superscript in this case. Moreover, a result of Gabber
identifies $\Br(R)$ in the regular noetherian case with $\H^2(\Spec R; \GG_m)$ \cite[Corollary
3.1.4.2]{lieblich-twisted}. As $\Pic(R) \cong \H^1(\Spec R; \GG_m)$, this gives one perspective on
why Brauer groups are a higher variant of Picard groups.

If $R = k$ is a field, every finite-dimensional division $k$-algebra with center $k$ is Azumaya.
Conversely, every Azumaya $k$-algebra is Morita equivalent to a unique such. Thus, $\Br(k)$ is in
bijection with isomorphism classes of finite-dimensional division $k$-algebras with center $k$. For example,
$\Br(\mathbb{R}) \cong \{[\mathbb{R}], [\mathbb{H}]\} \cong \ZZ/2$ and $\Br(\CC) = 0$. In contrast,
the Brauer group of a non-archimedean local field $K$ (like $\QQ_p$) is isomorphic to $\QQ/\ZZ$.

It will be important for our later calculations to understand the Brauer groups of rings like $\ZZ$
or $\ZZ[\frac12, \zeta_4]$. More generally, we consider a number field $K$ and let $R$ be a
localization of the ring of
integers of $K$. In this
case, by \cite[Proposition 2.1]{grothendieck-brauer-3}, there is an exact sequence
\[0\rightarrow\Br(R)\rightarrow\Br(K)\rightarrow\bigoplus_{\mathfrak{p}\in
\Spec R^{(1)}}\Br(\Spec K_{\mathfrak{p}}),\]
where $\Spec R^{(1)}$ denotes the set of
closed points of $\Spec R$ and $K_{\mathfrak{p}}$ denotes the completion. This exact sequence is compatible with
the exact sequence
\begin{equation}\label{eq:cft}
0\rightarrow\Br(K)\rightarrow\bigoplus_{\mathfrak{p}}\Br(\Spec K_{\mathfrak{p}})\rightarrow\QQ/\ZZ\rightarrow
0\end{equation}
of class field theory (see \cite[Theorem 8.1.17]{neukirch-etal}). The sum ranges over the finite and the
infinite places of $K$, and the map $\Br(\Spec K_\mathfrak{p})\rightarrow\QQ/\ZZ$ is
the isomorphism described above when $\mathfrak{p}$ is a finite place, the
natural inclusion $\ZZ/2\rightarrow\QQ/\ZZ$ when $K_{\mathfrak{p}}\iso\RR$, and
the natural map $0\rightarrow\QQ/\ZZ$ when $K_{\mathfrak{p}}\iso\CC$.

\begin{example}\label{ex:BrauerComputations}
    One can deduce the following vanishing results from the exact sequences above:
    \begin{enumerate}
        \item[(1)] $\Br(\ZZ)=0$;
        \item[(2)] $\Br(\ZZ[\frac16]) \cong \QQ/\ZZ \oplus \ZZ/2$;
        \item[(3)] $\Br(\ZZ[\tfrac1p, \zeta_{p^n}])=0$ for a prime $p$ and a natural number $n$
    \end{enumerate}
    We will only give an argument for the last vanishing and only if $p^n\geq 3$. Let
    $R=\ZZ[\tfrac1p,\zeta_{p^n}]$. The field
    $\QQ(\zeta_{p^n})$ is totally imaginary and there is a unique prime ideal $\mathfrak{p}\subset
    \ZZ[\zeta_{p^n}]$ lying over $(p)\subset \ZZ$ (since $p$ is totally ramified). Thus for every
    place $\mathfrak{q}$ of $\QQ(\zeta_{p^n})$ we have either $\Br(K_{\mathfrak{q}}) = 0$,
    $\mathfrak{q} \in \Spec R^{(1)}$, or $\mathfrak{q}= \mathfrak{p}$. Thus, $\Br(\ZZ[\frac1p,
    \zeta_{p^n}])$ can be identified with the kernel of the map
    $\Br(\QQ(\zeta_{p^n})_{\mathfrak{p}})
    \to \QQ/\ZZ$, which is zero.
\end{example}

Brauer groups have several nice properties, three of which we will summarize in the next theorem. 

\begin{theorem}\label{thm:BrauerProperties}
Let $R$ be a regular noetherian ring. 
\begin{enumerate}
    \item[{\em (1)}] If $\Spec R[\frac1f]\subset \Spec R$ is dense, then $\Br(R) \to \Br(R[\frac1f])$ is injective. If $\frac1p \in R$ and $f$ is a non-zero divisor, we have more precisely a short exact sequence 
        \[0\to \Br(R)_{(p)} \to \Br(R[\frac1f])_{(p)} \to \H^1(R/f; \QQ/\ZZ)_{(p)} \to 0.\]
    If there is a ring homomorphism right inverse of $R \to R/f$, the sequence is split, sending $[\chi]$ to the Brauer class of the cyclic algebra $(\chi, f)$.
    \item[{\em (2)}] If $\Spec R[\frac1p]\subset \Spec R$ is dense, then $\Br(R)_{(p)} \to \Br(R[x])_{(p)}$ is an isomorphism.
\end{enumerate}
\end{theorem}
\begin{proof}
    The first point of the first point follows from \cite[Proposition 3.1.3.3]{lieblich-twisted}.
    The rest is contained in \cite[Propositions 2.14 and 2.16]{antieau-meier}.
    
    For the second point, a proof in the case $\frac1p\in R$ can be found e.g.\ in \cite[Proposition 2.5]{antieau-meier}. To show that $\Br(R)_{(p)} \cong \Br(R[x])_{(p)}$ in general, consider the diagram
 \[
  \xymatrix{
  \Br(R[x])_{(p)} \ar[r]\ar[d] & \Br(R[\tfrac1p][x])_{(p)} \ar[d]^{\cong} \\
  \Br(R)_{(p)} \ar[r] & \Br(R[\tfrac1p])_{(p)}
  }
 \]
 induced by the morphism $R[x] \to X$, sending $x$ to $0$.
 The right vertical morphism is an isomorphism since $\frac1p\in R[\frac1p]$. The horizontal arrows are injections by the first point. Thus, $\Br(R[x])_{(p)} \to \Br(R)_{(p)}$ must be an injection as well. On the other hand, it is a split surjection, using the map $R\to R[x]$. This implies that it is an isomorphism. 
\end{proof}

\begin{corollary}
    \label{cor:Brauerjpm}
    Let $S\subset \Q$ be a subring. Then the morphism 
    \[\Phi\colon \Br(S) \oplus \H^1(S; \QQ/\ZZ) \to \Br(S[j^{\pm 1}])\]
    sending $[\chi]\in \H^1(S; \QQ/\ZZ)$ to the cyclic algebra $[(\chi, j)]$ is an isomorphism. In particular, $\Br(\Z[j^{\pm 1}]) = 0$.
\end{corollary}
\begin{proof}
    Let $p$ be a prime and assume first that $S\subset \QQ$ with $\frac1p\in S$. By \cref{thm:BrauerProperties},
    we obtain a split short exact sequence
    \[0 \to \Br(S[j])_{(p)} \to \Br(S[j^{\pm1}])_{(p)} \to \H^1(S; \QQ/\ZZ)_{(p)} \to 0.\]
     By $\AA^1$-invariance, $\Br(S[j]) \cong \Br(S)$. This proves the claim $p$-locally if $\frac1p
     \in S$. Thus we obtain it for $S = \QQ$ without localization. As for a general $S\subset \QQ$
     the maps $\Br(S[j^{\pm 1}]) \to \Br(\QQ[j^{\pm 1}])$ and
     $\H^1(S;\QQ/\ZZ)\rightarrow\H^1(\QQ,\QQ/\ZZ)$ are injections, $\Phi$ is an injection in general.
     
     Next we will show the statement in the case $S = \Z_{(p)}$. Let $(a, \chi) \in \Br(\QQ) \oplus \H^1(\Q; \QQ/\ZZ)$. If $\Phi(a, \chi)$ lies in $\Br(\Z_{(p)}[j^{\pm1}])$, the image $\Phi_u(a, \chi)$ of the class $j_u^*\Phi(a, \chi) \in \Br(\Z_{(p)})$ (for an arbitrary $u\in \Z_{(p)}^\times \subset \Z_p^\times$ inducing $j_u\colon \Spec \Z_{(p)} \to \Spec \Z_{(p)}[j^{\pm1 }]$) must have image zero in $\Br(\QQ_p)$ since $\Br(\ZZ_p) =0$. 
     \[\xymatrix{
     \Br(\ZZ_{(p)}[j^{\pm 1}]) \ar[r]^{j_u^*} \ar[d]& \Br(\ZZ_{(p)}) \ar[r]\ar[d] & \Br(\ZZ_p) = 0\ar[d] \\
     \Br(\QQ[j^{\pm1}]) \ar[r] &  \Br(\QQ) \ar[r] & \Br(\QQ_p)
     }
     \]
     Assume now that $(a, \chi) \notin \Br(\ZZ_{(p)}) \oplus \H^1(\ZZ_{(p)}; \QQ/\ZZ)$. If $a \notin \Br(\ZZ_{(p)})$, then its image $b\in \Br(\QQ_p)$ is non-trivial and thus $\Phi_1(a, \chi) = b \neq 0$. Now suppose $a\in \Br(\ZZ_{(p)})$, but $\chi \notin \H^1(\ZZ_{(p)}; \QQ/\ZZ)$. Then the corresponding extension $K$ of $\QQ_p$ must be ramified. Hence the image $N(K) = N_{K|\QQ_p}(K^\times) \subset \QQ_p^\times \cong \ZZ_p^\times \times \ZZ$ cannot contain $\ZZ_p^\times$. As $N(K)$ is of finite index, the surjections $\ZZ_{(p)}^\times \to (\Z/p^n)^\times$ imply that there exists a $u\in \ZZ_{(p)}^\times$ such that $u \notin N(K)$. By \cite[Corollary 4.7.4]{gille-szamuely}, the class $\Phi_u(a, \chi) = [(\chi, u)]$ is thus nonzero in $\Br(\QQ_p)$. This shows that $\Phi$ is indeed surjective if $S =\ZZ_{(p)}$.
     
     For $S\subsetneq \QQ$ general, we have 
     \[\Br(S) \oplus \H^1(S; \QQ/\ZZ) = \left(\Br(S[\tfrac1p]) \oplus \H^1(S[\tfrac1p];
     \QQ/\ZZ)\right) \cap \left(\Br(\Z_{(p)}) \oplus \H^1(\ZZ_{(p)}; \QQ/\ZZ)\right)\] as subgroups of  $\Br(\QQ) \oplus \H^1(\QQ; \QQ/\ZZ)$.
     As we have for every $p$ with $\frac1p\neq S$ a containment
     \[\Br(S[j^{\pm 1}])_{(p)} \subset \Br(S[\tfrac1p, j^{\pm 1}])_{(p)}\cap \Br(\ZZ_{(p)}[j^{\pm 1}])_{(p)}\]
     we see that $\Phi_{(p)}$ is actually surjective for every $S\subset \QQ$ and every prime $p$, which proves the claim. 
\end{proof}

In some of our examples below, we will also use the following classical results, which will help us compute
Brauer groups of various ring spectra. The first is Grothendieck's rigidity result for the Brauer
group \cite[Corollaire 6.2]{grothendieck-brauer-1}
\begin{theorem}\label{thm:rigidity}
Suppose $R$ is
Hensel local with residue field $k$; then $\Br(R)\iso\Br(k)$. If $R$ is also regular, then
$\Br(R)\iso\H^2(\Spec R,\Gm)$ so that $\H^2(\Spec
R,\Gm)\iso\H^2(\Spec k,\Gm)$.
\end{theorem}

The next is a corollary of the affine analogue of proper base change as proved in
Gabber--Huber~\cite{gabber-affine,huber-affine}, see also~\cite[Corollary
1.18(a)]{bhatt-mathew-arc}.
\begin{theorem}\label{thm:Gabber-Huber}
If $R$ is a Hensel local ring with residue field
$k$, then $\H^*(\Spec R,\Ascr)\iso\H^*(\Spec k,i^*\Ascr)$ for all torsion \'etale sheaves
$\Ascr$ on $\Spec R$, where $i\colon\Spec k\hookrightarrow\Spec R$.
\end{theorem}

The next result can be found in \cite[Corollaire II.3.5, Proposition II.3.6]{sga4half} or
\cite[Corollary II.3.6]{milne-etale} and we will use it several times in the setting of closed
immersions.

\begin{theorem}\label{thm:closedimmersion}
Let $f\colon X \to Y$ be a finite morphism of schemes. If $\Fscr$ is an \'etale sheaf on $X$, then
    $\R^qf_* = 0$ for $q>0$ and hence $\H^i(X; \Fscr) \cong \H^i(Y; f_*\Fscr)$ for all $i\geq 0$.
\end{theorem}

\subsection{Brauer groups of ring spectra}
In this subsection, we will recall the Brauer group and Brauer space of a commutative ring
spectrum, which were first introduced by \cite{baker-richter-szymik} and \cite{szymik}. For our
purposes, an approach will be convenient that sees Brauer groups as categorified Picard groups. Let
us thus first recall the definition of the Picard group and Picard space.

If $\Cscr$ is a symmetric monoidal $\infty$-category, its underlying
$\infty$-groupoid $\iota\Cscr$ naturally admits the structure of an
$\EE_\infty$-space and the counit map $\iota\Cscr\rightarrow\Cscr$ is symmetric
monoidal. We define the {\bf Picard space} $\ShPic(\Cscr)$ to be the maximal
grouplike $\EE_\infty$-groupoid in $\iota\Cscr$. In other words,
$\ShPic(\Cscr)$ is the space of $\otimes$-invertible objects of $\Cscr$ and
equivalences. The {\bf Picard group} of $\Cscr$ is $\Pic(\Cscr)=\pi_0\ShPic(\Cscr)$.
We refer to \cite{mathew-stojanoska} for more background on Picard groups and spaces. 

\begin{example}
    If $R$ is a commutative ring spectrum, its {\bf Picard space} $\ShPic(R)$ is
     $\ShPic(\Mod_R)$ and its {\bf Picard group} is $\Pic(R)=\pi_0\ShPic(R)$.
\end{example}

Next we introduce the Brauer group $\Br(R)$ of a commutative ring spectrum $R$ as a categorification
of the Picard group. In the case that $R$ is a regular noetherian ring, this will agree with the
classical Brauer group (see \cref{rem:classical}).

\begin{definition}\label{def:brauer}
    Let $R$ be a commutative ring spectrum and let $\Cat_R$ denote the
    presentably symmetric monoidal $\infty$-category of compactly generated $R$-linear stable
    $\infty$-categories and compact object-preserving left adjoint functors.\footnote{Background on
    this can for example be found in \cite[Section 3.1]{antieau-gepner}, where the notation
    $\Cat_{R, \omega}$ is used.}
    \begin{enumerate}
        \item[(a)]
            We let $\ShBr(R)=\ShPic(\Cat_R)$ denote the {\bf Brauer space of
            $R$}. The {\bf
            Brauer group} of a commutative ring spectrum $R$ is
            $\Br(R)=\pi_0\ShBr(R)$. 
        \item[(b)] If $A$ is an $R$-algebra, we say that $A$ is an {\bf Azumaya algebra
            over $R$} if $\Mod_A$ defines a point of $\ShBr(R)$.
        \item[(c)] An Azumaya $R$-algebra $A$ is {\bf trivial} if
            $\Mod_A\we\Mod_R$, i.e., if $A$ is $R$-linearly (derived) Morita equivalent to
            $R$.
        \item[(d)] An Azumaya $R$ algebra $A$ is {\bf \'etale-locally trivial}
            if there is an 
            \'etale cover $R\rightarrow S$ such that $S\otimes_RA$ is trivial. 
    \end{enumerate}
\end{definition}

This definition of an Azumaya algebra is due to To\"en~\cite{toen-derived}. It agrees with the
original definition of an Azumaya algebra in this setting due
to~\cite{baker-richter-szymik}; see~\cite{antieau-gepner} for more details.

\begin{lemma}\label{lem:loops}
    If $R$ is a commutative ring spectrum, then there is a natural equivalence $\ShPic(R)\we\Omega\ShBr(R)$, where $\Omega\ShBr(R)$ is
    computed via loops based at the trivial Brauer class.
\end{lemma}

\begin{proof}
    By construction, $\Omega\ShBr(R)$ is the space of autoequivalences of the
    unit object of $\Cat_R$. The unit object is $\Mod_R$ and the
    autoequivalences must be $R$-linear, so they correspond to tensoring with
    invertible $R$-modules.
\end{proof}

We will prove in the next section that $R\mapsto\ShBr(R)$ is a Postnikov complete \'etale sheaf. To do
so, we first establish that $R\mapsto\Cat_R$ is an \'etale sheaf (with values in
$\PrL$). The result was discovered in the context of the present project,
but appeared first in~\cite[Thm.~2.16]{antieau-elmanto}.

\begin{proposition}\label{prop:sheaf}
    The presheaf $R\mapsto\Cat_R$ is an \'etale sheaf with values in $\PrL$.
\end{proposition}

\begin{proof}
    The fact that each $\Cat_R$ is presentable follows from~\cite[Corollary~4.25]{bgt1}. That for any map $R \to S$ of commutative ring spectra, the induced functor
    $\Cat_R\rightarrow\Cat_S$ is a left adjoint follows because
    $\Cat_S\we\Mod_{\Mod_S}(\Cat_R)$. Thus, since the forgetful functor
    $\PrL\rightarrow\widehat{\Cat}_\infty$ preserves limits
    by~\cite[Prop.~5.5.3.13]{htt}, it suffices to see that
    $R\mapsto\Cat_R$ is an \'etale sheaf with values in
    $\widehat{\Cat}_\infty$. This is part of \cite[Thm.~2.16]{antieau-elmanto}.
\end{proof}

\begin{corollary}
    The construction $R\mapsto\ShBr(R)$ is an \'etale sheaf on $\CAlg^{\op}$.
\end{corollary}

\begin{proof}
    The construction $\Cscr\mapsto\ShPic(\Cscr)$ preserves limits as a functor from symmetric monoidal $\infty$-categories to spaces \cite[Proposition 2.2.3]{mathew-stojanoska}. As
    limits of symmetric monoidal $\infty$-categories are computed on the level of underlying
    $\infty$-categories, the result follows from \cref{prop:sheaf}.
\end{proof}

\subsection{The local Brauer group}
While in classical algebras, Azumaya algebras are always \'etale-locally Morita equivalent to the
ground ring, this is no longer true in the spectral setting. In this subsection (and actually the
whole article), we will concentrate on those which \emph{are} \'etale-locally trivial.

\begin{definition}
    Let $\pi_0\ShBr$ denote the \'etale sheaf of connected components of
    $\ShBr$. We let $\ShLBr$ be the fiber of the natural map
    $\ShBr\rightarrow\pi_0\ShBr$ in \'etale sheaves. The space $\ShLBr(R)$ is the {\bf local
    Brauer space} of $R$ and $\LBr(R)=\pi_0(\ShLBr(R))$ is the {\bf local Brauer
    group} of $R$.
\end{definition}

\begin{remark}
    Thanks to \cref{lem:loops},
    we could equivalently have defined $\ShLBr$ as $\ShBPic$, the \'etale
    classifying space of $\ShPic$, computed in \'etale sheaves. However, note that the
    functor $R \mapsto\B\ShPic (R)$, sending $R$ to the classifying space of its Picard space, is
    not a sheaf, and $\ShBPic$ is its sheafification.
\end{remark}

The name `local Brauer group' is short-hand for `locally-trivial Brauer group',
which is justified by the following lemma.

\begin{lemma} Let $R$ be a commutative ring spectrum.
    \begin{enumerate}
        \item[{\rm (a)}]   The natural map $\LBr(R)\rightarrow\Br(R)$ is an
            injection and hence $\ShLBr(R)\rightarrow\ShBr(R)$ is the inclusion
            of a subspace of connected components.
        \item[{\rm (b)}]   An element $\alpha\in\Br(R)$ is contained in $\LBr(R)$ if and only if there
            is a faithful \'etale map $R\rightarrow S$ such that $\alpha$ maps to zero in
            $\Br(S)$.
    \end{enumerate}
\end{lemma}

\begin{proof}
    For (a), use the fiber sequence 
    $$\ShLBr(R)\rightarrow\ShBr(R)\rightarrow\Gamma(\Spec
    R,\pi_0\ShBr)$$
    of spaces. Recall here that $\Gamma$ denotes the \emph{space} (as opposed to the set) of global
    sections of an \'etale sheaf and $\pi_0\ShBr$ denotes the \'etale-sheafified homotopy group.
    Since $\pi_i\Gamma(\Spec R,\pi_0\ShBr)=0$ for
    $i>0$, the first claim follows from the long exact sequence in homotopy. Thus we can identify $\LBr(R)$ as a subgroup of $\Br(R)$. 
    
    If
    $\alpha\in\LBr(R)$, then $\alpha$ is \'etale-locally trivial since
    $\pi_0\ShLBr=0$. Conversely, if
    $\alpha\in\Br(R)$ is such that there exists a faithful \'etale map $R\rightarrow S$ such
    that $\alpha_S=0\in\Br(S)$, then the image of $\alpha$ in $\pi_0\R\Gamma(\Spec
    R,\pi_0\ShBr)$ is zero. Thus, $\alpha\in\LBr(R)$. This proves (b).
\end{proof}

If $R$ is a commutative ring spectrum, we will always equip $\Spec R$ with the \'etale topology. The small \'etale sites of $\Spec R$ and $\Spec \pi_0R$ agree, so we
can compute cohomology of sheaves of abelian groups on $\Spet R$ via 
\'etale cohomology on $\Spec\pi_0R$: given a sheaf of abelian groups $\Ascr\in\Shv_{\Sp}(\Spet R)^\heartsuit$, we
have 
$$\Gamma(\Spet
R,\Ascr)\we\Gamma(\Spet\pi_0R,\Ascr)$$ 
and thus
\[\pi_{-i}\Gamma(\Spet
R,\Ascr)\cong \pi_{-i}\Gamma(\Spet\pi_0R,\Ascr) \cong \H^i(\Spec \pi_0R; \Ascr)\]
This will be used constantly below.

Specifically, to compute $\ShLBr(R)$, we can restrict $\ShLBr$ and $\ShBr$ to \'etale
sheaves $\ShLBr_{\Oscr}$ and $\ShBr_\Oscr$ on the small \'etale site of $R$. Note that these are
already sheaves and no additional sheafification is necessary. While we may and will view the
$\ShLBr_\Oscr$ and $\ShBr_\Oscr$  as sheaves on $\Spec \pi_0 R$, they certainly depend crucially on
$R$, not only on $\pi_0R$. We use the notation
$\ShPic_\Oscr$ also for the restriction of $\ShPic$ to $\Spet R$. 

\begin{lemma}\label{lem:hosheaves}
    Let $\ShLBr_\Oscr$ be the local Brauer space sheaf constructed above on
    $\Spet R$ for a commutative ring spectrum $R$. The homotopy sheaves of
    $\ShLBr_\Oscr$ are given by
    $$\pi_t\ShLBr_\Oscr\iso\begin{cases}
        0&\text{if $t=0$,}\\
        \pi_0\ShPic_\Oscr&\text{if $t=1$,}\\
        \Gm&\text{if $t=2$, and}\\
        \pi_{t-2}\Oscr&\text{if $t\geq 3$,}
    \end{cases}$$ where $\Gm$ is the \'etale sheaf $\Gm(S)\iso(\pi_0S)^\times$
    for an \'etale commutative $R$-algebra $S$. In particular, $\pi_t\ShLBr_\Oscr$ is quasi-coherent for $t\geq 3$.
\end{lemma}

\begin{proof}
    This follows from \Cref{lem:loops}, using that \'etale sheafification commutes with restriction
    along the morphism $\CAlg_R^{\et} \to \CAlg$ from commutative \'etale $R$-algebras to all
    commutative ring spectra.
\end{proof}

\begin{remark}
    Analysis of $\pi_1\ShLBr_{\Oscr}\iso\pi_0\ShPic_{\Oscr}$ is often the most difficult part
    of local Brauer group computations.
\end{remark}

\begin{definition}
    Let $\Cscr$ be a prestable $\infty$-category in the sense of \cite[Appendix C]{sag} having all
    limits, which is automatically the nonnegative part of a t-structure on a stable
    $\infty$-category. We say that $X\in\Cscr$
    is {\bf $\infty$-connective} if $\Map(X,Y)\we 0$ for every truncated
    object $Y$. An object $Y$ of $\Cscr$ is {\bf hypercomplete} if
    $\Map(X,Y)\we 0$ for every $\infty$-connective object $X$. Finally,
    $Y$ is {\bf Postnikov complete} if the natural map
    $Y\rightarrow\lim_n\tau_{\leq n}Y$ is an equivalence; this occurs if and only if 
    $\lim_n\tau_{\geq n+1}Y\we 0$ as fiber sequences are closed under limits.
\end{definition}

Postnikov complete objects are hypercomplete, but the converse is not always
true. The significance of Postnikov completeness is that it allows us to compute
global sections by using  descent spectral sequences. As our prestable $\Cscr$ we will use sheaves
with values in grouplike $\EE_{\infty}$-spaces (i.e., connective spectra). Note that the forgetful functor from grouplike
$\EE_{\infty}$-spaces to spaces preserves and detects limits.

\begin{proposition}\label{prop:hypersheaf}
    The assignments $R\mapsto\ShLBr(R)$ and $R\mapsto\ShBr(R)$ define Postnikov
    complete \'etale sheaves of grouplike $\EE_{\infty}$-spaces on $\CAlg^\op$. Similarly, $\ShLBr_\Oscr$ and
    $\ShBr_\Oscr$ are Postnikov complete \'etale sheaves on $\Spet R$ for any
    commutative ring spectrum $R$.
\end{proposition}

\begin{proof}
    The proofs of all four cases are the same, so we give it only for $\ShBr$
    on $\CAlg^\op$. As $\ShBr$ is Postnikov complete if and only if $\lim_n \tau_{\geq n+1}\ShBr \simeq
    \ast$, it is enough to prove Postnikov completeness for any sufficiently
    connective cover of $\ShBr$. We prove that $\tau_{\geq 3}\ShBr$ is Postnikov complete in two steps.
    First, $R\mapsto\Mod_R$ satisfies hyperdescent (see~\cite[Corollary
    D.6.3.3]{sag}), so $\Pic$ preserving limits as a functor from symmetric monoidal $\infty$-categories implies $\ShPic$ being hypercomplete. 
    
    This implies that
    $\ShLBr\we\ShBPic$ is
    hypercomplete: On $1$-connective \'etale sheaves, $\Omega$ is fully faithful. If $X$ is an $\infty$-connective \'etale sheaf, then
    $$\Map(X,\ShLBr)\we\Map(\Omega X,\ShPic)\we 0,$$ 
    since $\Omega X$ is still $\infty$-connective.
    
    Second, the fact that $\ShLBr$ is hypercomplete implies that its
    $3$-connective cover $\tau_{\geq 3}\ShLBr$ is hypercomplete. However, the homotopy
    sheaves $\pi_i\tau_{\geq 3}\ShLBr$ are all quasi-coherent by
    \cref{lem:hosheaves}. Therefore, there are enough objects (affines for example) of cohomological dimension $\leq 0$ with
    $\{\pi_*\tau_{\geq 3}\ShLBr\}$-coefficients in the sense
    of~\cite[Def.~2.8]{clausen-mathew}. By~\cite[Prop.~2.10]{clausen-mathew},
    it follows that $\tau_{\geq 3}\ShLBr$ is in fact Postnikov complete, which is what we wanted to show.
\end{proof}

\begin{remark}
    A form of \Cref{prop:hypersheaf} was proved in~\cite[Section
    7]{antieau-gepner} in the special case of connective
    commutative ring spectra using a different argument, although the proof in the
    connective case and of~\cite[Proposition~6.5]{lurie-dag11}, which is used
    in the proof of \Cref{prop:sheaf}, are closely related under the
    hood. The main point
    of~\cite{antieau-gepner} is that when $R$ is connective, every Azumaya
    $R$-algebra is \'etale locally trivial. This is not true in general, as
    \cref{ex:quasi-affine} below shows.
\end{remark}

As the sheaves $\ShPic$, $\ShLBr$ and $\ShBr$ take values in grouplike $\EE_{\infty}$-spaces, we can
deloop them to presheaves of spectra. Sheafifying these results in sheaves $\Shpic$, $\mathbf{lbr}$
and $\mathbf{br}$ and the restrictions $\Shpic_{\Oscr}$, $\mathbf{lbr}_{\Oscr}$ and
$\mathbf{br}_{\Oscr}$ when working on the \'etale site of $\Spec R$. Note that by
construction, $\mathbf{lbr}_{\Oscr} \simeq \Shpic[1]$. Note moreover that $\pi_n\mathbf{lbr} \cong
\pi_n\ShLBr$, but the global sections can acquire additional negative homotopy groups.

\begin{corollary}
    There is a convergent spectral sequence
    \begin{equation}\label{eq:descentss}\E_2^{s,t}\iso\H^s(\Spec\pi_0R,\pi_t\ShLBr_\Oscr)
    \quad\mathlarger{\Longrightarrow}\quad \pi_{t-s}\mathbf{lbr}_\Oscr(\Spet R)\underset{t-s\geq
    0}{\iso}\pi_{t-s}\ShLBr(R)
    \end{equation}
    with differentials $d_r\colon \E_r^{s,t} \to E_r^{s+r, t+r-1}$, which degenerates at the $\E_3$-page. 
\end{corollary}

\begin{proof}
    This spectral sequence is the descent spectral sequence for the sheaf $\mathbf{lbr}_{\Oscr}$ of
    spectra, associated to the tower of global sections of the truncations of
    $\mathbf{lbr}_{\Oscr}$. Convergence follows from the fact that $\pi_t\ShLBr_\Oscr$ is
    quasi-coherent for $t\geq 3$ by \cref{lem:hosheaves} so that $\E_2^{s,t}=0$ for $t\geq 3$ and $s\geq 1$ and hence the
    spectral sequence degenerates at the $\E_3$-page.
\end{proof}

The next proposition is our main tool to attack the local Brauer group of a commutative ring
spectrum. Recall to that purpose that a commutative ring spectrum $R$ is \textbf{weakly $2k$-periodic}
if $\pi_{2k}R$ is an invertible $\pi_0R$-module and $\pi_{2k}R\tensor_{\pi_0R}\pi_nR \to
\pi_{2k+n}R$ is an isomorphism for all $n\in\ZZ$.
\begin{proposition}\label{prop:omni} Let $R$ be a commutative ring spectrum.
    \begin{enumerate}
        \item[{\rm (i)}] There is a natural exact sequence
            \begin{gather*}
            0\rightarrow\H^1(\Spec\pi_0R,\Gm)\rightarrow\Pic(R)\rightarrow\H^0(\Spec\pi_0R,\pi_1\ShLBr_\Oscr)\rightarrow\\
                \H^2(\Spec\pi_0R,\Gm)\rightarrow\LBr(R)\rightarrow\H^1(\Spec\pi_0R,\pi_1\ShLBr_\Oscr)\rightarrow
            \H^3(\Spec\pi_0R,\Gm).
            \end{gather*}
        \item[{\rm (ii)}]   If $R$ is connective, then there are natural
            identifications
        $\LBr(R)=\Br(R)\iso\H^1(\Spec\pi_0R,\ZZ)\times\H^2(\Spec\pi_0R,\Gm)$.
        \item[{\rm (iii)}]   Fix $k\geq 1$. If $R$ is weakly $2k$-periodic, with $\pi_iR=0$ for $i$ not divisible by $2k$, and such that $\pi_0R$ is regular noetherian, then there is a
            natural exact sequence
            $$0\rightarrow\H^2(\Spec\pi_0R,\Gm)\rightarrow\LBr(R)\rightarrow\H^1(\Spec\pi_0R,\ZZ/2k)\rightarrow
            \H^3(\Spec\pi_0R,\Gm).$$
    \end{enumerate}
\end{proposition}

\begin{proof}
    Part (i) is the exact sequence of low-degree terms of the spectral
    sequence~\eqref{eq:descentss} using that $\pi_0\ShLBr_\Oscr=0$ and that the spectral sequence degenerates at the $\E_3$-page.

    Part (ii) is the content of~\cite[Theorem~5.11 and Corollary~7.13]{antieau-gepner}.
    Note that the exact sequence from part (i) splits in the case as $\ZZ$ is the free grouplike
    $\EE_1$-space, which lets us split the map
    $\ShLBr_\Oscr\rightarrow\ShB\pi_1\ShLBr_\Oscr\we\ShB\ZZ$ in the connective case. Here we use
    that $\Pic(R) \cong \Pic(\pi_*R) \cong \Pic(\pi_0R) \times \ZZ$ by a result of \cite[Theorem
    21]{baker-richter} and \cite[Theorem 2.4.4]{mathew-stojanoska} and that the functor $R \mapsto
    \Pic(\pi_0R) \times \ZZ$ sheafifies in the \'etale topology to the constant sheaf $\ZZ$ since
    every Picard group element is \'etale locally trivial.
    
    For Part (iii), we first claim  that $\pi_1\ShLBr_\Oscr\iso\pi_0\ShPic_\Oscr\iso\ZZ/2k$ when $R$
    satisfies the conditions of (iii). Indeed, \'etale-locally we can assume that $R$ is
    $2k$-periodic and thus we obtain $\Pic(R) \cong \Pic(\pi_0R) \times \ZZ/2k$ by
    \cite[Theorem~37]{baker-richter} when $k=1$ and~\cite[Corollary 2.4.7]{mathew-stojanoska} in
    general. As above, this sheafifies to $\ZZ/2k$. Hence
    $\Pic(R)\rightarrow\H^0(\Spec\pi_0R,\ZZ/2k)$ is
    surjective. 
\end{proof}

\begin{remark}\label{rem:classical}
    If $R$ is a classical commutative ring, then $\LBr(R)=\Br(R)$ differs from the
    classical notion of the Brauer group, because we allow derived Azumaya
    algebras. In this case, $$\Br(R)\iso\H^1(\Spec R,\ZZ)\times\H^2(\Spec
    R,\Gm)$$ by part (ii) of \Cref{prop:omni} or by~\cite[Theorem 1.1]{toen-derived}.
    The Brauer group of ordinary Azumaya algebras is given by $\H^2(\Spec
    R,\Gm)_{\tors}$, by Gabber~\cite{gabber}. As before, we write
    $\Br^{\mathrm{cl}}(R)$ for the {\bf classical Brauer group} of ordinary Azumaya algebras
    when $R$ is a commutative ring. If $R$ is regular noetherian, then
    $\Br^{\mathrm{cl}}(R) = \Br(R)$ since in this case the $\H^1(\Spec
    R,\ZZ)$ term vanishes because $R$ is normal (see~\cite[2.1]{deninger}) and since $\H^2(\Spec
    R,\Gm)$ is all torsion (see~\cite[Cor.~1.8]{grothendieck-brauer-2}).
\end{remark}

\begin{example}
    When $\pi_0R=\ZZ$ and $R$ is either connective or satisfies the conditions
    of (iii) in \cref{prop:omni}, then the lemma implies that $\LBr(R)=0$. Indeed,
    $\H^1(\Spec\ZZ,\ZZ)=0=\H^1(\Spec\ZZ,\ZZ/2k)$ and
    Grothendieck proved that $\H^2(\Spec\ZZ,\Gm)=0$ in~\cite{grothendieck-brauer-3}.
    This covers the sphere spectrum $\SS$, the complex cobordism ring $\MU$, the periodic complex
    $K$-theory spectrum, as well as connective complex or real $K$-theory and connective topological
    modular forms.
\end{example}

\begin{example}\label{ex:en}
    Let $k$ be a perfect field of positive characteristic $p$ and let $\GG$ be a $1$-dimensional
    formal group law of height $n$ on $k$. Let $E_n(\GG,k)$ be the Lubin--Tate spectrum associated
    to $\GG$ so that $\pi_*E_n(\GG,k)\iso \WW(k)\lBrack
    u_1,\ldots,u_{n-1}\rBrack[u^{\pm 1}]$, where $\WW(k)$ denotes the ring of
    $p$-typical Witt
    vectors of $k$, each $u_i$ has degree $0$, and $u$
    has degree $2$. We want to show that the local Brauer group $\LBr(E_n(\GG,k))$ is typically
    non-zero. Note that this is related to but different than the results of Hopkins and Lurie in
    \cite{hopkins-lurie}, who look at the Brauer group of \emph{$K(n)$-local} $E_n(\GG,k)$-modules,
    which is different from that of $E_n(\GG,k)$-modules. Moreover, they study the Brauer
    group and not just the local Brauer group.
    
    Since $E_n(\GG,k)$ is $2$-periodic, part (iii) of \Cref{prop:omni} applies. To compute the groups that build $\LBr(E_n(\GG,k))$, note first that
    $\H^2(\Spec \pi_0E_n(\GG,k),\Gm)\iso\H^2(\Spec k,\Gm)$ by \Cref{thm:rigidity} and $\H^2(\Spec k,\Gm)\iso\Br(k)$ is typically
    non-zero. Moreover, $\H^1(\Spec \pi_0E_n(\GG,k),\ZZ/2)\iso\H^1(\Spec
    k,\ZZ/2)$ by \Cref{thm:Gabber-Huber}. 
    
    If $k=\FF_{p^r}$ is finite, then the contribution from
    $\H^2(\Spec \FF_{p^r},\Gm) \cong \Br(\FF_{p^r})$ vanishes by Wedderburn's theorem, while
    the group $\H^1(\Spec \FF_{p^r},\ZZ/2) $ equals $\ZZ/2$, as there is a unique $\ZZ/2$-Galois
    extension of $\FF_{p^r}$. Thus, we obtain an exact sequence
    \[0 \to \LBr(E_n(\GG,\FF_{p^r})) \to \ZZ/2 \to \H^3(\Spec \pi_0E_n(\GG,\FF_{p^r}), \Gm),\]
    which implies that $\LBr(E_n(\GG,\FF_{p^r}))$ is either zero or $\ZZ/2$.

    We can be more specific if we assume that $n=1$, since in that case we have
    $\pi_0E_1(\GG,\FF_{p^r}) \cong\WW(\FF_{p^r})$ and $\H^3(\Spec \WW(\FF_{p^r}),\Gm)=0$ by
    \cite[1.7a]{Mazur-notes}. Thus $\LBr(E_1(\GG,\FF_{p^r})) $ must be $\ZZ/2$. Note that
    $E_1(\GG,\FF_{p^r})$ is closely related to periodic complex $K$-theory.
    
    For a general height $n$, we claim that $\H^3(\Spec \pi_0E_n(\GG,\FF_{p^r}), \Gm)$ is $p$-local. Indeed, for $l$ prime to $p$, we have a short exact sequence
    \[0 \to \mu_l \to \Gm \xrightarrow{l} \Gm \to 0\]
    of \'etale sheaves. By \Cref{thm:Gabber-Huber} we have $\H^i(\Spec \pi_0E_n(\GG,\FF_{p^r});
    \mu_l) \cong \H^i(\Spec \FF_{p^r}; \mu_l)$. The latter group is zero for $i\geq 2$ since the
    absolute Galois group $\widehat{\ZZ}$ of $\FF_{p^r}$ has cohomological dimension $1$. Thus,
    multiplication by $l$ is an isomorphism on $\H^3(\Spec \pi_0E_n,\Gm)$, proving the claim. As a
    consequence, the map
    \[\ZZ/2 =\H^1(\Spec \pi_0E_n(\GG,\FF_{p^r}), \ZZ/2) \to \H^3(\Spec \pi_0E_n(\GG,\FF_{p^r}), \Gm)  \]
    must be zero in the case of odd primes $p$, and again we obtain that $\LBr(E_n(\GG,\FF_{p^r})) \cong \ZZ/2$ at all heights if $p$ is odd.

    On the other hand, if we work over a separably closed (rather than finite) field
    $k^{sep}$, then $\H^2(\Spec k^{sep},\Gm)$ is again zero, but so is $\H^1(\Spec k^{sep}, \ZZ/2)$.
    This results in the fact that $\LBr(E_n(\GG,k^{sep}))$ is zero, and in particular, it implies
    that any non-trivial Brauer class in $\LBr(E_n(\GG,k)$ is necessarily split by $E_n(\GG,k) \to
    E_n(\GG,k^{sep})$.
    \end{example}

\begin{remark}
    We do not in fact know an example where the differential 
    \[\H^1(\Spec
    \pi_0R,\pi_1\ShLBr_\Oscr)\rightarrow\H^3(\Spec\pi_0R,\Gm)\]
    from \cref{prop:omni}(iii) is non-zero.
    Similarly, it would be informative to know if there is a commutative ring spectrum $R$ such that
    \[\Pic(R)\rightarrow\H^0(\Spec\pi_0R,\pi_1\ShLBr_\Oscr)\]
    is not surjective.
\end{remark}

\begin{remark}\label{rem:brw}
	We are primarily interested in integral results as we want to understand
	contributions to the Brauer group for commutative ring spectra such as the various forms of topological modular forms.
	Nevertheless, when $R$ is even and weakly $2$-periodic and if
	additionally $2$ is a unit on $R$, then there is an identifiable non-\'etale-locally trivial
	contribution to the Brauer group in general. If $R$ is actually $2$-periodic and $u\in\pi_2R$ a unit, let $A_u$ be the
	Azumaya algebra constructed in~\cite[Example 7.2]{gepner-lawson}: it is an
    $R$-algebra with $\pi_*A_u=\pi_*R[x]$ where $|x|=1$.
	We let $\ShLBrW_\Oscr$ be the sheafification of the components of
    $\ShBr_\Oscr$
	containing $0$ and the $[A_u]$ for units $u\in\pi_2S$ for \'etale extensions $S$ of $R$. Using that $[A_u] + [A_v]$ lies in $\LBr(S)$ for any units $u,v\in\pi_2S$, there is a natural fiber sequence
	$$\ShLBr_\Oscr\rightarrow\ShLBrW_\Oscr\rightarrow\ZZ/2$$ of sheaves on $\CAlg_R^{\et}$.
	More generally, if $2$ is not a unit
	on $R$ (but $R$ is still even and weakly periodic), then we can construct
	an extension $$\ShLBr_\Oscr\rightarrow\ShLBrW_\Oscr\rightarrow j_!\ZZ/2,$$ where
	$j\colon\Spec\pi_0R[\tfrac{1}{2}]\rightarrow\Spec\pi_0R$. An easy check
    using \cite[Proposition 7.6]{gepner-lawson} verifies that
    the algebraic Brauer group of $R$, as defined in~\cite{gepner-lawson}, is a subgroup of
    $\LBrW(R)=\pi_0(\ShLBrW_{\Oscr}(R))$.
\end{remark}

\section{The Picard sheaf and local Brauer group of KO}\label{sec:PicKO}

\newcommand{\Escru}{\underline{\Escr}}
\newcommand{\Hscru}{\underline{\Hscr}}

The aim of this section is to show that the local Brauer group of $\KO$ is $\ZZ/2$. By the previous section, the key is to understand the \'etale Picard sheaf $\pi_0\ShPic_{\Oscr_\KO}$ on
$\Spet\KO$. 
To achieve that, we essentially re-run the calculations of $\ShPic(\KO)$ from Gepner--Lawson and Mathew--Stojanoska, but this time
in sheaves of spaces on $\Spet\KO$. As an aside we will also compute $\Pic(\KO_R)$ for any \'etale extension $R$ of $\ZZ$, where $\KO_R$ denotes the \'etale extension of $\KO$ lifting $R$. (We will use similar notation for other ring spectra as well.) 

Recall that $\KO \to \KU$ is a $C_2$-Galois extension, and consequently $\Shpic(\KO)\we\tau_{\geq 0}\left(\Shpic(\KU)^{hC_2}\right)$ by Galois descent. 
Similarly, if $R$ is an \'etale $\ZZ$-algebra, then
$$\Shpic(\KO_R)\we\tau_{\geq 0}\left(\Shpic(\KU_R)^{hC_2}\right).$$  
Thus, there is an equivalence $$\Shpic_{\Oscr_{\KO}}\we\tau_{\geq
0}(\Shpic_{\Oscr_\KU}^{hC_2})$$ of sheaves of connective spectra on $\Spet\KO$, which results in a homotopy fixed point descent spectral sequence with signature
\begin{equation}\label{ss:picKO}
    \Escr_2^{s,t}=\Hscr^s(C_2,\pi_t\Shpic_{\Oscr_\KU})\Rightarrow\pi_{t-s}\Shpic_{\Oscr_\KU}^{hC_2}.
\end{equation}
The notation
$\Escr_2^{s,t} \iso \Hscr^s(C_2,\pi_t\Shpic_{\Oscr_\KU})$ means that the $C_2$-cohomology is taken
in \'etale sheaves, and the differentials are
\[d_r^{s,t}\colon\Escr_r^{s,t} \to \Escr_r^{s+r,t+r-1}. \]
Note that in the figures below, we will depict this spectral sequence with the Adams indexing convention, i.e. in the $(t-s,s)$-plane. 

The action of $C_2$ on the homotopy sheaves of $\Shpic_{\Oscr_\KU}$ is as follows
\[
\pi_i \Shpic_{\Oscr_\KU} = \begin{cases} 
\Z/2, \text{ with trivial action, when } i=0,\\
\Oscr^\times, \text{ with trivial action, when } i=1,  \\
\Oscr, \text{ with trivial action, when } i>1,\, i\equiv 1 \mod 4 ,\\
\Oscr, \text{ with sign action, when } i>1,\, i\equiv 3 \mod 4.
\end{cases}
\] This allows us to compute $C_2$-cohomology and hence the $\Escr_2$-page of \eqref{ss:picKO}.

\begin{example}
    The action of $C_2$ on $\pi_0\Oscr^\times$ is trivial, so the cohomology
    sheaves are $$\Hscr^s(C_2,\Oscr^\times)\iso\begin{cases}
        \Oscr^\times&\text{if $s=0$,}\\
        \mu_2&\text{if $s>0$ is odd, and}\\
        \omega_2&\text{if $s>0$ is even,}
    \end{cases}$$
    where $\mu_2$ and $\omega_2$ fit into the exact sequence
    $$0\rightarrow\mu_2\rightarrow\Oscr^\times\xrightarrow{x\mapsto
    x^2}\Oscr^\times\rightarrow\omega_2\rightarrow 0.$$
    Note that on $\Spet \ZZ$, the sheaf $\mu_2$ is isomorphic to the constant sheaf $\ZZ/2$; indeed, every \'etale extension of $\ZZ$ is a product of integral domains with $2\neq 0$. 
\end{example}

The following identification will not be necessary for our computation of $\LBr(\KO)$, but we add it for completeness. 
\begin{lemma}
    The sheaf $\omega_2$ is isomorphic to $\Oscr/2$ on $\Spet \ZZ$.
\end{lemma}

\begin{proof}
    Since $\omega_2$ is supported only at $2$ with stalk given by
    $A^\times/(A^\times)^2$ where $A=\ZZ_{(2)}^{sh}$ is the strict Henselization, it is enough to compute the value of this group with its structure as a module over the absolute Galois group  $\widehat{\ZZ}$ of $\FF_2$ (cf.\ \cite[Corollary II.3.11]{milne-etale}).
    Let $\WW=\WW(\overline{\FF_2})$ be the ring of Witt vectors. There is an injection $A\hookrightarrow \WW$ and $\WW$ is the $2$-adic completion of $A$. We will see that the induced map $A^\times/(A^\times)^2 \to \WW^\times/(\WW^\times)^2$ will turn out to be an isomorphism.
    
    To prove that this map is injective, it suffices to show that if $u\in A^\times$ is a square in $\WW^\times$, then it is already a square in $A^\times$. To see this, let $R=A[x]/(x^2-u)$.
    This is a finite $A$-algebra with $2$-adic completion
    $R_2^\wedge\iso \WW[x]/(x^2-u)$. By the Hensel property for $A$ and $\WW$, the ring $R$ is a product of either $1$ or $2$ local rings
    (see for example~\cite[\href{https://stacks.math.columbia.edu/tag/04GG}{Tag~04GG}]{stacks-project})
    and $R_2^\wedge$ is a product of the same number by looking at fraction fields.
    If $u$ is a square in $\WW$, then $R_2^\wedge$ is a product of $2$ local rings,
    but then the same is true of $R$.
    
    Next, we explicitly describe $\WW^\times/(\WW^\times)^2$, which will help us prove surjectivity of the above quotient map.
    Let $U_n=\{u\in \WW^\times:v_2(u-1)\geq n\}$, where $v_2$ denotes the 2-adic valuation. One has $\WW^\times/U_1\iso\overline{\FF}_2^\times$ and $U_n/U_{n+1}\iso\overline{\FF}_2$ for $n\geq 2$. The snake lemma for the diagram
    \[\xymatrix{
    0 \ar[r] &U_1 \ar[r]\ar[d] &\WW^\times \ar[r]\ar[d] &\overline{\FF}_2^\times \ar[r]\ar[d] & 0\\
    0 \ar[r] &U_1 \ar[r] &\WW^\times \ar[r] &\overline{\FF}_2^\times \ar[r] & 0,
    }
    \]
    where the vertical maps are all given by squaring, gives an isomorphism $U_1/U_1^2 \iso \WW^\times/(\WW^\times)^2$, as squaring is an isomorphism on $\overline{\FF}_2^\times$. This also shows that the kernel of squaring on $U_1$ is isomorphic to $\ZZ/2$, identified as $\{\pm 1\}\subset \WW^\times$. 
    
    Classical results imply that there is a logarithmic isomorphism $U_2\iso \WW$ (see for example the argument in~\cite[Sec.~II.3]{serre-course}). 
    Thus, there is an exact sequence $$0\rightarrow \WW\iso U_2 \rightarrow U_1\rightarrow\overline{\FF}_2\rightarrow 0.$$
    Using the snake lemma for the squaring map for this sequence as above, gives an exact sequence
    $$0\rightarrow \ZZ/2\rightarrow\overline{\FF}_2
    \xrightarrow{\partial}\overline{\FF}_2\rightarrow U_1/U_1^2\rightarrow \overline{\FF}_2\rightarrow 0,$$
    where we have used that $\overline{\FF}_2$ is $2$ torsion.

    The boundary map $\partial$ is computed by lifting $x\in \overline{\FF}_2$ to $U_1$ as $1+2\tilde{x}$ for some
    $\tilde{x}$ lifting $x$ to $\WW$ and then squaring, to find $1+4\tilde{x}+4\tilde{x}^2=1+4(\tilde{x}+\tilde{x}^2)$,
    which is in $U_2$ with residue modulo squares given by $x+x^2$. We see that $\partial$ is surjective and that $\WW^\times/(\WW^\times)^2 \iso U_1/U_1^2\iso\overline{\FF}_2$. Explicitly, this isomorphism sends $1 + 2\tilde{x}$ to $\tilde{x}\! \mod 2$.
    
    Returning to the question of surjectivity of $A^\times/(A^\times)^2\to \WW^\times/(\WW^\times)^2$, since the residue fields of $A$ and $\WW$ agree, we can lift any element of $\overline{\FF}_2$ in the map above to an element of the form $1+2\tilde{x}$ with $\tilde{x}\in A \subset \WW$; moreover, $1+2\tilde{x}$ will be in $A^\times$ by the Hensel property.
    It follows that $A^\times/(A^\times)^2\rightarrow \WW^\times/(\WW^\times)^2$ is surjective as well and that both groups are
    Galois-equivariantly isomorphic to $\overline{\FF}_2$.
\end{proof}

\newcommand{\TC}{\mathrm{TC}}

\begin{remark}
    The above result can also be read off from a much more sophisticated result due to Clausen, Mathew, and Morrow.
    They show in~\cite[Thm.~A]{cmm} that if $R$ is $p$-torsion free, henselian along $p$, and $R/p$ is perfect, then $\K(R)/p\we\TC(R)/p$.
    Let $A=\ZZ_{(2)}^{sh}$ and $\WW=\WW(\overline{\FF_2})$, the $2$-completion of $A$.
    Applying the Clausen--Mathew--Morrow result to $A$ and $\WW$, one obtains
    $$A^\times/(A^\times)^2\iso\K_1(A)/2\iso\TC_1(A)/2\iso\TC_1(\WW)/2\iso\K_1(\WW)/2\iso \WW^\times/(\WW^\times)^2,$$
    using that for any local ring $R$ we have an isomorphism $\K_1(R)\iso R^\times$ and that $\TC(R)/p\we\TC(R_p^\wedge)/p$ for any $R$ and any prime $p$, for
    example by~\cite[Lem.~5.3]{cmm}.
\end{remark}

To depict the spectral sequence \eqref{ss:picKO}, we will use symbols to denote the various sheaves
and \cref{fig:c2mackey} can be used as a legend.

\begin{table}[h] 
\begin{center}
    \begin{tabular}{|p{1.5cm}|c|c|c|c|c|c|c|}
\hline 
Symbol
    &$\Box$
    &$\Box^\times$
    &$\bullet$
    &$\circ$\\
\hline 
Sheaf
& $\Oscr$ 
& $\Oscr^\times$ 
& $\Oscr/2$ 
& $\Z/2$

\\
\hline 
\end{tabular}
\end{center}
\caption{An assortment of \'etale sheaves.}
\label{fig:c2mackey}
\end{table}   

\begin{sseqdata}[name = koss, Adams grading, classes = {draw = none } ]
    \class["\circ"](0,0)
    \foreach \x in {1,...,11} {
        \class["\circ"](-\x,\x)
    }

    \class["\Box^\times"](1,0)
    \foreach \y in {1,3,5,7,9,11} {
        \FPeval{\x}{clip(1-\y)}
        \class["\circ"](\x,\y)
    }
    \foreach \y in {4,6,8,10,12,14} {
        \FPeval{\x}{clip(1-\y)}
        \class["\bullet"](\x,\y)
    }
    \class["\bullet"](-1,2)

    \foreach \y in {1,7,9,11,13} {
        \FPeval{\x}{clip(3-\y)}
        \class["\bullet"](\x,\y) }
    \class["\bullet",red](0,3)
    \class["\bullet"](-2,5)

    \class["\Box"](5,0)
    \foreach \y in {2,4,8,10,12,14} {
        \FPeval{\x}{clip(5-\y)}
        \class["\bullet"](\x,\y)
    }
    \class["\bullet",red](-1,6)

    \foreach \y in {1,3,5,7,9,11,13} {
        \FPeval{\x}{clip(7-\y)}
        \class["\bullet"](\x,\y)
    }

    \class["\Box"](9,0)
    \foreach \y in {2,4,6,8,10,12,14} {
        \FPeval{\x}{clip(9-\y)}
        \class["\bullet"](\x,\y)
    }

    \foreach \y in {1,3,5,7,9,11,13} {
        \FPeval{\x}{clip(11-\y)}
        \class["\bullet"](\x,\y)
    }

    \class["\Box"](13,0)
    \foreach \y in {2,4,6,8,10,12,14} {
        \FPeval{\x}{clip(13-\y)}
        \class["\bullet"](\x,\y)
    }

    \foreach \y in {1,3,5,7,9,11,13,15} {
        \FPeval{\x}{clip(15-\y)}
        \class["\bullet"](\x,\y)
    }

    \foreach \y in {4,6,8,10,12,14} {
        \FPeval{\x}{clip(17-\y)}
        \class["\bullet"](\x,\y)
    }

    \foreach \y in {5,7,9,11,13,15} {
        \FPeval{\x}{clip(19-\y)}
        \class["\bullet"](\x,\y)
    }

    \foreach \y in {8,10,12,14} {
        \FPeval{\x}{clip(21-\y)}
        \class["\bullet"](\x,\y)
    }
   
    \foreach \y in {9,11,13,15} {
        \FPeval{\x}{clip(23-\y)}
        \class["\bullet"](\x,\y)
    }
    \foreach \y in {12,14} {
        \FPeval{\x}{clip(25-\y)}
        \class["\bullet"](\x,\y)
    }

    \foreach \z in {1,2,3,4,5} {
        \FPeval{\x}{clip(-4+\z)}
        \FPeval{\y}{clip(7+\z)}
        \d3(\x,\y)
        \replacesource[""]
        \replacetarget[""]
    }

    \d["",dashed]3(-4,7)
    \replacesource[""]
    \replacetarget[""]

    \foreach \z in {0,1,2,3,4,5,6,7,8} {
        \FPeval{\x}{clip(1+\z)}
        \FPeval{\y}{clip(4+\z)}
        \d3(\x,\y)
        \replacesource[""]
        \replacetarget[""]
    }

    \class["\bullet"](15,2)

    \d3(15,2)
    \replacetarget[""]
    \replacesource[""]

    \d3(5,0)
    \replacetarget[""]
    \replacesource["2\Box"]

    \foreach \z in {1,2,3,4,5,6,7,8,9} {
        \FPeval{\x}{clip(5+\z)}
        \FPeval{\y}{clip(0+\z)}
        \d3(\x,\y)
        \replacetarget[""]
        \replacesource[""]
    }

    \d3(13,0)
    \replacesource["2\Box"]
    \d3(14,1)

    \class["\bullet"](-10,13)
    \class["\bullet"](-8,15)

    \d["",dashed]3(-8,11)
    \replacesource[""]
    \replacetarget[""]

    \d3(-7,12)
    \replacesource[""]
    \replacetarget[""]

    \d[""]2(-1,1)
    \replacesource[""]
    \d[""]2(-2,2)
    \d[dashed]3(-1,2)
    \replacesource["?",red]

    \d[red,dash dot]3(0,3)
    \replacesource["i_*\circ",red]

    \d[dashed]3(-5,6)
    \replacesource["?"]
    \replacetarget["?"]

    \structline(0,1)(0,0)
    \structline(0,1)(0,3)
\end{sseqdata}

\begin{figure}[h]
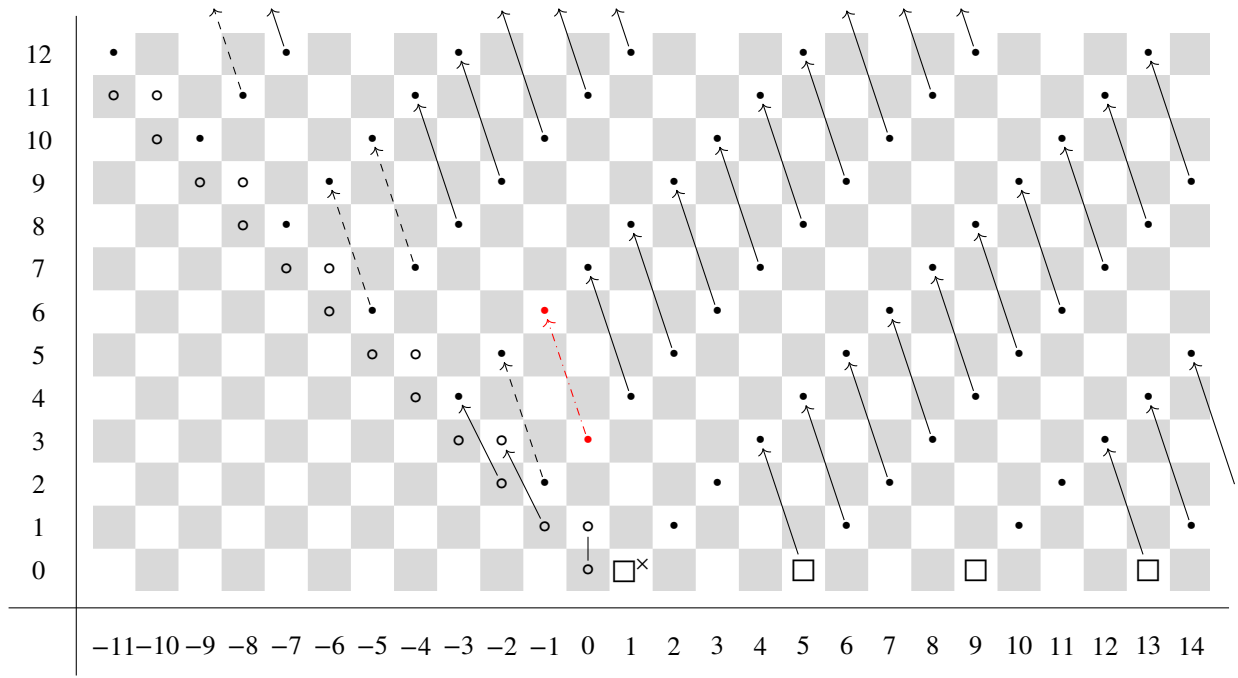

\printpage[ name = koss, grid=chess, page =
0,xscale=0.573,yscale=0.573,x range = {-11}{14},y range = {0}{12}]
    \caption{The $\Escr_2$-page of the spectral sequence~\eqref{ss:picKO}.
    All differentials on all pages above the anti-diagonal line $x+y=4$ agree
    with their linear counterparts by~\cite{mathew-stojanoska}. Not all
    information is shown in degrees $\leq -2$. Dashed black arrows potentially
    differ from their linear partners, but they do not figure into the
    calculation of $\pi_0\Shpic_{\Oscr_\KO}$. The dashed and dotted red arrow
    is non-linear and figures into the calculation of
    $\pi_0\Shpic_{\Oscr_\KO}$.}
    \label{fig:e2pickoss}
\end{figure}
\begin{figure}[h]
\printpage[ name = koss, grid=chess, page = 4,xscale=1,yscale=1,x range = {-1}{13},y range = {0}{3}]
    \caption{A part of the $\Escru_4$-page of the spectral
    sequence~\eqref{ss:picKO}.}
    \label{fig:e4pickoss}
\end{figure}

\cref{fig:e2pickoss} and \cref{fig:e4pickoss} show the spectral
sequence~\eqref{ss:picKO}. Several lemmas explain the nature of the
differentials and the calculation of the $\Escr_4$-page.

\begin{lemma}
    The $\Escru_4$-page is zero in column 0 above row 3. 
\end{lemma}
\begin{proof}
    Note that our spectral sequence consists on the $\Escru_2$-page of quasi-coherent sheaves above
    the antidiagonal $x+y = t = 1$. We will identify quasi-coherent sheaves on $\Spec \pi_0\KO$ with
    their abelian groups of global sections. 
    
    Since our spectral sequence can be seen as the sheafification of a presheaf of Picard
    homotopy fixed point spectral sequences, we can freely use the tools from
    \cite{mathew-stojanoska}. In particular, \cite[Comparison Tool 5.2.4]{mathew-stojanoska}
    implies that any $d_3$-differential originating from above the $x+y = t=3$ antidiagonal can
    be directly read off its counterpart in the homotopy fixed point spectral sequence for
    $\KU^{hC_2}\simeq \KO$. As in \cite[Example 7.1.1]{mathew-stojanoska}, the claim follows.
\end{proof}

\begin{lemma}[$d_3^{3,3}$]
    The differential $d_3^{3,3}\colon\bullet\rightarrow\bullet$ is given by $x\mapsto
    x+x^2$. In particular, it is a surjective map of sheaves and the kernel is
    $i_*\circ$, where $i\colon\Spet\FF_2\rightarrow\Spet\ZZ$ is the closed
    inclusion.
\end{lemma}

\begin{proof}
    The first claim follows from~\cite[Theorem 6.1.1]{mathew-stojanoska}, see also Example 7.1.1 in loc.cit. for the worked example in the case of the abelian group version of the spectral sequence \eqref{ss:picKO}. The map is surjective away from $2$ since
    both sides vanish in that case. At $2$, the map is surjective because
    $\Oscr/2$ has stalks given by separably closed fields. The identification
    of the kernel is similar.
\end{proof}

\begin{remark}
    By \cite[Proposition 7.15]{gepner-lawson}, the differentials $d_2^{1,0}, d_2^{2,0}$ and $d_3^{2,1}$
    are nonzero on global sections (where our spectral sequence is isomorphic, at least before
    differentials, to the usual Picard spectral sequence for $\KO$). The first two differentials
    have $\ZZ/2$ as source and are thus determined by global sections: $d_2^{1,0}$ is an isomorphism
    and $d_2^{2,0}$ is the unique injection $\Z/2 \to \Oscr/2$. The differential $d_3^{2,1}\colon
    \Oscr/2 \to \Oscr/2$ is not determined by global sections, however, and thus remains unresolved.
    None of these differentials are needed for our computation of the Picard sheaf and hence of
    $\LBr(\KO)$, though their result on global sections is used in the Gepner--Lawson computation of
    $\Br(\KU|\KO)$, which we will come back to in \cref{rem:KUKO}..
\end{remark}

These computations determine the associated graded of $\pi_0{\Shpic}_{\Oscr_\KO}$, but we can also resolve the extension problems as follows.

\begin{proposition}\label{prop:PicKO}
    There is a filtration on $\pi_0{\Shpic}_{\Oscr_\KO}$ with associated graded
    pieces ${\ZZ}/2$, $\ZZ/2$, and $i_*\Z/2$, where $i$ is the closed inclusion $\Spec \FF_2 \to \Spec \ZZ$. There is a surjective map from the constant sheaf $\ZZ/8 $ to $\pi_0\Shpic_{\Oscr_\KO}$, resulting in a non-trivial extension
    \begin{equation}\label{eq:extensionPicKO}
        0 \to i_*\ZZ/2 \to \pi_0\Shpic_{\Oscr_\KO} \to \ZZ/4 \to 0.
    \end{equation}
\end{proposition}
\begin{proof}
    The first statement was proved in the lemmas above, namely we get a filtration on the $\E_\infty$-page of the spectral sequence \eqref{ss:picKO} with
    \[ \xymatrix{ 
    \F^2\ar@{^{(}->}[r] \ar[d]^{\simeq} &\F^1 \ar@{^{(}->}[r] \ar@{->>}[d] & \pi_0 \Shpic_{\Oscr_\KO}\ar@{->>}[d] \\
    i_* \ZZ/2 & \ZZ/2 & \ZZ/2.
    } \]
    This filtration gives an inclusion $  i_*\ZZ/2\cong\F^2 \to \pi_0\Shpic_{\Oscr_\KO}$, and we need to identify the quotient $Q$ with $\ZZ/4$. This quotient sits in an extension 
    \begin{equation}\label{eq:extenstionQ}
    0 \to \ZZ/2 \to Q \to \ZZ/2 \to 0.
    \end{equation}
    
    The filtration implies that the group of global sections $\H^0(\Spet \ZZ, \pi_0
    \Shpic_{\Oscr_\KO})$ is a
    finite group of cardinality at most $8$. On the other hand, note that since $\H^1(\Spet \ZZ,
    \Gm) =\Pic(\ZZ)=0$, \cref{prop:omni} implies that the homomorphism $\Pic(\KO) \to \H^0(\Spec \Z, \pi_0\Shpic_{\Oscr_\KO})$ is an injection. Composing with the isomorphism
    \[\ZZ/8 \to \Pic(\KO), \qquad [1] \to \Sigma \KO,\]
    we obtain a map of sheaves $\ZZ/8\to\pi_0\Shpic_{\Oscr_\KO}$, which must be an isomorphism on global sections.
    
    The above also gives a map $\ZZ/8\to Q$ that is the surjection $\ZZ/8 \to \ZZ/4$ on global sections, implying that the extension \eqref{eq:extenstionQ} is non-trivial.
   But the only non-trivial extension of $\ZZ/2$ by $\ZZ/2$ on $\Spec \ZZ$, which has $\ZZ/4$ as global sections, is the constant sheaf $\ZZ/4$.\footnote{Indeed, $\Ext_{\Spec \Z}(\Z, \Z/2) \cong \H^1(\Spec \Z;\Z/2) =0$ and thus the short exact sequence \[
   0 \to \Z \xrightarrow{2}\Z \to \Z/2 \to 0\]
   implies that $\Ext_{\Spec \Z}(\Z/2, \Z/2) \cong \coker(\Z/2 \xrightarrow{2}\Z/2) \cong \Z/2$.}
    This identifies the quotient in \eqref{eq:extensionPicKO}, and to see that this extension is also not split, we again compare with the global sections.
\end{proof}

\begin{corollary}\label{cor:PicKOR}
Let $R$ be an \'etale extension of $\ZZ$. Then there is a short exact sequence
\[ 0\to \Pic(R) \to \Pic(\KO_R) \to (\pi_0\Shpic_{\Oscr_\KO})(R) \to 0.\]
If $\Spec R$ is connected, the last term sits in an extension of the form
\[0 \to (\ZZ/2)^d \to (\pi_0\Shpic_{\Oscr_\KO})(R) \to \ZZ/4 \to 0, \]
where $d$ is the number of factors when decomposing $R/2$ as a product of fields. 
\end{corollary}
\begin{proof}
    We first show the second part. The long exact sequence in cohomology associated to the extension in  \cref{prop:PicKO} takes the form
   \[0 \to (\ZZ/2)^d \to (\pi_0\Shpic_{\Oscr_\KO})(R) \to \ZZ/4 \to \H^1(R; i_*\ZZ/2) \to \cdots \]
   The composite $\Pic(\KO_R) \to (\pi_0\Shpic_{\Oscr_\KO})(R) \to \ZZ/4$ is a surjection and thus we obtain the second claim. 
   
   For the first part, we can assume that $\Spec R$ is connected and thus $R$ a regular integral
    domain. From \cref{prop:omni}, we have a natural exact sequence 
   \[ 0\to \Pic(R) \to \Pic(\KO_R) \to (\pi_0\Shpic_{\Oscr_\KO})(R) \xrightarrow{\partial_R} \Br(R).\]
   Since $\Pic(\KO_R)$ maps surjectively onto $\ZZ/4$, the image of $\partial_R$ is the image of the restriction $\partial_R'\colon (\ZZ/2)^d \to \Br(R)$. The map $R \to R[\tfrac12]$ induces a commutative square
   \[
   \xymatrix{
   (\ZZ/2)^d\ar[d]\ar[r] & \Br(R) \ar[d] \\
   0 \ar[r] & \Br(R[\textstyle{\tfrac12}]),
   }
   \]
   in which the horizontal arrows are the restricted boundaries $\partial'$ for $R$ and $R[\tfrac12]$ respectively.
   The right-hand vertical map is an injection by \cref{thm:BrauerProperties} since $\Spec R[\tfrac12] \subset \Spec R$ is dense. Thus $\partial_R' =0$.
\end{proof}

\begin{remark}\label{rem:17}
	As a consequence of the preceding corollary, we see that it is not true that for 
	every \'etale extension $\Z \subset R$ with $\Spec R$ connected, we have 
	$\Pic(\KO_R)\cong \Pic(R) \times \Z/8$ or $\Pic(R) \times \Z/4$. For example take the field
	$K = \Q(\sqrt{17})$, whose ring of integers is $\Z[\omega]$, where $\omega = \frac{1+\sqrt{17}}2$, and set
	$R = \Z[\omega][\frac1{17}]$. Here we have $2 = -(1+\omega)(2-\omega)$ and thus 
	$R/2 \cong \FF_2\times \FF_2$. We obtain $\Pic(\KO_R) \cong \Z/8 \times \Z/2$. In the Picard spectral sequence for $\KO_R$, the ``exotic'' elements arise as the kernel of the $d_3$-differential
	\[d_3\colon R/2 \cong\H^3(C_2; \pi_2\KU_R) \to\H^6(C_2; \pi_4\KU_R) \cong R/2, \quad x\mapsto x + x^2\]
	is bigger than $\Z/2$, namely $(\ZZ/2)^2$ in our example.  
	
	How can we understand these additional classes? Let us sketch a conjectural general picture of 
	the filtration on $\Pic(A)$ from the Picard spectral sequence for a faithful $G$-Galois extension 
	$A \to B$. Let $M\in \Pic(A)$. The $0$-line detects the image $M\tensor_A B \in \Pic(B)$. If 
	$M\tensor_AB\simeq B$ (and such an equivalence is chosen), the $1$-line $\H^1(G;\pi_0B)$ describes 
    how the $G$-action on $\pi_*(M\tensor_A B)$ is twisted in comparison to that on $\pi_*B$. Thus, the 
	$\E_2$-term of the homotopy fixed point spectral sequence for $(M\tensor_A B)^{hC_2} \simeq M$ 
	is isomorphic to that for $B^{hC_2} \simeq A$ if $M$ has filtration at least $2$, which we will 
	assume now. We fix such an isomorphism. We conjecture that if $M$ has filtration $i\geq 2$, its reduction to 
    $\H^i(G,\pi_i\Shpic(B))\iso\H^i(G; \pi_{i-1}B)$ equals $d_i(1)$ in the homotopy fixed point spectral sequence 
	for $(M\tensor_A B)^{hG} \simeq M$. 
	
	Back to our example, this means that the three non-trivial classes in
    $\Pic(\KO_R)$ of filtration 
	$3$ correspond conjecturally to invertible $\KO_R$-modules $M$ such that $d_3^M(1)$ is $1$, $\omega$ and 
	$1+\omega$ respectively. 
\end{remark}  
    
The identification of $\sPic{\KO}$ allows us to compute the local Brauer group of $\KO$. Recall in this context that Gepner and Lawson proved in~\cite[Proposition~7.17]{gepner-lawson} that the subgroup
$\Br(\KU|\KO)\subseteq\Br(\KO)$ of classes killed by the extension
$\KO\rightarrow\KU$ is isomorphic to $\ZZ/2$. We will show that $\LBr(\KO)$ is also $\ZZ/2$, and in fact it will be isomorphic to $\Br(\KU|\KO)$.

\begin{theorem}
    \label{ex:ko}
	There is an isomorphism $\LBr(\KO) \cong \ZZ/2$. The unique non-trivial
    class is killed by the \'etale cover
    \[\KO\rightarrow\KO[\tfrac{1}{2},\zeta_4]\times\KO[\tfrac{1}{3},\zeta_3].\]
\end{theorem}

Here, we use that the cyclotomic fields $\QQ(\zeta_4)$ and $\QQ(\zeta_3)$ are
ramified only at the primes $(2)$ and $(3)$, respectively,
to produce $\KO[\tfrac{1}{2},\zeta_4]$ and $\KO[\tfrac{1}{3},\zeta_3]$ as commutative
ring spectra.

\begin{proof}
    To use the exact sequence in \Cref{prop:omni}, we first need to compute $\H^1(\Spet \ZZ, \pi_0\Shpic_{\Oscr_\KO} )$, which we will do using \Cref{prop:PicKO}.
    Since there is a unique $\ZZ/2$-Galois extension of $\Spec\FF_2$, \cref{thm:closedimmersion} implies
    $$\H^1(\Spec\ZZ,j_*\ZZ/2)\iso\H^1(\Spec\FF_2,\ZZ/2)\iso\ZZ/2.$$
	Moreover $\H^1(\Spec\ZZ,\ZZ/4)=0$ as there are no unramified $\ZZ/4$-Galois extensions of
	$\QQ$. Since furthermore $$\H^0(\Spec\ZZ, \sPic{\KO}) \cong \Pic(\KO) \to \H^0(\Spec\ZZ, \ZZ/4)$$ is surjective, the long exact cohomology sequence associated with the short exact
    sequence of sheaves in \eqref{eq:extensionPicKO} implies that $\H^1_{\et}(\Spec \ZZ, \pi_0\sPic{\KO})$ is isomorphic to $\ZZ/2$.

    To conclude $\LBr(\KO) \cong \ZZ/2$ using \cref{prop:omni}, it remains to show the vanishing of the differential
    \[\H^1(\Spec\ZZ,\pi_0\sPic{\KO})\xrightarrow{d_2}\H^3_{\et}(\Spec\ZZ,\Gm).\]
    We show this by comparison to $\KU$: 
    The map $\KO\rightarrow\KU$ induces a map
    of presheaves
    $\Shlbr_{\KO}\rightarrow\Shlbr_{\KU}$ on the \'etale site of $\Spec\ZZ$, which we identify with either of the \'etale sites of $\KO$ and $\KU$ using the isomorphism $\pi_0\KO \cong \ZZ \cong \pi_0\KU$. 
  Thus, we get an induced map of
    descent spectral sequences and in particular a commutative diagram
    $$\xymatrix{
    \H^1(\Spec\ZZ,\sPic{\KO})\ar[r]^-{d_2}\ar[d]&\H^3(\Spec\ZZ,\Gm)\ar[d]\\
    \H^1(\Spec\ZZ,\sPic{\KU})\ar[r]^-{d_2}&\H^3(\Spec\ZZ,\Gm),
    }$$
    where the right vertical map is an equality. Since
    $\H^1(\Spec\ZZ,\pi_0\sPic{\KU})\iso\H^1(\Spec\ZZ,\ZZ/2)=0$, we
    see that the top differential must vanish. Therefore,
    $\LBr(\KO)\iso\H^1(\Spec\ZZ,\pi_1\ShBPic_{\Oscr_{\KO}})\iso\ZZ/2$.
	
	For the second part of our claim, note first that $\Br(\ZZ[\frac12, \zeta_4])$
    and $\Br(\ZZ[\frac13, \zeta_3])$ vanish by \cref{ex:BrauerComputations}.     
    The Brauer groups of $\ZZ[\frac12, \zeta_4]$ and $\ZZ[\frac13, \zeta_3]$ agree with
    the second \'etale cohomology with $\Gm$-coefficients since the rings are
    regular and noetherian. Using \cref{prop:omni} again, we thus see that the non-trivial class in $\LBr(\KO)$ must be killed by the extension $\KO\rightarrow\KO[\tfrac{1}{2},\zeta_4]\times\KO[\tfrac{1}{3},\zeta_3]$ if the image of the non-trivial element of
    $\H^1(\Spec\ZZ,j_*\ZZ/2)$ vanishes in $\H^1(\Spec \ZZ[\frac12,
    \zeta_4],j_*\ZZ/2)$
    and $\H^1(\Spec \ZZ[\frac13, \zeta_3],j_*\ZZ/2)$. This is clear in the first
    case as $j_*\ZZ/2$ restricted to $\Spec \ZZ[\frac12]$ vanishes. In the second
    case, we use that the extension $\FF_2 \subset \FF_4 \iso \FF_2[\zeta_3]$ kills
    the non-trivial element of $\H^1(\Spec\FF_2,\ZZ/2)$.
\end{proof}
  
\begin{remark}\label{rem:KUKO}  Note that since
    $\LBr(\KU)=0$, functoriality of the local Brauer group implies that the non-zero class $\alpha\in\LBr(\KO)\iso\ZZ/2$ is killed by the
    $\ZZ/2$-Galois extension $\KO\rightarrow\KU$, i.e.\ lies in the relative Brauer group $\Br(\KU|\KO)$. By the main result of \cite{gepner-lawson}, $\LBr(\KO)$ thus agrees with $\Br(\KU|\KO)$ though a priori we only get an inclusion. This gives a
    new proof of that the Galois-cohomological Brauer class found
    in~\cite[Proposition~7.15]{gepner-lawson} is representable by an Azumaya
    algebra, which Gepner and Lawson prove instead with an unstable descent
    spectral sequence. See also \cref{ex:KTMFGalois} for another perspective.
    
    We urge the reader to consider  the analogue of the descent spectral sequence computation of
    $\Br(\KU|\KO)$ as in \cite[Figure 7.2]{gepner-lawson} in the case of the relative Brauer group
    of $\KO[\tfrac13, \zeta_3]$ with respect to  $\KU[\tfrac13, \zeta_3]$. As the class in
    filtration six contributing to $\Br(\KU|\KO)$ has to die in $\Br(\KO[\tfrac13, \zeta_3])$, there
    must be a new $d_3$ killing it. This $d_3$ is given by the formula in \cite[Theorem
    6.1.1]{mathew-stojanoska}, the point being that the image of $x\mapsto x +x^2$ on $\ZZ[\tfrac13,
    \zeta_3]/2 \cong \FF_4$ is $\ZZ/2 =\FF_2 \subset \FF_4$.
\end{remark}

\section{Brauer groups of nonconnective spectral DM stacks}\label{sec:dm}

In this section, we turn to Brauer groups of nonconnective spectral Deligne--Mumford (DM) stacks. A
significant difference will be that the Brauer group is in general no longer $\pi_0$ of the global
sections of the Brauer sheaf, yielding to a distinction between Brauer group and cohomological
Brauer group, which we will explain below.

To fix notation, we recall the following definition from Lurie~\cite{sag}.

\begin{definition}
    A {\bf nonconnective spectral DM stack} is a spectrally ringed $\infty$-topos
    $(\Xscr,\Oscr)$ such that there exists a covering $\coprod_{i\in I}
    U_i\rightarrow\ast$ of the final object where for each $i$ there is an
    equivalence
    $(\Xscr_{/U_i},\Oscr|_{\Xscr_{U_i}})\we\Spet R_i$ for some
    commutative ring spectrum $R_i$.\footnote{Lurie writes $\SpetLurie R$ for what we write as
    $\Spec R$.} If $\Oscr$ is connective, we say that
    $(\Xscr,\Oscr)$ is a {\bf connective spectral DM stack}; if $\Oscr$ is
    discrete, we say that $(\Xscr,\Oscr)$ is a {\bf classical DM stack}.
\end{definition}

\begin{remark}
    \begin{enumerate}
        \item[(a)] In~\cite{sag}, Lurie calls connective spectral DM stacks simply spectral DM
            stacks.
        \item[(b)] Given a nonconnective spectral DM stack $(\Xscr,\Oscr)$,
            there is a diagram $(\Xscr,\Oscr)\rightarrow(\Xscr,\tau_{\geq
            0}\Xscr)\leftarrow(\Xscr,\pi_0\Xscr)$ of nonconnective spectral DM
            stacks. The right arrow is the inclusion of the classical locus, at least if $\Xscr$ arises from a $1$-topos.
    \end{enumerate}
\end{remark}

\begin{construction}
For a nonconnective spectral DM stack, \'etale sheaves on $\Xscr$ are equivalent to \'etale sheaves
    on the site $\mathrm{Aff}^{\et}_{/(\Xscr, \Oscr)}$ of \'etale maps $\Spec R \to (\Xscr, \Oscr)$ for some
    commutative ring spectrum $R$. Restricting the sheaves $\ShPic$, $\ShBr$, and $\ShLBr$ on
    $\CAlg_{\mathbb{S}}^{\op} \simeq \mathrm{Aff}$ from \cref{sec:affine}, we obtain sheaves
    $\ShPic_\Oscr$, $\ShBr_\Oscr$, and $\ShLBr_\Oscr$ on $\mathrm{Aff}^{\et}_{/(\Xscr, \Oscr)}$ or, equivalently, on $\Xscr$.
\end{construction}

\begin{remark}
    There is a natural map $\ShBPic_\Oscr\rightarrow\ShBr_\Oscr$ which induces
    an equivalence $\ShBPic_\Oscr\we\ShLBr_\Oscr$, since again this can be
    checked locally. The computation of the homotopy sheaves of $\ShLBr_\Oscr$
    given in \cref{lem:hosheaves} goes through verbatim here.
\end{remark}

\begin{example}
    In general, $\pi_0\ShBr_{\tau_{\geq 0}\Oscr}=\pi_0\ShBr_{\pi_0\Oscr}=0$ since Brauer classes on
    connective commutative ring spectra are \'etale-locally trivial by \cite[Theorem
    5.11]{antieau-gepner}. We also
    have $\pi_1\ShBr_{\tau_{\geq
    0}\Oscr}\iso\pi_1\ShBr_{\pi_0\Oscr}\iso\ZZ$ by the computation of Picard groups of connective
    commutative ring spectra. On the other hand,
    $\pi_0\ShBr_\Oscr$ and $\pi_1\ShBr_\Oscr$ are highly dependent on the
    nature of $\Oscr$ itself.
\end{example}

\begin{definition}
    We let $\Br'(\Xscr,\Oscr)=\pi_0\Gamma(\Xscr,\ShBr_\Oscr)=\pi_0(\ShBr_\Oscr(\Xscr))$.
    This is the {\bf cohomological Brauer group} of $\Xscr$. Similarly, the
    {\bf cohomological local Brauer group} of $(\Xscr,\Oscr)$ is
    $$\LBr'(\Xscr,\Oscr)=\pi_0\ShB\ShPic_\Oscr(\Xscr).$$
    We call the space of global sections $\ShBr_\Oscr(\Xscr)$ the Brauer space
    and similarly for the local Brauer space $\ShBPic_\Oscr(\Xscr)\we\ShLBr_\Oscr(\Xscr)$.
\end{definition}

\begin{remark}
    The subgroup $\LBr'(\Xscr,\Oscr)\subseteq\Br'(\Xscr,\Oscr)$ consists of
    those cohomological Brauer classes which are \'etale locally trivial on
    $\Xscr$. Since $(\Xscr,\Oscr)$ is a nonconnective spectral DM stack
    this means that for $\alpha\in\LBr'(\Xscr,\Oscr)$, there is a surjective family of \'etale maps $\{p_i\colon\Spet
    R_i\rightarrow(\Xscr,\Oscr)\}_{i\in I}$ such that $p_i^*\alpha=0$ for all
    $i$.
\end{remark}

\begin{construction}\label{constr:spectral-sequence}
In order to compute $\Br'(\Xscr,\Oscr)$ and $\LBr'(\Xscr,\Oscr)$, it is
convenient to deloop $\ShBr_\Oscr$ and $\ShLBr_\Oscr$ and view them as
presheaves of spectra; \'etale sheafification yields sheaves
$\Shbr_\Oscr$ and $\Shlbr_\Oscr$. As such we have $\pi_t\Shbr_\Oscr\iso\pi_t\ShBr_\Oscr$ for all
$t\in\ZZ$ and similarly for $\Shlbr_\Oscr$; in particular, the homotopy sheaves vanish for $t<0$.
We have $\Omega^\infty\Shbr_\Oscr(\Xscr)\we\ShBr_\Oscr(\Xscr)$. Analogously to
    \cref{prop:hypersheaf}, we argue that $\Shbr_\Oscr$ and $\Shlbr_\Oscr$ are Postnikov complete.
    Thus, we obtain a descent spectral sequence
$$\E_2^{s,t}=\H^s(\Xscr,\pi_{t}\Shbr_\Oscr)\quad\mathlarger{\Longrightarrow}\quad \pi_{t-s}\Shbr_\Oscr(\Xscr)\underset{t-s\geq
0}\iso\pi_{t-s}\ShBr_\Oscr(\Xscr)$$
and similarly for $\Shlbr_\Oscr(\Xscr)$.
\end{construction}

For the following definition, recall that a quasi-coherent sheaf is perfect if it is dualizable or,
equivalently, if it becomes a compact object when restricted to an affine.
\begin{definition}\label{def:azumaya}
    A quasi-coherent sheaf $\Ascr$ of $\Oscr$-algebras on a nonconnective
    spectral DM stack $(\Xscr,\Oscr)$ is an {\bf Azumaya algebra} if the
    following equivalent conditions hold:
    \begin{enumerate}
        \item[(i)] $\Ascr$ is perfect, locally generates
            $\QCoh(\Xscr,\Oscr)$, and the natural map
            $\Ascr^\op\otimes_\Oscr\Ascr\rightarrow\ShEnd_\Oscr(\Ascr)$ is an
            equivalence;
        \item[(ii)] there is an \'etale cover $\{\Spet
            R_i\xrightarrow{p_i}(\Xscr,\Oscr)\}_{i\in I}$ such that $p_i^*\Ascr$ is an
            Azumaya $R_i$-algebra for all $i$.
    \end{enumerate}
\end{definition}

\begin{definition}\label{def:BrXO}
    Any Azumaya algebra $\Ascr$ on $(\Xscr,\Oscr)$ defines a point of $\ShBr_\Oscr$ and
    hence an element $[\Ascr]$ of $\Br'(\Xscr,\Oscr)$, called the class of
    $\Ascr$. If $\Ascr$ is an Azumaya algebra, then so is the opposite algebra
    $\Ascr^\op$ and we have $[\Ascr^\op]=-[\Ascr]$; if $\Bscr$ is a second Azumaya
    algebra, then $\Ascr\otimes_{\Oscr}\Bscr$ is Azumaya and
    $[\Ascr\otimes_{\Oscr}\Bscr]=[\Ascr]+[\Bscr]$. These assertions may be
    verified locally using \cref{def:azumaya}(ii) and
    \cref{def:brauer}(b). Let
    $\Br(\Xscr,\Oscr)\subseteq\Br'(\Xscr,\Oscr)$ be the subgroup consisting of
    the classes of Azumaya algebras.
    Let
    $\LBr(\Xscr,\Oscr)=\LBr'(\Xscr,\Oscr)\cap\Br(\Xscr,\Oscr)$ inside
    $\Br'(\Xscr,\Oscr)$. We call these the {\bf Brauer} and {\bf local Brauer
    groups} of $(\Xscr,\Oscr)$.
\end{definition}

\begin{example}\label{ex:affine-cohomologicalBrauer}
    For any commutative ring spectrum $R$, \cref{prop:hypersheaf} implies $\Br'(\Spec R) = \Br(\Spec R)$. 
\end{example}

\begin{definition}
    Let $(\Xscr,\Oscr)$ be a nonconnective spectral DM stack and let $\alpha\in\Br'(\Xscr,\Oscr)$.
    Using the inclusion $\ShBr_\Oscr\rightarrow\Cat_\Oscr$, the section
    $\alpha\in\Br'(\Xscr,\Oscr)$ defines a section of $\Cat_\Oscr$ and hence a
    stack of stable presentable $\infty$-categories, $\QCoh_{\Oscr,\alpha}$.
    This is the stack of \textbf{$\alpha$-twisted quasi-coherent sheaves} on
    $(\Xscr,\Oscr)$. The stable $\infty$-category of global sections will be
    denoted by $\QCoh(\Xscr,\alpha)$.
    An object $\Fscr\in\QCoh(\Xscr,\alpha)$ is {\bf perfect} if for every \'etale
    $p\colon\Spet R\rightarrow(\Xscr,\Oscr)$ the complex $p^*\Fscr$ is a
    compact object of $\QCoh(\Spet R,p^*\alpha)$. Note that the latter stable $\infty$-category is
    equivalent to $\Mod_A$ where $A$ is any Azumaya $R$-algebra with Brauer class
    $p^*\alpha$. We say that $\Fscr$ is a {\bf perfect local generator} if it is perfect and $p^*\Fscr$
    generates $\QCoh(\Spet R,p^*\alpha)$ for any $\Spet
    R\rightarrow(\Xscr,\Oscr)$.
\end{definition}

\begin{lemma}\label{lem:local}
    Let $(\Xscr,\Oscr)$ be a nonconnective spectral DM stack. If
    $\alpha\in\Br'(\Xscr,\Oscr)$, then
    $\alpha\in\Br(\Xscr,\Oscr)\subseteq\Br'(\Xscr,\Oscr)$ if and only if there
    exists a perfect local generator of $\QCoh(\Xscr,\alpha)$.
\end{lemma}

\begin{proof}
    If $\Ascr$ is an Azumaya algebra representing $\alpha$, define $\QCoh(\Xscr,\Ascr)$ as the limit
    of $\Mod_{\Ascr(R)}$ over all \'etale maps $\Spec R \to (\Xscr, \Oscr)$; this can be identified
    with a full subcategory of $\Mod_{\Ascr}(\Shv_{\Sp}(\Xscr))$. We have 
    $\QCoh(\Xscr,\Ascr)\we\QCoh(\Xscr,\alpha)$ and under this equivalence $\Ascr$
    corresponds to a perfect local generator. Conversely, given a perfect local
    generator $\Fscr$ of $\QCoh(\Xscr,\alpha)$, the sheaf of endomorphisms
    $\ShEnd_{\Oscr}(\Fscr)$ is an Azumaya algebra with class $\alpha$.
\end{proof}

Here is one example where every cohomological Brauer class is representable by
an Azumaya algebra.

\begin{proposition}\label{prop:etalecover}
    Let $(\Xscr,\Oscr)$ be a nonconnective spectral DM stack.
    If $(\Xscr,\Oscr)$ admits a finite \'etale cover $\pi\colon\Spet
    R\rightarrow(\Xscr,\Oscr)$, then
    $\Br(\Xscr,\Oscr)=\Br'(\Xscr,\Oscr)$.
\end{proposition}

\begin{proof}
    There is a compact generator $\Fscr$ of
    $\QCoh(\Spet R,\pi^*\alpha)$ by \cref{ex:affine-cohomologicalBrauer} and \cref{lem:local}.
    The pushforward
    $\pi_*\Fscr$ is a perfect local generator of $\QCoh(\Xscr,\alpha)$, as
    one can check \'etale locally.
\end{proof}

\begin{example}\label{ex:KTMFGalois}
    If a finite group $G$ acts on a commutative ring spectrum $R$, we obtain a finite \'etale map
    $\Spec R \to [\Spec R/G]$ to the stack quotient. In particular, $\Br([\Spec R/G]) \cong
    \Br'([\Spec R/G])$ by the preceding proposition. This is especially interesting if $R^{hG} \to
    R$ is a faithful $G$-Galois extension, when Galois descent implies that $\Mod_{R^{hG}} \simeq
    \QCoh([\Spec R/G])$. Examples include $\KO \to \KU$, $\TMF[\frac12] \to \TMF(2)$, and
    $\TMF[\frac13] \to \TMF(3)$ (see \cite{rognes} and \cite{mathew-meier}).
\end{example}

\cref{prop:etalecover} will not be enough to show the agreement of $\Br'$ and $\Br$ for the derived
moduli stack of elliptic curve since the moduli stack of elliptic curves does not have a finite
\'etale cover by an affine scheme \cite{mell-1connected}. This issue will be solved by
\cref{thm:spectraldm} below. Before we state it, we introduce the following definition needed for
its proof.

\begin{definition}
    Let $(\Xscr,\Oscr)$ be a nonconnective spectral DM stack. Let
    $\alpha\in\Br'(\Xscr,\Oscr)$ be a Brauer class and let
    $\Fscr\in\QCoh(\Xscr,\alpha)$ be a perfect local generator. We say that
    $\Fscr$ is a {\bf global generator} if $\Fscr$ is compact and if $\QCoh(\Xscr,\alpha)$ is compactly 
    generated by $\Fscr$.
\end{definition}

\begin{theorem}\label{thm:spectraldm}
    Let $(\Xscr,\Oscr)$ be a nonconnective spectral DM stack
    and fix $\alpha\in\Br'(\Xscr,\Oscr)$. If $\Xscr$ admits
    a Zariski open cover $\{\Uscr_i\}_{i=1}^n$ such that
    \begin{enumerate}
        \item[{\rm (a)}] for each $1\leq i,j\leq n$, the kernel of
            $\QCoh(\Uscr_i,\Oscr)\rightarrow\QCoh(\Uscr_i\cap\Uscr_j,\Oscr)$
            is generated by a single compact object $\Kscr_{i,j}$, and
        \item[{\rm (b)}] there is a global generator $\Fscr_i$ of
            $\QCoh(\Uscr_i,\alpha)$ for each $i=1,\ldots,n$,
    \end{enumerate}
    then $\alpha\in\Br(\Xscr,\Oscr)$ and there is a global generator of
    $\QCoh(\Xscr,\alpha)$.
\end{theorem}

The proof follows the work of~\cite{toen-derived} and~\cite{antieau-gepner}
which uses older arguments of B\"okstedt--Neeman~\cite{bokstedt-neeman} and Bondal--van
den Bergh~\cite{bondal-vdb} who showed that for a quasi-compact and
quasi-separated scheme $X$, the derived category of complexes of
$\Oscr_X$-modules with quasi-coherent cohomology sheaves admits a single compact generator, which is
global in the sense above.
Other important examples of $\Br=\Br'$ in the non-derived and derived context have been established
in~\cite{gabber,dejong-gabber,hall-rydh,chough}.

\begin{proof}
    Note first that each $\Uscr_i\subseteq\Xscr$ is relatively scalloped
    in the sense of~\cite[2.5.4.1]{sag}.\footnote{Quasi-affine morphisms are
    relatively scalloped; these will be enough for our
    applications.}

    We glue local perfect generators as
    in~\cite[Theorem~6.11]{antieau-gepner}
    or~\cite[Proposition~5.9]{toen-derived}, taking care in each step to
    produce a global generator. Let $\Yscr_k$ be the union
    $\Uscr_1\cup\cdots\cup\Uscr_k$ in $\Xscr$. It is enough to prove
    that there is a global generator
    of $\QCoh(\Yscr_k,\alpha)$ for each $k=1,\ldots,n$ and hence $\alpha|_{\Yscr_k}$ is in
    $\Br(\Yscr_k,\Oscr)$ for each $k$. The base case follows
    from assumption (b). Suppose the conclusion holds for some $1\leq k<n$.
    Set $\Wscr=\Yscr_k\cap\Uscr_{k+1}$ and consider the pullback square
    \begin{equation}\label{eq:a}\begin{gathered}\xymatrix{\QCoh(\Yscr_{k+1},\alpha)\ar[r]\ar[d]&\QCoh(\Uscr_{k+1},\alpha)\ar[d]\\
    \QCoh(\Yscr_k,\alpha)\ar[r]&\QCoh(\Wscr,\alpha)}\end{gathered}\end{equation}
    of stable presentable $\infty$-categories. As each
    inclusion $\Wscr\subseteq\Uscr_{k+1}$, $\Wscr\subseteq\Yscr_k$,
    $\Yscr_k\subseteq\Yscr_{k+1}$, and $\Uscr_{k+1}\subseteq\Yscr_{k+1}$ is
    quasi-affine and hence relatively scalloped, it follows from~\cite[2.5.4.3]{sag}, or rather its
    proof,\footnote{One just has to repeat the proof in the twisted setting and use that a left adjoint preserves compact objects if its right adjoint preserves filtered colimits.} that the functors
    in~\eqref{eq:a} preserve compact objects.

    We want to show that the kernel of
    $\QCoh(\Uscr_{k+1},\alpha)\rightarrow\QCoh(\Wscr,\alpha)$ is generated by
    the compact object
    $$\Gscr:=\Fscr_{k+1}\otimes_{\Oscr}\Kscr_{k+1,1}\otimes_{\Oscr}\cdots\otimes_{\Oscr}\Kscr_{k+1,k}.$$
    Compactness follows by construction. For generation,
    suppose that $\Mscr$ is an object of $\QCoh(\Uscr_{k+1},\alpha)$ which restricts to zero
    on $\QCoh(\Wscr,\alpha)$ and suppose additionally that the mapping spectrum $\Map(\Gscr,\Mscr)$ is
    zero. We want to show that $\Mscr\we 0$. But,
    $$0\we\Map(\Gscr,\Mscr)\we\Map(\Kscr_{k+1,1}\otimes_\Oscr\cdots\otimes_\Oscr\Kscr_{k+1,k},\ShMap(\Fscr_{k+1},\Mscr))$$
    by adjunction,
    where $\ShMap(\Fscr_{k+1},\Mscr)$ denotes the internal mapping spectrum, a quasi-coherent
    sheaf on $\Uscr_{k+1}$.
    Since the $\Kscr_{k+1,j}$ are compact generators of the kernels of
    $\QCoh(\Uscr_{k+1},\Oscr)\rightarrow\QCoh(\Uscr_{k+1}\cap\Uscr_j,\Oscr)$, 
    their tensor product is a compact generator of the kernel of $\QCoh(\Uscr_{k+1}) \to \QCoh(\Wscr)$. Denoting the inclusion $\Wscr \to \Uscr_{k+1}$ by $i$, it follows that
    $\ShMap(\Fscr_{k+1},\Mscr) \to i_*i^*\ShMap(\Fscr_{k+1},\Mscr)$ is an equivalence since its fiber lies in the kernel of $\QCoh(\Uscr_{k+1}) \to \QCoh(\Wscr)$.
    But, this implies that
    $\ShMap(\Fscr_{k+1},\Mscr)\we\ShMap(\Fscr_{k+1}|_{\Wscr},\Mscr|_\Wscr)$. The latter is zero as
    $\Mscr|_\Wscr\we 0$, so
    $\ShMap(\Fscr_{k+1},\Mscr)\we 0$ and hence the mapping spectrum $\Map(\Fscr_{k+1},\Mscr)$ is zero, which in
    turn implies that $\Mscr\we 0$ since $\Fscr_{k+1}$ is a compact generator of
    $\QCoh(\Uscr_{k+1},\alpha)$.

    Using that the square~\eqref{eq:a} is a pullback, the vertical fibers are
    equivalent stable $\infty$-categories. Thus, $\Gscr$ corresponds to a
    compact object of $\QCoh(\Yscr_{k+1},\alpha)$ which vanishes on $\Yscr_k$.
    On the other hand, by induction there is a global generator $\Hscr$ of
    $\QCoh(\Yscr_k,\alpha)$. Our goal will be to lift $\Hscr$ to $\Yscr_{k+1}$. The fact that
    $\QCoh(\Uscr_{k+1},\alpha)\rightarrow\QCoh(\Wscr,\alpha)$ is a localization
    and preserves compact objects implies that $\QCoh(\Wscr,\alpha)$ is
    generated by the image of $\Fscr_{k+1}$. Since the kernel is compactly
    generated by a compact object of $\QCoh(\Uscr_{k+1},\alpha)$ we are in the
    setting of Thomason's
    extension proposition~\cite[5.2.2]{thomason-trobaugh}
    (see~\cite[Corollary~0.9]{neeman-localization} for the generality needed
    here), which says that if $\Bscr\rightarrow\Cscr\rightarrow\Dscr$ is a Verdier sequence of
    idempotent complete
    stable $\infty$-categories, then an object $\Mscr\in\Dscr$ lifts to $\Cscr$ if and
    only if its class $[\Mscr]\in\K_0(\Cscr)$ lifts to $\K_0(\Dscr)$. Thus, possibly by replacing $\Hscr$ by
    $\Hscr\oplus\Sigma\Hscr$ (which always has vanishing class in $\K_0$), we see that the restriction of $\Hscr$ to
    $\QCoh(\Wscr,\alpha)$ lifts to a compact object $\Hscr_{k+1}$ of $\QCoh(\Uscr_{k+1},\alpha)$.
    Gluing $\Hscr$ and $\Hscr_{k+1}$ via the pullback~\eqref{eq:a}, we obtain a
    compact object $\Escr$ of $\QCoh(X,\alpha)$. Let $\Dscr=\Escr\oplus\Gscr$.
    We claim that $\Dscr$ is a global generator of $\QCoh(\Yscr_{k+1},\alpha)$.
    Verification is standard and left to the reader.
\end{proof}

\begin{corollary}\label{cor:BrauerCoincidences}
    If a nonconnective spectral DM stack $(\Xscr, \Oscr)$ satisfies the assumptions of
    \cref{thm:spectraldm} for every $\alpha \in S \subset \Br'(\Xscr, \Oscr)$, we have
    \begin{enumerate}
        \item[{\em (1)}] $\Br(\Xscr, \Oscr) = \Br'(\Xscr, \Oscr)$ if $S = \Br'(\Xscr, \Oscr)$, 
        \item[{\em (2)}] $\LBr(\Xscr,\Oscr)=\LBr'(\Xscr,\Oscr)$ if $S = \LBr'(\Xscr, \Oscr)$, and
        \item[{\em (3)}] $\LBrW(\Xscr,\Oscr)=\LBrW'(\Xscr,\Oscr)$ if $\Oscr$ is weakly $2$-periodic and $S= \LBrW'(\Xscr, \Oscr)$. 
    \end{enumerate}
\end{corollary}

This corollary will be applied in \cref{prop:BrIdentifications} to the derived moduli stack of elliptic curves.

\section{The $0$-affine case}

Let $(\Xscr,\Oscr)$ be a nonconnective spectral DM stack.

\begin{definition}
    We say that $(\Xscr,\Oscr)$ is \textbf{$0$-affine} if the global sections functor
    $$\Gamma\colon\QCoh(\Xscr,\Oscr)\rightarrow\Mod_{\Gamma(\Xscr,\Oscr)}$$
    is an equivalence; equivalently, $(\Xscr,\Oscr)$ is $0$-affine if $\Oscr$
    is a compact generator of $\QCoh(\Xscr,\Oscr)$.
\end{definition}

In classical algebraic geometry, there are few $0$-affine DM stacks. If $X$ is
a scheme, $X$ is $0$-affine if and only if it is
quasi-affine, which is to say quasi-compact and can be embedded as an open
subscheme of $\Spec A$ for some $A$. In this case, one can take
$A=\H^0(X,\Oscr)$. More generally, quasi-affine connective spectral DM stacks
are $0$-affine.

Remarkably, in the theory nonconnective spectral DM stacks, there is an
additional wealth of non-classical examples, as supplied by the following theorem
of~\cite{mathew-meier}.

\begin{theorem}[\cite{mathew-meier}]\label{thm:0affine-general}
    Let $(\Xscr, \Oscr)$ be a nonconnective spectral DM stack such that $\Oscr$ is weakly $2$-periodic. 
    Suppose that $(\Xscr, \pi_0\Oscr)$ is separated and noetherian and that the associated map
    $(\Xscr, \pi_0\Oscr) \rightarrow\Mscr_{\mathrm{FG}}$ to the moduli of formal groups is
    quasi-affine and flat. 
    Then $(\Xscr,\Oscr)$ is $0$-affine.
\end{theorem}

Our main example will be $(\Mscr, \Oscr)$, where $\Mscr$ is the moduli stack of elliptic curve and
$\Oscr$ is the weakly $2$-periodic sheaf of $\EE_{\infty}$-ring spectra defined by Goerss, Hopkins
and Miller \cite{TMF}. Later, the nonconnective spectral DM stack $(\Mscr, \Oscr)$ was reinterpreted
and reconstructed by Lurie to classify oriented spectral elliptic curves \cite{lurie-survey} and we
will refer to it as the \emph{derived moduli stack of elliptic curves}.

\begin{corollary}[\cite{mathew-meier}]\label{cor:Mell0-affine}
    The derived moduli stack $(\Mscr,\Oscr)$ of elliptic curves is $0$-affine, i.e.\
    $\Gamma\colon\QCoh(\Mscr,\Oscr)\xrightarrow{\we}\Mod_{\TMF}$ is an equivalence.
\end{corollary}

The main point of this section is to show that for a $0$-affine spectral DM stack,
the canonical map
$p\colon(\Xscr,\Oscr)\rightarrow\Spet\Gamma(\Xscr,\Oscr)$
induces an isomorphism $p^*\colon\Br(\Spet\Gamma(\Xscr,\Oscr))\iso\Br(\Xscr,\Oscr)$.

\begin{theorem}\label{thm:0affine}
    If $(\Xscr,\Oscr)$ is a $0$-affine nonconnective spectral DM
    stack with $R=\Gamma(\Xscr,\Oscr)$ and
    $p\colon(\Xscr,\Oscr)\rightarrow\Spet R$, then
    $p^*\colon\Br(R)\rightarrow\Br(\Xscr,\Oscr)$ is an isomorphism.
\end{theorem}

\begin{proof}
    By hypothesis, the functors $p^*\colon \Mod_R\rightarrow\QCoh(\Xscr,\Oscr)$ and
    $p_*\colon \QCoh(\Xscr,\Oscr)\rightarrow\Mod_R$ are symmetric monoidal adjoint equivalences.
    The functor $p^*$ preserves Azumaya algebras. It is enough to show
    that $p_*$ preserves Azumaya algebras.
    The condition that for an $\Oscr$-algebra $\Ascr$ we have
    $\Ascr^{\op}\otimes_\Oscr\Ascr\we\ShEnd(\Ascr)$ is preserved by $p_*$ since
    it is symmetric monoidal and hence also preserves internal mapping objects. We must see that if $\Ascr$ is a perfect local
    generator of $\QCoh(\Xscr,\Oscr)$, then $p_*\Ascr$ is a compact generator
    of $\Mod_R$. However, since $\Oscr$ is compact by the definition of
    $0$-affineness, it follows that every perfect object is compact. In
    particular, $\Ascr$ is compact in $\QCoh(\Xscr,\Oscr)$ and hence $p_*\Ascr$ is
    compact in $\Mod_R$. Now, we need to see that $p_*\Ascr$ generates
    $\Mod_R$. But, if $M\in\Mod_R$ is such that $\ShMap_R(p_*\Ascr,M)\we 0$, then
    $\ShMap_\Oscr(\Ascr,p^*M)\we 0$ and hence $p^*M\we 0$ (since
    $\Ascr$ is a perfect local generator). But,
    $M\we p_*(p^*M)$ so finally $M\we 0$. Thus, $p_*\Ascr$ is
    a compact generator.
\end{proof}

\begin{corollary}\label{cor:LBrAffine}
    If $(\Xscr,\Oscr)$ is a $0$-affine nonconnective spectral DM stack with $R=\Gamma(\Xscr,\Oscr)$,
    then the isomorphism $\Br(R)\iso\Br(\Xscr,\Oscr)$ restricts to an injection
    $\LBr(R)\subseteq\LBr(\Xscr,\Oscr)$.
\end{corollary}

\begin{remark}
    The proof of \cref{thm:0affine} uses Azumaya algebras and does not
    say anything about cohomological Brauer classes.
\end{remark}

Suppose that $(\Xscr,\Oscr)$ is a quasi-affine nonconnective spectral scheme. Thus,
$(\Xscr,\Oscr)$ is a quasi-compact open inside $\Spet S$ for some commutative ring
spectrum $S$. Let $R=\Gamma(\Xscr,\Oscr)$. Then, $(\Xscr,\Oscr)$ is $0$-affine
by~\cite[Proposition~2.4.1.4]{sag} so we
see that $\Br(\Xscr,\Oscr)\iso\Br(R)$. Moreover, in this case we have
$\Br(\Xscr,\Oscr)\iso\Br'(\Xscr,\Oscr)$ by \cref{thm:spectraldm}, which
applies because $(\Xscr,\Oscr)$ has a finite cover by affine schemes.
In the next example, we use this to completely compute the Brauer group of a
nonconnective $\EE_\infty$-ring.

\begin{example}\label{ex:quasi-affine}
    Let $(\Xscr,\Oscr_0)$ be the classical quasi-affine scheme given by the complement of $0$ inside the affine space 
    $\AA^4_k=\Spec k[x_1,x_2,x_3,x_4]$ where $k$ is some algebraically
    closed field. Let $\Oscr = \Oscr_0[S^1] = \Oscr_0\tensor \Sigma^{\infty}_+S^1$ be the sheaf of $\EE_\infty$-rings on $\Xscr$ given by
    $S^1$-chains on $\Oscr_0$. Thus, $\pi_0\Oscr\iso\Oscr_0$, 
    $\pi_1\Oscr\iso\Oscr_0$, and all other homotopy sheaves vanish. Since $(\Xscr,\Oscr_0)$ is normal,
    $\H^1(\Xscr,\ZZ)=0$. By purity for the Brauer group,
    $\H^2(\Xscr,\Gm)\iso\H^2(\AA^4_k,\Gm)=0$ \cite{cesnavicius-purity}. Since $\Oscr$ is connective,
    \cite{antieau-gepner} gives that $\Br(\Xscr, \Oscr) \cong \LBr(\Xscr, \Oscr)$ (cf.\ \cref{prop:omni}). Thus, the only
    contribution to $\Br(\Xscr, \Oscr)$ in the descent spectral sequence in
    \cref{constr:spectral-sequence} comes from
    $\H^3(\Xscr,\pi_1\Oscr)\iso\H^3(\Xscr,\Oscr_0)\iso k[x_1^{-1},x_2^{-1},x_3^{-1},x_4^{-1}]\cdot (x_1\cdots x_4)^{-1}$.\footnote{Use
    the fact that $\H^3(\Xscr,\Oscr)$ is isomorphic to the local cohomology module
    $\H^4_{\{0\}}(\AA^4_k,\Oscr)$; see~\cite[Example 7.16]{twenty-four}. Alternatively one can use
    \v{C}ech cohomology as in the computation of the cohomology of $\mathbb{P}^3_k$.}
    All differentials out must vanish and the only thing that can hit this is
    $\H^1(\Xscr,\Gm)=\Pic(\Xscr)$, which vanishes as $\Pic(\AA^4_k) = 0$ and every line bundle extends (cf.\ e.g.\ the argument after (5.6) in \cite{antieau-meier}). Thus, with $R = \Gamma(\Oscr_\Xscr)$ we obtain
    \[\Br(R)\iso\Br(\Xscr,\Oscr)\iso\Br'(\Xscr,\Oscr)\iso k[x_1^{-1},x_2^{-1},x_3^{-1},x_4^{-1}]\cdot (x_1\cdots x_4)^{-1}.\]
    Note that the descent sequence computing $\pi_*R$ degenerates so that
    $$\pi_nR\iso\begin{cases}k[x_1,x_2,x_3,x_4]&\text{if $n=0,1$,}\\k[x_1^{-1},x_2^{-1},x_3^{-1},x_4^{-1}]\cdot (x_1\cdots x_4)^{-1}&\text{if $n=-3,-2$,
    and}\\0&\text{otherwise.}\end{cases}$$ 
    Our computation gives examples of Brauer classes on
    a commutative ring spectrum $R$ which are {\em not} killed by an \'etale cover.
    To see this, pick $\alpha\in\Br(R)$ and suppose that $\alpha$ is killed by
    an \'etale cover $R\rightarrow S$. Then, $S\otimes_R\Oscr$ defines a new
    quasi-affine connective spectral DM stack $(\Yscr,\Oscr)$. (The underlying
    $\infty$-topos is naturally equivalent to $\Xscr\times_{\AA^4_k}\Spec\pi_0S$.)
    By quasi-affineness, it follows that $\alpha$ restricts to $0$ on
    $(\Yscr,\Oscr)$. However, the induced map on the Brauer group is
    $$\Br(\Xscr,\Oscr)\iso\H^3(\Xscr,\Oscr_0)\rightarrow\H^3(\Yscr,\Oscr_0)\iso\Br(\Yscr,\Oscr).$$
    This map is equivalent to $$k[y_1,y_2,y_3,y_4]\rightarrow
    k[y_1,y_2,y_3,y_4]\otimes_{k[x_1,x_2,x_3,x_4]}\pi_0S,$$ which is injective as $R\rightarrow S$ is
    faithfully flat.
    Thus, $\alpha=0$ and so no nonzero class in the Brauer group can be killed by an \'etale cover. 
\end{example}

\section{The Picard sheaf of $\TMF$}\label{sec:PicSheafTMF}
To compute the local Brauer group of $\TMF$, it is necessary to first compute
the Picard sheaf of $\TMF$, which we will attack in this section.
Our key tool is a sheafy version of the
Picard spectral sequence used in \cite{mathew-stojanoska}, which we will
introduce next. 

Let $(\MM, \OO)$ be the derived moduli stack of elliptic curves, where $\OO$ denotes the
Goerss--Hopkins--Miller--Lurie sheaf of $\EE_{\infty}$-ring spectra. By \cref{prop:LongDiff}, the
descent spectral sequence identifies $\pi_0\TMF$ with $\H^0(\MM, \pi_0\OO)$ and the latter one
computes to be $\ZZ[j]$. Thus, the underlying classical morphism of $(\MM, \OO) \to \Spet \TMF$ is
the map $j\colon \MM
\to \AA^1 = \Spec \ZZ[j]$; we will denote $(\MM, \OO) \to
\Spec \TMF$ by $j$ as well

For every \'etale map $f\colon \Spec R \to \AA^1$, we obtain an induced sheaf of $\EE_{\infty}$-ring
spectra $\Oscr_R$ on the base change $\Mscr_R = \Mscr\times_{\AA^1}\Spec R$. Let $\ShPic_{\Oscr_R}$
denote the Picard sheaf corresponding to $\Oscr_R$ on $\Mscr_R$ (with subscript left out if $\Spec R
= \AA^1$). We obtain a Picard spectral sequence $\H^s(\MM_R;
\pi_t\ShPic_{\Oscr_R}) \Rightarrow \pi_{t-s}\ShPic(\MM_R, 
\Oscr_R)$. Sheafification thus yields a spectral sequence
\begin{equation}\label{eq:tmfss}
    \E_2^{s,t}=\R^sj_*\pi_t\Shpic_{\Oscr}\qquad\mathlarger{\Rightarrow}\qquad\pi_{t-s}j_*\Shpic_{\Oscr}\underset{t-s\geq
    0}{\iso}\pi_{t-s}j_*\ShPic_{\Oscr}
\end{equation}
in the abelian category of \'etale sheaves of abelian groups on $\Spet\ZZ[j] =
\AA^1$. We note that $\ShPic(\MM_R, \Oscr_R) \simeq
\ShPic(\TMF_R)$ where $\TMF_R$ is the \'etale extension of $\TMF$ realizing
$f$. Indeed: the natural map 
\[\TMF_R \to \OO_R(\MM_R) \simeq (\OO
\tensor_{\TMF}\TMF_R)(\MM)\] 
is an equivalence since taking global sections and $\Oscr \tensor_{\TMF} -$ are inverse equivalences
between $\QCoh(\Mscr, \Oscr)$ and $\Mod_{\TMF}$ by \cref{cor:Mell0-affine}. Moreover, $(\MM_R,
\Oscr)$ is $0$-affine by \cref{thm:0affine-general} and thus $\QCoh(\Mscr_R, \Oscr_R) \simeq
\Mod_{\TMF_R}$. It follows that
\begin{equation}\label{eq:picpush}
    j_*\ShPic_{\Oscr_\Mscr}\we\ShPic_{\Oscr_\TMF}
\end{equation}
as sheaves of grouplike $\EE_\infty$-spaces on $\Spec\ZZ[j]$.
    
As \eqref{eq:tmfss} arises as the sheafification of Picard spectral sequences,
we can freely apply the tools from \cite{mathew-stojanoska} for Picard spectral
sequences. More precisely, these apply to the comparison to the analogous
spectral sequence $\R^sj_*\pi_t\Oscr \Rightarrow \pi_{t-s}\Oscr_{\Spec \TMF}$.
Viewing a quasi-coherent sheaf on $\Spec \pi_0\TMF \cong \Spec \ZZ[j]$ as a
$\ZZ[j]$-module, this agrees with the usual descent spectral sequence for
computing $\pi_*\TMF$, but remembering the $\ZZ[j]$-module structure. See in particular \cref{comp-tool} for a precise statement we will be using. 

\begin{warning}
    In contrast to the descent spectral sequence for $\pi_*\TMF$, the Picard spectral sequence will
    in general not be $\ZZ[j]$-linear even in the range where its $\E_2$-term agrees with a shift of
    the descent spectral sequence (i.e.\ for $t\geq 2$). We do, however, have $\ZZ[j]$-linearity in
    the range specified by \cref{comp-tool} below. This should be seen in light of (a sheafy
    analogue of) \cite[Corollary 5.2.3]{mathew-stojanoska}.
\end{warning}

\begin{remark}\label{rem:relative-descent}
    Alternatively, the sheafy Picard spectral sequence can be constructed as the relative descent
    spectral sequence for $j_*\Shpic_{\Oscr}$, i.e.\ the spectral sequence associated to applying
    (sheafy) $\pi_*$ to the tower $j_*\tau_{\leq \star}\Shpic_{\Oscr}$. Indeed: the presheaf of
    Picard spectral sequence considered above is obtained by applying presheaf homotopy groups
    $\pi_*^{\mathrm{pre}}$ to the tower $j_*\tau_{\leq \star}\Shpic_{\Oscr}$, and thus its
    sheafification agrees with the relative descent spectral sequence.
\end{remark}

We will not compute the whole spectral sequence \eqref{eq:tmfss}, but obtain the following result
about the $0$-stem, which will be crucial to our results about the local Brauer group.
  
\begin{theorem}\label{thm:PicTMF}
    The spectral sequence~\eqref{eq:tmfss} induces a complete decreasing
    filtration $\F^\star\pi_0j_*\Shpic_{\Oscr_\Mscr}$ on $\pi_0j_*\Shpic_{\Oscr_\Mscr}$ with
    \begin{enumerate}
        \item[{\rm (0)}]   $\gr^0\pi_0j_*\Shpic_{\Oscr_\Mscr}\iso\ZZ/2$,
        \item[{\rm (1)}]   $\gr^1\pi_0j_*\Shpic_{\Oscr_\Mscr}\iso\R^1j_*\Gm$, which sits
            in an exact sequence $$0\rightarrow (i_0)_*\ZZ/3\oplus
            (i_{1728})_*\ZZ/2\rightarrow\R^1j_*\Gm\rightarrow\ZZ/2\rightarrow 0$$
            as established in \cref{lem:R1Gm},
        \item[{\rm (3)}]   $\gr^3\pi_0j_*\Shpic_{\Oscr_\Mscr}\iso k_*v_!\ZZ/2$, where $k$ and $v$ denote
            the inclusions $\Spet\FF_2[j]\hookrightarrow\Spet\ZZ[j]$ and $\Spet \FF_2[j^{\pm1}] \hookrightarrow \Spet \FF_2[j]$,
            respectively,
        \item[{\rm (5)}]   $\gr^5\pi_0j_*\Shpic_{\Oscr_\Mscr}$ is a sum of $b_*\ZZ/3$ and a subsheaf
            of an abelian sheaf $\Ascr$, where $\Ascr$ sits in a non-trivial extension
            $$0\rightarrow\Oscr/(2,j)\rightarrow \Ascr \rightarrow
            a_*\ZZ/2\rightarrow 0,$$
            $a$ is the closed inclusion of $\Spet\FF_2$ into
            $\Spet\ZZ[j]$ at $j=2=0$ and $b$ is the closed inclusion of
            $\Spet\FF_3$ into $\Spet\ZZ[j]$ at $j=3=0$,
        \item[{\rm (7)}]    $\gr^7\pi_0j_*\Shpic_{\Oscr_\Mscr}$ is a subsheaf of $\Oscr/(2,j)$;
    \end{enumerate}
    all other graded pieces vanish.
\end{theorem}

In fact, in the last two items we describe the graded pieces as subsheaves of what we see on the
$\E_6$-page, but there are (at most) $2$ more potential differentials originating from these spots.

The rest of this section will be devoted to the proof of the theorem.
We will use \cref{fig:sheaves} for notation for sheaves on $\Spet\ZZ[j]$
appearing in the spectral sequence~\eqref{eq:tmfss}.
\cref{fig:e3tmfss} on \cpageref{fig:e3tmfss} shows the $\E_2$-page of
the spectral
sequence~\eqref{eq:tmfss}. The general pattern follows from the work of
Mathew--Stojanoska~\cite{mathew-stojanoska} and the computations of the
homotopy groups of $\TMF$, as in Bauer~\cite{bauer}.

To prove \cref{thm:PicTMF}, we show in the next subsection that there
are no contributions in filtration
degrees above $7$. Then, we analyze each remaining filtration in turn.

\begin{table}[h] 
\begin{center}
    \begin{tabular}{|p{1.5cm}|c|c|c|c|c|c|c|c|c|}
\hline 
Symbol
&$\Oscr$
&$\Oscr^\times$
&$\smblkcircle$
&$\circ$
        &$\circledbullet$
        &$\bullseye$
        &$\circledast$
        &$\smblkdiamond$
        &$\smwhtdiamond$\\
\hline 
Sheaf
&Structure sheaf&Units in $\Oscr$&
$
        \Oscr/(2,j)
$         &$
\ZZ/2
$         &
        $\Oscr/2$&
        $i_{1728,*}\ZZ/2$&
        $k_*v_!\ZZ/2$&
        $\Oscr/(3,j)$&
        $i_{0,*}\ZZ/3$\\
\hline 
\end{tabular}
\end{center}
    \caption{An assortment of sheaves on $\Spet\ZZ[j]$.}
\label{fig:sheaves}
\end{table}   

\begin{sseqdata}[name = tmfss, Adams grading, classes = {draw = none } ]
    \class["\mdwhtcircle"](0,0)
    \class["\Oscr^\times"](1,0)
    \class["\mdwhtdiamond"](0,1)
    \class["\bullseye"](0,1)
    \class["\mdwhtcircle"](0,1)
    \structline
    \foreach \x in {1,...,11} {
        \class["?"](-\x,\x)
        \class["?"](-\x,\x+1)
    }

    \foreach \x in {2,...,8} {
        \class["\circledbullet"](\x,\x-1)
    }

    \foreach \x in {-1,...,8} {
        \class["\circledbullet"](\x,\x+7)
    }

    \class["\smblkcircle"](-2,15)
    \class["\smblkcircle"](-2,5)
    \class["\smblkcircle"](-1,6)
    
    \foreach \x in {1,...,8} {
        \class["\smblkcircle"](\x,\x+7)
    }
    
    \class["\smblkcircle"](-1,18)
    \class["\smblkcircle"](0,7)
    
    \class["\smblkcircle"](-1,30)
    
    \foreach \x in {1,...,8} {
        \class["\circledbullet"](\x,\x+3)
        \d3
    }
    
    \class["\circledbullet"](0,3)
    \d[dash dot]3
    \replacesource["\circledast"]

    \foreach \x in {-5,...,8} {
        \class["\circledbullet"](\x,\x+15)
    }

    \foreach \x in {-3,...,8} {
        \class["\circledbullet"](\x,\x+11)
        \d3
    }
    
    \class["\circledbullet"](-4,7)
    \d[dash dot]3
    \replacesource["?"]

    \foreach \x in {2,...,8} {
        \class["\smblkcircle"](\x,\x+3)
        \d3(\x,\x+3,-1,-1)
    }

    \foreach \x in {-5,...,8} {
        \class["\smblkcircle"](\x,\x+15)
    }
        
    \foreach \x in {-2,...,8} {
        \class["\smblkcircle"](\x,\x+11)
        \d3(\x,\x+11,-1,-1)
    }

    \class["\smblkcircle"](4,1)
    \class["\smblkcircle"](4,1)
    \structline(4,1,-1)(4,1,-2)
    \structline[dashed](4,1,-2)(4,3)

    \class["\smblkcircle"](5,6)
    \class["\smblkcircle"](6,7)
    
    \class["\smblkcircle"](0,19)
    
    \class["\smblkcircle"](1,10)
    \d9
    \class["\smblkcircle"](2,11)

    \class["\smblkcircle"](3,6)
    \d5

    \replaceclass["\smblkcircle"](-1,10)

    \class["\smblkcircle"](0,5)
    \class["\smblkcircle"](0,5)
    \structline
    \d[dash dot]5
    \replacesource["a_*\circ"]
    
    \class["\smblkcircle"](-4,9)
    \class["\smblkcircle"](-4,9)
    \structline
    
    \class["\smblkcircle"](-5,34)
    \class["\smblkcircle"](-4,23)
    \d11
    
    \class["\smblkcircle"](-3,14)
    
    \class["\smblkcircle"](-3,28)
    \class["\smblkcircle"](-3,28)
    \structline
    \class["\smblkcircle"](-3,28)
    \structline
    
    \class["\smblkcircle"](-2,15)
    \class["\smblkcircle"](-3,28)
    \class["\smblkcircle"](-2,23)
    \d5
    \class["\smblkcircle"](-3,38)
    \class["\smblkcircle"](-2,29)
    \d9
    
    \class["\smblkcircle"](0,29)
    \class["\smblkcircle"](0,29)
    \structline
    
    \class["\smblkcircle"](0,31)
    
    \class["\smblkcircle"](1,24)
    \class["\smblkcircle"](1,24)
    \structline
    \class["\smblkcircle"](1,24)
    \structline
    \d5(1,24, 1, 1)
    \d5(1,24,2,2)
    \d7
    
    \class["\smblkcircle"](2,11)
    \class["\smblkcircle"](2,19)
    
    \class["\smblkcircle"](1,34)
    \class["\smblkcircle"](2,25)
    \d9
\end{sseqdata}

\begin{figure}[p!]
\printpage[ name = tmfss, grid=chess, page =
3,xscale=1.6,yscale=1.3,x range = {-4}{4},y range = {0}{14},
class pattern = linear, class placement transform = { rotate = 90 }]
    \caption{The $\E_3$-page of the sheafy spectral sequence~\eqref{eq:tmfss} for $\Pic$ of the
    moduli stack. Above the $i+j=1$ diagonal, only $2$-primary torsion information is shown.}
    \label{fig:e3tmfss}
\end{figure}

\begin{sseqdata}[name = tmfss3, Adams grading, classes = {draw = none } ]
    \class["\mdwhtcircle"](0,0)
    \class["\Oscr^\times"](1,0)
    \class["\mdwhtdiamond"](0,1)
    \class["\bullseye"](0,1)
    \class["\mdwhtcircle"](0,1)
    \structline
    \foreach \x in {1,...,11} {
        \class["?"](-\x,\x)
        \class["?"](-\x,\x+1)
    }
    
    \class["\smblkdiamond"](-4,9)
    \class["\smblkdiamond"](-1,14)
    \class["\smblkdiamond"](0,5)
    \d[dash dot]9

    \class["\smblkdiamond"](2,15)
    \class["\smblkdiamond"](3,10)
    \d5
    
    \class["\smblkdiamond"](4,1)
    
    \class["\smblkdiamond"](6,11)
    
    \class["\smblkdiamond"](7,6)
    \d5
\end{sseqdata}

\subsection{High filtrations}

In this section, we use the comparison tool of~\cite{mathew-stojanoska} to
narrow down the possible filtration degrees computing to $\pi_0
j_*\Shpic_{\Oscr_\Mscr}$.

We use the following facts about the large-scale structure of the spectral sequence $$\E_2^{s,t}=\H^s(\Mscr,\pi_t\Oscr_\Mscr)\Rightarrow\pi_{t-s}\TMF,$$
which can be read off from the charts in \cite{bauer} for $\tmf$ or \cite{konter} for $\Tmf$ by
inverting the discriminant modular form $\Delta$ (or rather $\Delta^{24}$ since only this is a permanent cycle).

\begin{proposition}\label{prop:LongDiff}
    \begin{enumerate}
        \item[{\em (1)}] The $\E_\infty$-page of the additive spectral sequence 
        \begin{enumerate}
            \item[{\em (a)}] vanishes in columns $-1$ and $-2$ and
            \item[{\em (b)}] vanishes above row $0$ in column $0$.
        \end{enumerate}
        \item[{\em (2)}] The longest differential in the additive spectral sequence is a $d_{23}$. 
    \end{enumerate}
\end{proposition}
Here, column $n$ always refers to $t-s = n$, i.e.\ to the column if drawn in Adams grading.
We recall the following key tool from \cite{mathew-stojanoska}.

\begin{proposition}[Comparison Tool]\label{comp-tool} 
   For $2\leq r\leq t-1$, $d_r^{s,t}$ in the Pic spectral sequence
    ``is'' $d_r^{s,t-1}$ in the additive spectral sequence.
\end{proposition}

\begin{proof}
    By \cite[Comparison Tool 5.2.4]{mathew-stojanoska}, this is true for each term in the presheaf
    of Picard spectral sequence and is thus also true after sheafification.
\end{proof}

Using these results, we can indeed show that the Picard spectral sequence eventually vanishes in high enough degrees. 

\begin{proposition}\label{prop:thingsvanishhighabove}
    Everything above row $7$ in column $0$ vanishes in the $\E_{\infty}$-page of the Picard spectral
    sequence; likewise above row $30$ in column $-1$. After inverting $2$, the latter vanishing
    holds already above row $14$.
\end{proposition}

\begin{proof}
    The claim about column $0$ follows from the Comparison Tool (\cref{comp-tool}) and the further claim that in the additive spectral sequence 
$$\E_2^{s,t}=\H^s(\Mscr,\pi_t\Oscr_\Mscr)\Rightarrow\pi_{t-s}\TMF$$
every spot in the $(-1)$-column above row $7$ is killed by or supports a $d_r$-differential, which is an
    isomorphism and satisfies $r\leq t$ for $t = x+y$ being the antidiagonal of origin. Indeed, the corresponding
    differential also has to occur in the Picard spectral sequence and the isomorphism of groups
    becomes an isomorphism of quasi-coherent sheaves.

By inspection the further claim is true up to row $23$ (on the $\E_5$-page there is only one class
    in column $-1$ between row $7$ and $24$, namely in row $19$ and this is killed by a $d_9$). As
    noted in \cref{prop:LongDiff}, the longest possible differential is a $d_{23}$ and the
    $\E_{\infty}$-term vanishes; thus everything above row $23$ is killed by or supports a
    differential, which is at most a $d_{23}$. Moreover, by inspection, nothing in the additive
    spectral sequence in column $0$ below row $23$ supports a differential killing a class above row
    $23$ in column $-1$.

The proof for column $-1$ of the Picard spectral sequence is analogous. 
\end{proof}

The \cref{prop:thingsvanishhighabove} implies that to prove \cref{thm:PicTMF} it is enough to analyze $\gr^n\pi_0j_*\Shpic_{\Oscr_{\Mscr}}$ for $0\leq n\leq 7$. 

\subsection{Row 0}

Since the geometric fibers of $j\colon\Mscr\rightarrow\Spet\ZZ[j]$ are
connected and $\pi_0\Shpic_{\Oscr_\Mscr}\iso\ZZ/2$, we have
$\R^0j_*\ZZ/2\iso\ZZ/2$. This term does not support any differentials since
$\TMF[1]$ is a global section of the Picard sheaf which restricts to a
generator of $\ZZ/2$ everywhere; this proves part (0) of
\cref{thm:PicTMF}.

\subsection{Row 1 and the algebraic Picard sheaf}

The next term to understand is $\R^1j_*\Gm$, which appears on the $\E_2$-page at
$(s,t)=(1,1)$. This calculation is done on the classical moduli stack.
The sheaf $\R^1j_*\Gm$ is the sheafification of the presheaf that sends every
\'etale $U \to \Spec \ZZ[j]$ to $\Pic(\MM\times_{\Spec \ZZ[j]} U)$. Thus our
next lemma can be seen as a sheafy analogue of the classical computation that
$\Pic(\MM) \cong \ZZ/12$ (see~\cite{fulton-olsson}), where a generator is given
by the Hodge bundle $\lambda$ that arises as the pushforward of the sheaf of
differentials of the universal elliptic curve. We will indeed use the stronger
result from \cite{fulton-olsson} that the same isomorphism holds over any
reduced and normal base ring $R$ with vanishing Picard group. Moreover, we will use that
$\Pic(\AA^1_R) \cong \Pic(R)$ for any regular noetherian $R$, where $\AA^1_R =\AA^1\times \Spec R$;
this follows e.g.\ from the $\AA^1$-invariance of the divisor class group as in \cite[Proposition
6.6, Corollary 6.16]{hartshorne}.

\begin{proposition}
    \label{lem:R1Gm}
    Denote by $i_t\colon \Spec \ZZ \to \Spec \Z[j]$ the inclusion corresponding to the value $t$ of the
    function $j$ on $\Spec \Z[j]$ and by $u_t$ the inclusion of its complement.

    The morphism $\Pic(\MM) \cong \ZZ/12 \to \R^1q_*\GG_m$ is surjective with kernel $(u_0)_!\ZZ/3
    \oplus (u_{1728})_!\ZZ/2$. Thus $\R^1j_*\GG_m$ sits in the extension
    \[0 \to (i_0)_*\Z/3\oplus (i_{1728})_*\Z/2 \to\R^1j_*\GG_m \to \Z/2 \to 0 \]
    that is pushed forward from the non-trivial extension of $\ZZ/3\oplus \ZZ/2$
    and $\ZZ/2$ of constant sheaves along the unit map $\ZZ/3\oplus \ZZ/2 \to (i_0)_*\ZZ/3\oplus (i_{1728})_*\ZZ/2$. 
\end{proposition} 

\begin{proof}
    We will explain first why it suffices to show the surjectivity of $\ZZ/12 \to \R^1q_*\GG_m$ and identify its kernel. Note that there is an exact sequence
    \[0 \to (u_t)_!u_t^*\Fscr \to \Fscr \to (i_t)_*i_t^*\Fscr \to 0\]
    for any $t$ and any \'etale sheaf $\Fscr$. Thus we obtain the claimed
    extension from the proposition by quotienting the first two terms of the exact sequence
    \[0 \to \ZZ/3\oplus \ZZ/2 \to \ZZ/12 \to \ZZ/2 \to 0\]
    by $(u_0)_!\ZZ/3 \oplus (u_{1728})_!\ZZ/2$ and using the snake lemma if indeed
    \[0\to (u_0)_!\ZZ/3 \oplus (u_{1728})_!\ZZ/2 \to \ZZ/12 \to \R^1j_*\GG_m \to 0\]
    is exact. This exactness can be checked on the level of stalks, which is the content of the rest of the argument. 

    Let $\overline{x}\colon \Spec k \to \AA^1$ be a geometric point (corresponding to some $j\in k$)
    and $x\colon \Spec k \to \MM$ its unique lift. We will show that $\ZZ/12 \to
    (\R^1j_*\GG_m)_{\overline{x}}$ is surjective with the prescribed kernel.     
    One can deduce from  \cite[Lemma 2.2.3]{abramovich-vistoli} that the
    base change of $\MM$ to the \'etale stalk of $\AA^1$ at $\overline{x}$ is equivalent to the quotient stack
    $[\Spec S/\Aut(x)]$, where $S$ is strictly Henselian with
    residue field $k$ and $\Aut(x)$ is acting trivially on $k$ (cf.\ \cite[Proposition 8]{meier}).
    We can compute the stalk $(\R^1j_*\GG_m)_{\overline{x}}$ as $\Pic([\Spec S/\Aut(x)]) \cong
    \H^1(\Aut(x); \GG_m(S))$. For the values of $\Aut(x)$ we refer to \cite[Section
    III.10]{silverman}. We will proceed with a case distinction based on $j$ and the characteristic
    of $k$. \\
    
    \emph{Case 1: $j\neq 0, 1728$:}
    
    If $j$ is neither $0$ nor $1728$ in $k$, we have $\Aut(x) \cong C_2$, generated by $[-1]$. As
    the $[-1]$-automorphism is defined on all elliptic curves, $[\Spec S/\Aut(x)] \simeq \B C_{2,
    S}$, i.e.\ the $C_2$-action on $S$ is trivial (cf.\ \cite[Lemma 3.2]{shin-brauer}). Hence
    $(\R^1j_*\GG_m)_{\overline{x}} \cong \H^1(C_2;\GG_m(S)) \cong \mu_2(S) = \ZZ/2$.\\
    
    \emph{Case 2: $\mathrm{char}(k) \neq 2,3$ and $j = 0$ or $1728$:}

    In general, if $k$ is of characteristic $p \geq 0$, the group $\GG_m(S)[\frac1p]$ is divisible
    (where $[\frac10]$ is understood not to have any effect). Using the structure theory of
    divisible abelian groups and \cite[Tag 06RR]{stacks-project}, one can show that $\GG_m(S)[\frac1p]$
    decomposes into a $\QQ$-vector space and a torsion group, which maps isomorphically to
    $\GG_m(k)[\frac1p] \cong \QQ/\ZZ[\frac1p]$ (cf.\ the proof of \cite[Lemma 9]{meier}). We obtain
    \[\H^1(\Aut(x); \GG_m(S))[\tfrac1p] \cong  \H^1(\Aut(x); \QQ/\ZZ[\tfrac1p]) \cong \Hom(\Aut(x); \QQ/\ZZ[\tfrac1p]).\]
    If $k$ is of characteristic not $2$ or $3$, we have $\Aut(x)\cong \ZZ/4$ if
    $j=1728$ and $\Aut(x)\cong \ZZ/6$ if $j=0$, which implies that the
    corresponding stalks of $\R^1j_*\GG_m$ are the Pontryagin duals of $\ZZ/4$ and $\ZZ/6$, i.e.\ isomorphic to $\ZZ/4$ and $\ZZ/6$ as well.
    (Note that in these cases $\H^1(\Aut(x); \GG_m(S))$ is $12$-torsion, so inverting $p$ changes nothing.)
    
    Concretely, the map $\ZZ/12 \cong \Pic(\Mscr) \to \Hom(\Aut(x); \QQ/\ZZ[\tfrac1p])$ sends a line
    bundle $\Lscr$ to the action of $\Aut(x)$ on $\Lscr_x$ by the roots of unity $\QQ/\ZZ[\tfrac1p]
    \cong \mu_{\infty} \subset \GG_m(k)$. By the proof of \cite[Theorem III.10.1]{silverman}, in our
    case a generator of $\Aut(x)$ acts by a fourth respectively a sixth root of unity on the
    invariant differential and thus on $\lambda_x$ (for $\lambda$ the standard generator of
    $\Pic(\Mscr)$ as above). Thus summarizing, we see that the map $\ZZ/12 \to
    (\R^1j_*\GG_m)_{\overline{x}}$ is surjective with the prescribed kernel
    unless $\mathrm{char}(k) = 2,3$ and $\overline{x}$ corresponds to $j=0\equiv 1728$.
    In particular, we see that $\varphi\colon \ZZ/12 \to \R^1q_*\GG_m$ factors through $\Fscr = (\ZZ/12)/(u_0)_!\ZZ/3 \oplus (u_{1728})_!\ZZ/2$.  \\
    
    \emph{Case 3: $\mathrm{char}(k) = 2$ or $3$ and $j = 0 =1728$:}
    
    From now on let $\overline{x}\colon k \to \mathbb{A}^1$ be a geometric point with
    $\mathrm{char}(k) = p$ for $p=2,3$ corresponding to $j=0$. For a base ring $R$, denote by
    $\MM_R$ the base change $\MM\times \Spec R$. We will show that
    $\varphi_{\overline{x}}\colon\ZZ/12 \to (\R^1j_*\GG_m)_{\overline{x}}$ is an isomorphism by
    comparison with the known computation of the Picard group of $\Pic(\MM_R)$ for certain $R$. To
    that purpose we will use the Leray spectral sequence
    \[\E_2^{s,t} = \H^s(\AA^1_R; \R^tj_*^R\Gm) \Rightarrow \H^{s+t}(\MM_R; \Gm) \]
    for the map $j^R\colon \MM_R \to \AA^1_R$. Let us display the part relevant for the computation of $\Pic$. 
    
    \begin{figure}[H]
    	\centering
    	\begin{equation*}
    	\xymatrix@R=3pt{
    		\H^0(\AA^1_R; \R^1j^R_*\Gm)\ar[rrd] & \H^1(\AA^1_R; \R^1j^R_*\Gm) &\\
    		\Gm(\AA^1_R) & \Pic(\AA^1_R) & \H^2(\AA^1_R,\Gm) 
    	}
    	\end{equation*}
    \end{figure}
    
    Denoting by $R$ the strict Henselization of the image of $\overline{x}$ in $\Spec \ZZ$, the
    spectral sequence implies that the map $\Phi_R\colon \ZZ/12 \cong \Pic(\MM_R) \to \H^0(\AA^1_R;
    \R^1j^R_*\GG_m)$ is an isomorphism. Here we use that as $\ZZ$ is regular noetherian, $R$ is so
    as well by \cite[Tags 06LJ, 06LN]{stacks-project} and thus $\Pic(\AA^1_R) \cong \Pic(R) = 0$ and

    \[\H^2(\AA^1_R,\Gm) \cong \H^2(R, \Gm) \cong \H^2(\Spec \FF_p, \Gm) =0\]
    by \cref{thm:BrauerProperties} and \cref{thm:rigidity}. The same argument shows that the map
    $\Phi_{\overline{\FF}_p}\colon \ZZ/12 \cong \Pic(\MM_{\overline{\FF}_p}) \to
    \H^0(\AA^1_{\overline{\FF}_p}; \R^1j^{\overline{\FF}_p}_*\GG_m)$ is an isomorphism.
    
    If $\varphi_{\overline{x}}$ were not injective, there would be some $[m]\in\ZZ/12$ (namely $[4]$
    or $[6]$) such that $\varphi([m])$ is zero in every \'etale stalk in characteristic $p$ and
    hence $\Phi_{\overline{\FF}_p}$ would not be injective either. We see that
    $\varphi_{\overline{x}}$ is thus indeed injective and thus $\Fscr$ agrees with the image of
    $\varphi$.
    
    As above we use the notation $\Fscr = (\ZZ/12)/(u_0)_!\ZZ/3 \oplus (u_{1728})_!\ZZ/2$. Denote
    the cokernel of $\ZZ/12 \to \Fscr \to \R^1j_*\GG_m$ by $\Cscr$. Its base change $\Cscr_R$ agrees
    with the cokernel of $\Fscr_R \to \R^1j_*^R\GG_m$ and it suffices to show its vanishing. By the
    arguments of the first paragraph, we have a short exact sequence
    \[0 \to (i_0)_*\Z/3\oplus (i_{1728})_*\Z/2 \to \Fscr \to \Z/2 \to 0. \]
    Moreover, \cref{thm:closedimmersion} implies $\H^1(\AA^1_R; (i_0)_*\Z/3\oplus (i_{1728})_*\Z/2)
    \cong \H^1(R; \Z/3\oplus \Z/2) = 0$ since the \'etale cohomology of anything on $R$ vanishes.
    Thus $\ZZ/12 \to \H^0(\AA^1_R; \Fscr_R)$ is an isomorphism.
    
    Moreover, $\H^1(\AA^1_R; \ZZ/2) \cong \H^1(\AA^1_R; \mu_2)$ sits in a short exact sequence with
    $\H^1(\AA^1_R; \GG_m)/2 =0$ and $\H^2(\AA^1_R; \GG_m)[2] = 0$ and thus has to vanish as well. We
    conclude that $\H^1(\AA^1_R; \Fscr_R) = 0$. Summarizing, we have an exact sequence
    \[ 0 \to \H^0(\AA^1_R; \Fscr_R) \to \H^0(\AA^1_R; \R^1j_*^R\GG_m) \to \H^0(\AA^1_R; \Cscr_R) \to 0 = \H^1(\AA^1_R, \Fscr_R).\]
    We have seen above that the natural morphisms from $\ZZ/12$ to the first two non-trivial groups
    are isomorphisms and thus the map between them is an isomorphism. Thus $\H^0(\AA^1_R; \Cscr_R)$
    vanishes. As $\Cscr_R$ is supported at $\overline{x}$, we see that its stalk at $\overline{x}$
    vanishes and thus that $\varphi_{\overline{x}}$ is also surjective.
\end{proof}

We claim that there are no differentials out of $\E_2^{1,1} = \R^1j_*\Gm$. Indeed, by the preceding proposition
\[\Pic(\Mscr,\pi_0\Oscr_\Mscr)\iso\ZZ/12\rightarrow\R^1j_*\Gm\]
is a
surjective map of sheaves on $\AA^1 = \Spet\ZZ[j]$. Thus, as long as the classes of
$\ZZ/12$ lift to invertible sheaves on the derived moduli stack,
surjectivity of the map means there cannot be differentials. But, this
$\ZZ/12$ is generated by $\TMF[2]$, which gives part (1) of
\cref{thm:PicTMF}.

\subsection{$3$-torsion in Row 5}
    
For higher filtrations it is necessary to compute differentials. Differentials $d_r^{t,s}$
in the sheafified $\Pic$ spectral sequence where
$r\leq t-1$ (i.e., where the ``length'' of the differential is smaller than the coordinate $t= x+y$ of the antidiagonal of origin) can be directly read off the descent spectral sequence computing $\pi_*\TMF$
by \cref{comp-tool}. We will use this fact without further comment.

For the rest of the analysis, we will work separately with $2$ and $3$ inverted, to analyze the $3$
and $2$-torsion, respectively. The only possible contribution to
$3$-torsion in $\pi_0j_*\Shpic_{\Oscr_\Mscr}$ is the kernel of the $d_9$ differential on
$\E^{5,5}_9$. We will implicitly invert $2$ throughout this section.

\begin{lemma}\label{lem:3d9}
    The differential
    $$d_9\colon\R^5j_*\pi_5\Shpic_{\Oscr_\Mscr}\iso\R^5j_*\pi_4\Oscr_\Mscr\iso\Oscr/(3,j)\rightarrow\R^{14}j_*\pi_{13}\Shpic_{\Oscr_\Mscr}\iso\R^{14}j_*\pi_{12}\Oscr_\Mscr\iso\Oscr/(3,j).$$
    is surjective and the kernel is
    $b_*\ZZ/3$, where $b$ is the closed inclusion of
    $\Spet\FF_3$ into $\Spet\ZZ[j]$ at $j=3=0$.
\end{lemma}

\begin{proof}
    The differential is the first possible outside of the exponentiable range from \cref{comp-tool}.
    In the descent spectral sequence for $\pi_*\TMF$ the corresponding differential is an isomorphism
    (see~\cite[(8-4)]{mathew-stojanoska}). In contrast,~\cite[Theorem~7.1]{heard-mathew-stojanoska} implies that we can write the
    differential $d_9\colon \Oscr/(3,j)\rightarrow\Oscr/(3,j)$ in the Picard spectral sequence as $$x\mapsto
    x+\zeta\beta\Pscr^2(x),$$ where $\zeta$ is a unit in $\FF_3$ and
    $\beta\Pscr^2$ is certain power operation on $\EE_\infty$-rings in which
    $2$ is invertible. Moreover, $x\mapsto\zeta\beta\Pscr^2(x)$ is
    Frobenius-semilinear in the sense that
    $zx\mapsto\zeta\beta\Pscr^2(zx)=z^3\zeta\beta\Pscr^2(x)$.

    We know that our $d_9$ must be zero on global sections.by ~\cite[Sec.~8.1]{mathew-stojanoska}
    (as else $\Pic(\TMF)_{(3)}$ could have at most $3$ elements). Thus, $1 +
    \zeta\beta\Pscr^2(1)=d_9(1)$ is zero on global sections and thus also everywhere, and hence
    $\zeta\beta\Pscr^2(1) = -1$. By Frobenius-semilinearity, we see that $d_9(z) = z +
    \zeta\beta\Pscr^2(z\cdot 1) = z-z^3$.
    It follows from Artin--Schreier theory that this differential is a
    surjective map of \'etale sheaves and that the kernel is $b_*\ZZ/3$ (cf.\ \cite[Example 2.18c]{milne-etale}).   
\end{proof}

As \cref{fig:e5tmfss3} on \cpageref{fig:e5tmfss3} proves, the lemma shows that Part (5) of \cref{thm:PicTMF} holds with $2$ inverted and all the other graded pieces
vanish with $2$ inverted. Thus it remains to analyze the $2$-torsion and we will implicitly \emph{work $2$-locally everywhere}. We will compute two further differentials
affecting the zeroth column of the Picard spectral sequence and then give an outlook on what remains
to be done to compute all remaining differentials.

\begin{figure}[!p]
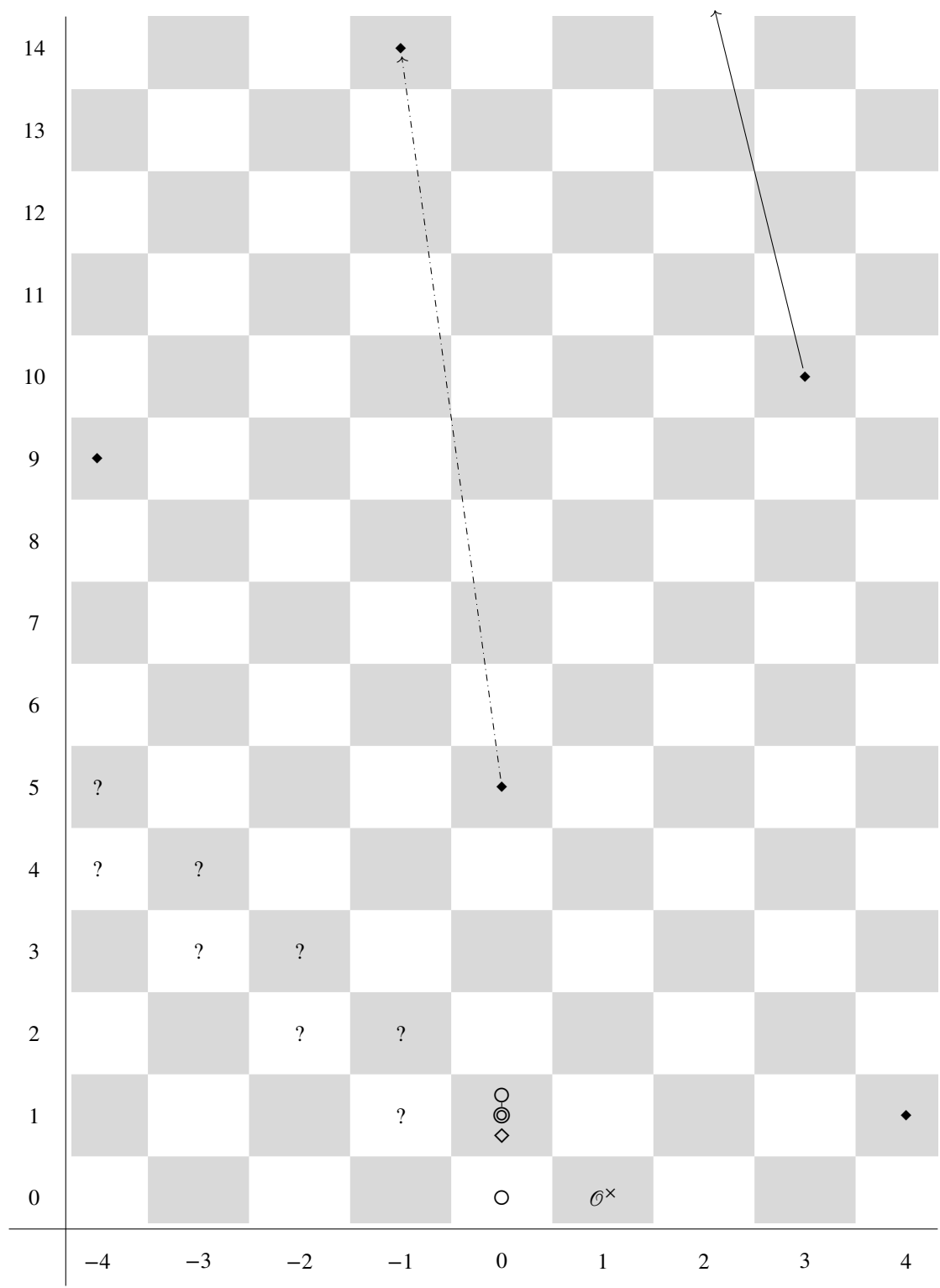

\printpage[ name = tmfss3, grid=chess, page =
0,xscale=1.6,yscale=1.3,x range = {-4}{4},y range = {0}{14},
class pattern = linear, class placement transform = { rotate = 90 }]
\caption{The $\E_5$-page of the sheafy spectral sequence~\eqref{eq:tmfss} for $\Pic$ of the moduli stack. Above the $i+j=1$ diagonal, only $3$-primary torsion information is shown.}
\label{fig:e5tmfss3}
\end{figure}

\subsection{Row 3}\label{sec:Row3}

There is a $d_3$-differential,
$$d_3\colon\R^3j_*\pi_3\Shpic_{\Oscr_\Mscr}\iso\R^3j_*\pi_2\Oscr_\Mscr\iso\Oscr/2\rightarrow\R^6j_*\pi_5\Shpic_{\Oscr_\Mscr}\iso\R^6j_*\pi_4\Oscr_\Mscr\iso\Oscr/2.$$
This differential is of the form $$x\mapsto
x+jx^2=x(1+jx),$$ as shown in~\cite[Sec.~8.2]{mathew-stojanoska}.
Recall from \cite[Corollary II.3.11]{milne-etale} that $k_*$ for $k\colon \Spec \FF_2[j] \to \Spec \ZZ[j]$ induces an exact equivalence between \'etale sheaves on
$\FF_2[j]$ and \'etale sheaves on $\ZZ[j]$ supported at the prime $(2)$; hence, we can work directly on the \'etale site of $\FF_2[j]$.
Under this equivalence we have $k_*\Oscr\iso\Oscr/2$.
We claim that the $d_3\colon \Oscr \to \Oscr$ is surjective (viewed as \'etale sheaves on
$\Spec\FF_2[j]$). Indeed, given any \'etale morphism
$\FF_2[j]\rightarrow R$ and element $c\in R$,
the extension $R \to R[x]/(jx^2+x-c)$ is \'etale and surjective on geometric points:
base-changing along any morphism $R\to K$ to a field, $j$ becomes either invertible or zero and in
either case $K[x]/(jx^2+x-c)$ is nonzero. Thus, $\Spec R[x]/(jx^2+x-c)\to \Spec R$ is an \'etale
cover and $c$ has per construction a preimage under $d_3$ on $R[x]/(jx^2+x-c)$.

Any nonzero element in the stalk of the kernel must be of the form $x=\tfrac{1}{j}$
(since all stalks of $\Oscr$ on $\FF_2[j]$ are integral domains);
thus the kernel corresponds to the \'etale sheaf $v_!\ZZ/2$ on $\FF_2[j]$
(with $v\colon \Spec \FF_2[j^{\pm 1}] \to \Spec \FF_2[j]$ being the inclusion) and is $k_*v_!\ZZ/2$ as an \'etale sheaf on $\ZZ[j]$. 
There will be no further differentials from this spot because all possible further targets are supported at $(2,j)$. 

\begin{figure}[!p]
\printpage[ name = tmfss, grid=chess, page =
5,xscale=1.6,yscale=1.3,x range = {-4}{4},y range = {0}{14},
class pattern = linear, class placement transform = { rotate = 90 }]
\caption{The $\E_5$-page of the sheafy spectral sequence~\eqref{eq:tmfss} for $\Pic$ of the moduli stack. Above the $i+j=1$ diagonal, only $2$-primary torsion information is shown.}
\label{fig:e5tmfss}
\end{figure}

\subsection{$2$-torsion in Row 5}

The next differential is the $d_5$-differential, 
$$d_{5}\colon\R^5j_*\pi_5\Shpic_{\Oscr_\Mscr}\iso\R^5j_*\pi_4\Oscr_\Mscr\iso\Oscr/(4,j)\rightarrow\R^{10}j_*\pi_{9}\Shpic_{\Oscr_\Mscr}\iso\R^{10}j_*\pi_{8}\Oscr_\Mscr\iso\Oscr/(2,j),$$
which factors through a map $\Oscr/(2,j)\rightarrow\Oscr/(2,j)$. This differential
is just outside of the exponentiable range and is given
by~\cite[Thm.~6.1.1]{mathew-stojanoska}. The map
$\Oscr/(2,j)\rightarrow\Oscr/(2,j)$ is given by $x\mapsto x+x^2$. This is a
surjective map of \'etale sheaves by Artin--Schreier theory: given
$y\in\Oscr/2$, the extension defined by $y=x+x^2$ is \'etale. The kernel is
$a_*\ZZ/2$, where $a\colon\Spet\FF_2\rightarrow\Spet\ZZ[j]$ is the
inclusion at $2=j=0$.

\subsection{Long differentials}\label{sec:LongDiff}

As already established in \cref{prop:thingsvanishhighabove}, in column 0 everything above row 7 must
be zero on the $\E_{\infty}$-page.  
As \cref{fig:e7tmfss} on \cpageref{fig:e7tmfss} and the preceding discussion shows, the only
remaining possible differentials are a $d_{13}$ and $d_{25}$ originating in row 5 and a $d_{11}$ and
a $d_{23}$ originating in row 7. We can show the vanishing of one of these differentials.

\begin{figure}[H]
    \printpage[ name = tmfss, grid=chess, page =
    7--25,xscale=1.6,yscale=0.6858,x range = {-4}{4},y range = {0}{30},
    class pattern = linear, class placement transform = { rotate = 90 }]
    \caption{The $\E_7$-page of the sheafy spectral sequence~\eqref{eq:tmfss} for $\Pic$ of the moduli stack. Above the $i+j=1$ diagonal, only $2$-primary torsion information is shown.}
    \label{fig:e7tmfss}
\end{figure}

\begin{lemma}
    The differential $$d_{11}\colon\E_{4}^{7,7}\iso\E_{11}^{7,7}\iso\Oscr/(2,j)\rightarrow\E_{11}^{18,17}\iso\Oscr/(2,j).$$ is zero.
\end{lemma}

\begin{proof}
    We use
    the sequence of maps of derived stacks
    $$\widehat{\Mscr}(3)^{\mathrm{ss}}\rightarrow\Mscr(3)\rightarrow\Mscr[\tfrac{1}{3}],$$
    where $\Mscr(3)$ is the moduli stack of elliptic curves with full level-3-structure (see e.g.\ \cite{stojanoska}) and $\widehat{\Mscr}(3)^{\mathrm{ss}}$ is
    the completion of $\Mscr(3)$ at the ideal $(2,j)$. We write $j$ for the map
    from any of these stacks to $\Spet\ZZ[j]$. Pushing the Picard sheaves down to
    $\Spet\ZZ[j]$ we obtain $$j_*\Shpic_{\Oscr_{\Mscr}}\rightarrow
    j_*\Shpic_{\Oscr_{\Mscr[1/3]}}\rightarrow
    j_*\Shpic_{\Oscr_{\widehat{\Mscr}(3)^{\mathrm{ss}}}}.$$
    Since $\Mscr(3)\rightarrow\Mscr[\tfrac{1}{3}]$ is a Galois cover with group
    $\GL_2(\ZZ/3)$,
    the first of these maps induces an equivalence
    \begin{equation}\label{eq:gl2}
        \tau_{\geq 0}j_*\Shpic_{\Oscr_{\Mscr[1/3]}}\rightarrow\tau_{\geq
    0}\left(j_*\Shpic_{\Oscr_{\Mscr(3)}}^{h\GL_2(\ZZ/3)}\right);
    \end{equation}
    this kind of Galois descent follows from the sheaf property of $\Shpic$. 
    The $2$-local Picard spectral sequence for $\TMF$ is for $t>1$ the same as
    the $\GL_2(\ZZ/3)$-based relative descent spectral sequence for~\eqref{eq:gl2} (cf.\ \cref{rem:relative-descent}).

    Completing $\MM(3)$ at the ideal $(2,j) \subset \ZZ[j]$ results in the formal deformation space
    of a supersingular elliptic curve $C$ over $\FF_4$, which can be coordinatized as $\Spet
    W(\FF_4)\llbracket u\rrbracket$. Thus completing $\MM$ itself at $(2,j)$ becomes identified with
    the stack quotient of $\Spet W(\FF_4)\llbracket u\rrbracket$ by $\GL_2(\ZZ/3)$. As source and
    target of the $d_{11}$-differential we care about are supported at $(2,j)$, completion at
    $(2,j)$ does not lose information; more precisely, the composition of $(2,j)$-completion and
    pushing forward along $\Spf \ZZ_2\llbracket j\rrbracket \to \Spet \ZZ[j]$ is an isomorphism at
    the relevant spots in the sheafy Picard spectral sequence.
    Note that the \'etale topos of $\Spf \ZZ_2\llbracket j\rrbracket$ agrees with that of $\FF_2$.
    Thus the $(2,j)$-completed sheafy Picard spectral sequence is the sheafification of the
    collection of Picard spectral sequences of the higher real K-theories, assigning
    $E_2(\widehat{C}, \FF_4\tensor_{\FF_2}k)^{h\GL_2(\ZZ/3)}$ to each finite extension $\FF_2\subset
    k$.

    Denote the corresponding spectral sequence for $E_2(\widehat{C}, \FF_4\tensor_{\FF_2}k)^{hH}$ with
    $H\subset \GL_2(\FF_3)$ by $\E_{*, k, H}^{**}$, where we consider $k=\FF_{2^n}$ or
    $\overline{\FF}_2$. We will deduce from \cite{bobkova-goerss} that the restriction map
    $\res_{C_4}^{G_{48}}\colon \E_{2, \FF_2, \GL_2(\ZZ/3)}^{s,t} \to \E_{2, \FF_2, C_4}^{s,t}$ is zero
    for $(s,t) = (7,7)$ and an isomorphism for $(18,17)$ on the $\E_{4}$-page (which agrees in this
    range in descent and Picard spectral sequence). Hence the same is true for all $k$ in place of
    $\FF_2$ (since the restriction maps are module maps). Moreover, we have established above that
    in the Picard spectral sequence $\E_4 = \E_{11}$ in these spots. As $d_{11}\res_{C_4}^{G_{48}} =
    \res_{C_4}^{G_{48}} d_{11}$, this implies the vanishing of our $d_{11}$.

    To read off our claim about the restriction map from \cite{bobkova-goerss}, we will use their
    notation (noting that they write $G_{48}$ for  $\GL_2(\ZZ/3)$). Section 2.3 of op.cit.\ implies
    that the generator of $\E_{2, \FF_2, \GL_2(\ZZ/3)}^{7,7}$ is
    $\Delta^{-1}\eta^3\overline{\kappa}$. Since $\Delta$ is a $d_3$-cycle, it will suffice for our vanishing claim about the restriction to show
    that $\res_{C_4}^{G_{48}}(\eta^3\overline{\kappa})$ is hit by a $d_3$. This restriction equals
    $\eta^3\delta\xi^2$ by the Table above Section 2.3 of op.cit. We have the differential $d_3(\xi) =
    \delta^{-1}\eta\xi^2$ by \cite[Proposition 2.3.1]{behrens-ormsby},\footnote{As in \cite{bobkova-goerss}, we will use \cite{behrens-ormsby} for information about differentials on the $C_4$-level, using that $\TMF_1(5)$ becomes a Lubin--Tate theory after $K(2)$-localization.} so
    $d_3(\delta^2\eta^2\xi) =\delta\eta^3\xi^2 =\res_{C_4}^{G_{48}}(\eta^3\overline{\kappa})$. Now
    we turn to the generator of $\E_{4, \FF_2, \GL_2(\ZZ/3)}^{18,17}$, which is
    $\Delta^{-4}\kappa\overline{\kappa}^4$. Using \cite[Lemma 2.2.4, Corollary
    2.3.2]{behrens-ormsby}, $\res_{C_4}^{G_{48}}(\Delta)$ acts like $\delta^3$ on a torsion
    class like $\res_{C_4}^{G_{48}}(\kappa\overline{\kappa}^4) = \delta^5\nu^2\xi^8$. Thus,
    $\Delta^{-4}\kappa\overline{\kappa}^4$ restricts to $\delta^{17}\nu^2\xi^8$. This can be used to show that the restriction is an isomorphism on
    $\E_4^{18,17}$. Moreover, the class in $\E_{4, \FF_2, \GL_2(\ZZ/3)}^{18,17}$ is only hit by a
    $d_{15}$ and its restriction by a $d_{13}$ (namely from $(\delta\nu^2)\delta^{20} (\delta \nu
    \xi)$; cf.\ \cite[Proposition 2.3.9]{behrens-ormsby}). In particular, in these spots the
    $\E_4$-page equals the $\E_{11}$-page.
\end{proof}

We included this lemma not primarily for its intrinsic importance, but rather to demonstrate that
the remaining computational mysteries of the sheafy Picard spectral sequence are purely $K(2)$-local
phenomena and might thus potentially be resolved purely in the setting of Lubin--Tate spectra.

\section{Applications to Picard groups}

We can use \cref{thm:PicTMF} to compute Picard groups of various spectra related to $\TMF$. 

\begin{example}
    Using \cref{thm:PicTMF}, we want to compute $\Pic(\TMF[c_4^{-1}])$. Noting $\pi_0\TMF[c_4^{-1}]
    \cong \ZZ[j^{\pm 1}]$, the relevant part of the exact sequence from \cref{prop:omni} is
    \[0 \to \H^1(\Spec \Z[j^{\pm 1}]; \GG_m) \to \Pic(\TMF[c_4^{-1}]) \to \H^0(\Spec \ZZ[j^{\pm 1}],
    \pi_0 \sPic{\TMF}) \to \H^2(\Spec \Z[j^{\pm 1}]; \GG_m) \to \cdots .\]
    The groups \[\H^1(\Spec \ZZ[j^{\pm 1}]; \Gm) = \Pic(\Spec \ZZ[j^{\pm 1}]) \subset \Pic(\Spec
    \ZZ[j]) \cong \Pic(\ZZ)\] and \[\H^2(\Spec \ZZ[j^{\pm 1}]; \Gm) = \Br(\Spec \ZZ[j^{\pm 1}])\]
    vanish since $\ZZ$ is a PID and $\Pic$ is $\AA^1$-invariant by \cite[Proposition II.6.6]{hartshorne}, and by \cref{cor:Brauerjpm}.
    Thus \[\Pic(\TMF[c_4^{-1}]) \to \H^0(\Spec \ZZ[j^{\pm 1}]; \pi_0 \sPic{\TMF})\] is an isomorphism.
    By \cref{thm:PicTMF}, the restriction of $\pi_0 \sPic{\TMF}$ to $\Spec \ZZ[j^{\pm 1}]$ has a
    filtration with associated graded $\ZZ/2$, $\ZZ/2$, $(i_{1728})_*\ZZ/2$ and $k_*\ZZ/2$, where
    $k\colon \Spec \FF_2[j^{\pm 1}] \to \Spec \ZZ[j^{\pm 1}]$ is the inclusion. We obtain directly
    that $\Pic(\TMF[c_4^{-1}])$ is $2$-power torsion.

    Let $\Qscr$ be the quotient of $(\pi_0 \sPic{\TMF})_{(2)}$ by everything of filtration at least $2$. We obtain a short exact sequence 
    \begin{equation}\label{eq:sesQ}0\to k_*\ZZ/2 \to u^*\pi_0 \sPic{\TMF} \to u^*\Qscr \to 0,\end{equation}
    where $u\colon \Spec \ZZ[j^{\pm 1}] \to \Spec \ZZ[j]$ is the inclusion. Applying the long
    exact sequence in cohomology to this short exact sequence and to the analogous one for $\TMF$,
    we obtain a diagram
    \[
    \xymatrix{
    0 \ar[r] & \ZZ/8\ar[d] \ar[r] & \Pic(\TMF)_{(2)}\ar[d] \ar[r] & \H^0(\ZZ[j]; \Qscr)\ar[d] \\
    0 \ar[r] & \ZZ/2 \ar[r] & \Pic(\TMF[c_4^{-1}])\ar[r] & \H^0(\ZZ[j^{\pm1}]; u^*\Qscr).
    }
    \]
    Investigating the associated graded pieces of the Picard sheaves, we see that the first vertical map
    is zero. The rightmost vertical map is an injection by the four-lemma since it is an
    isomorphism on global sections of graded pieces. Looking at the global sections of the graded
    pieces, we also obtain that source and target of this map have at most $8$ elements. Since
    $\Pic(\TMF)_{(2)} \cong \ZZ/64$, we see that $\H^0(\ZZ[j]; \Qscr)$ must be $\ZZ/8$ and thus the
    same is true for $\H^0(\ZZ[j^{\pm1}]; u^*\Qscr)$. As $\TMF[1]$ is sent to a generator of
    $\H^0(\ZZ[j]; \Qscr)$, we see that $\TMF[c_4^{-1}][1]$ is sent to a generator of
    $\H^0(\ZZ[j^{\pm1}]; u^*\Qscr)$. Moreover, $\TMF[c_4^{-1}][1]$ generates a group of order $8$
    inside $\Pic(\TMF[c_4^{-1}])$. Thus, the lower exact is sequence is split short exact and
    $\Pic(\TMF[c_4^{-1}]) \cong \ZZ/2 \oplus \ZZ/8$. The extra $\ZZ/2$ provides an example of an
    exotic Picard group element.
\end{example}

The following proposition is a higher (but less explicit) analogue of \cref{cor:PicKOR}. 

\begin{proposition}
    Denote by $\Qscr$ the quotient of $\pi_0\sPic{\TMF}$ corresponding to (0) and (1) in
    \cref{thm:PicTMF} and by $\Iscr$ the kernel of the quotient map. Let further $R$ be an \'etale
    extension of $\ZZ$. Then there is a short exact sequence
    \[0 \to \Pic(R) \to \Pic(\TMF_R) \to \H^0(\AA^1_R; \pi_0\sPic{\TMF}) \to 0.\]
    If $\Spec R$ is connected, the last term fits into a short exact sequence
    \[0 \to \H^0(\AA^1_R; \Iscr) \to \H^0(\AA^1_R; \pi_0\sPic{\TMF}) \to \ZZ/24 \to 0.\]
\end{proposition}

\begin{proof}
    We begin by proving the second claim. Applying the long exact sequence of cohomology to the extension
    \[0\to\Iscr\to\pi_0\sPic{\TMF} \to \Qscr \to 0,\] we 
    obtain an exact sequence
    \begin{equation}\label{eq:exseqIPQTMF}0 \to \H^0(\AA^1_R; \Iscr) \to \H^0(\AA^1_R; \pi_0\sPic{\TMF}) \to \H^0(\AA^1_R, \Qscr).\end{equation}
    The composition 
    \begin{align*} \ZZ &\to \Pic(\TMF_R) \to \H^0(\AA^1_R; \pi_0\sPic{\TMF}) \to \H^0(\AA^1_R; \Qscr)\\
    n&\mapsto \TMF_R[n]\end{align*}
    is a surjection by comparison to $\TMF$ (similar to the preceding example) since $\H^0(\AA^1;
    \Qscr) \to \H^0(\AA^1_R; \Qscr)$ is an isomorphism (using a comparison on associated graded
    pieces and
    the five lemma). This implies that $\H^0(\AA^1_R; \Qscr) \cong \ZZ/24$ and that
    \cref{eq:exseqIPQTMF} is short exact. 
    
    For the first claim, we recall from \cref{prop:omni} the exact sequence 
    \[0 \to \Pic(\AA^1_R) \to \Pic(\TMF_R) \to \H^0(\AA^1_R; \pi_0\sPic{\TMF}) \xrightarrow{\partial_R} \Br(\AA^1_R).\]
    Note that $\Pic(\AA^1_R) \cong \Pic(R)$ (e.g.\ by \cite[Proposition II.6.6]{hartshorne}). 
    The arguments proceeds now exactly as in \cref{cor:PicKOR}, using that $\Br(\AA^1_R)$ injects
    into $\Br(\AA^1_{R[\frac16]})$ by \cref{thm:BrauerProperties} and $\Iscr(\AA^1_{R[\frac16]}) = 0$.
\end{proof}

\section[The local Brauer groups of topological modular forms]{The local Brauer groups of $\TMF$ and $(\Mscr, \Oscr)$}\label{sec:mell}

The aim of this section is to show that  local Brauer groups of $\TMF$ and the derived moduli stack
$(\Mscr, \Oscr)$ are infinitely generated and to compute them up to finite ambiguity. First, we
observe the coincidences of various Brauer groups pertinent to this example.

\begin{proposition}\label{prop:BrIdentifications}
    If $(\Mscr,\Oscr)$ is the derived moduli stack of elliptic curves, then
    \begin{enumerate}
        \item[{\rm (i)}] $\Br(\Mscr,\Oscr)\iso\Br'(\Mscr,\Oscr)$,
        \item[{\rm (ii)}] $\Br(\Mscr,\Oscr)\iso\Br(\TMF)$, and
        \item[{\rm (iii)}]
            $\LBr(\Mscr,\Oscr)\iso\LBr'(\Mscr,\Oscr)$.
    \end{enumerate}
\end{proposition}
\begin{proof}
    Parts (i) and (iii) follow from \cref{cor:BrauerCoincidences}. Indeed, we can use the
    cover with opens $\Mscr[\tfrac12] = \Mscr \times \Spec \ZZ[\tfrac12]$ and $\Mscr[\tfrac13] =
    \Mscr \times \Spec \ZZ[\tfrac13]$. Condition (a) of \cref{thm:spectraldm} follows because the
    kernels of
    $\QCoh(\Mscr[\tfrac12])\rightarrow\QCoh(\Mscr[\tfrac16])$ and
    $\QCoh(\Mscr[\tfrac13])\rightarrow\QCoh(\Mscr[\tfrac16])$ are generated by
    the compact objects $\Oscr/3$ and $\Oscr/2$, respectively. Moreover, both $\Mscr[\tfrac12]$ and $\Mscr[\tfrac13]$ admit a
    finite \'etale cover from an affine scheme, for example the moduli stacks
    $\Mscr(4)$ and $\Mscr(3)$ of elliptic curves with full level $4$ and full
    level $3$ structures, respectively. Thus, by (the proof of) \cref{prop:etalecover},
    $\alpha$-twisted sheaves on $\Mscr[\tfrac12]$ and $\Mscr[\tfrac13]$ admit a local perfect
    generator for every Brauer class $\alpha$.
    Condition (b) follows because $\Mscr[\tfrac1p]$ is $0$-affine
    by \cref{thm:0affine-general} and hence a
    local perfect generator is a global generator by the proof of
    \cref{thm:0affine}.

    Part (ii) follows from \cref{thm:0affine} since $(\Mscr,\Oscr)$ is
    $0$-affine by \cref{cor:Mell0-affine}. 
\end{proof}

\begin{theorem}\label{thm:LBrMO1}
    The local Brauer group $\LBr(\Mscr,\Oscr)$ is a torsion group. There is no $p$-torsion for $p>3$. The
    $3$-torsion is $\ZZ/3$. Moreover, there is a surjection $\LBr(\Mscr,\Oscr)_{(2)} \to (\ZZ/2)^{\infty}$
    with kernel of order $8$.
\end{theorem}

\begin{proof}
    By the previous proposition, we can apply the spectral sequence from
    \cref{constr:spectral-sequence} for the computation of $\LBr(\Mscr, \Oscr) =
    \pi_0\ShLBr_{\Oscr}(\Mscr)$. Up to a onefold shift, this agrees with the Picard spectral
    sequence for $\TMF$ from \cite{mathew-stojanoska}. Note that $\H^1(\Mscr; \ZZ/2) = 0$ (since
    $\Mscr$ has no finite covers) and $\H^2(\Mscr, \Gm) = \Br(\Mscr) = 0$ by
    \cite{antieau-meier}. The non-sheafy version of \cref{prop:thingsvanishhighabove} holds by the same arguments and thus
    only terms of filtration at most $30$ can survive in the Picard spectral sequence in column $(-1)$. By the results from \cite{mathew-stojanoska}, we know all
    differentials from the $0$-column to the $(-1)$-column of the Picard spectral sequence: up to
    row $30$ there are only $d_3$ and they are $2$-local (cf.\ especially Figure 6 to 10 in
    \cite{mathew-stojanoska}).  One
    thus observes that the $p$-torsion is as stated for $p\geq 3$. In the $\E_{\infty}$-term,
    we have $2$-locally the kernel of an unknown $d_9$-differential from $\coker(d_3\colon \FF_2[j] \to \FF_2[j] \oplus \ZZ/2)$ in row 6 to a $\ZZ/2$ in row 15 (which
    must be abstractly isomorphic to $(\ZZ/2)^{\infty}$, as in the proof of \cref{eq:cohkv} below), and further copies of $\ZZ/2$ in rows 10, 18 and 30,
    which cannot support differentials. Here, we use \cite[Comparison Tool
    5.2.4]{mathew-stojanoska}, both to show that possible targets of differentials vanish and to
    show the vanishing of a possible $d_5$ on the class in row 10. This implies the result.
\end{proof}

Next, we give a similar (but less precise) computation for $\LBr(\TMF)$. Later, we will compare the two calculations.

\begin{theorem}\label{thm:LBrTMF}
    The local Brauer group $\LBr(\TMF)$ is a torsion group. There is no $p$-torsion for $p>3$. The
    $3$-torsion is $\ZZ/3$. Moreover, there is a split surjection $\LBr(\TMF)_{(2)}\to
    (\ZZ/2)^{\infty}$ with finite kernel.
\end{theorem}

\begin{proof}
    By \cref{prop:omni} and using $\Br(\pi_0\TMF) \cong \Br(\ZZ[j]) \cong \Br(\ZZ) =0$ by
    \cref{thm:BrauerProperties} and \cref{ex:BrauerComputations}, $\LBr(\TMF)$ is isomorphic to the
    kernel of the differential
    \[d^{\TMF}\colon \H^1(\Spec \ZZ[j]; \pi_0j_*\Shpic_{\Oscr_\Mscr}) \to \H^3(\Spec \ZZ[j]; \Gm),\] 
    where we use~\eqref{eq:picpush} to identify $\pi_0\ShPic_{\Oscr_\TMF}$ with
    $\pi_0j_*\Shpic_{\Oscr_\Mscr}$.
    We will first partially compute the source of the differential. Using \cref{lem:R1Gm}, the facts
    that $\H^1(\Spec \Z[j]; \ZZ/m) = \H^1(\Spec \ZZ;\ZZ/m) = 0$ for any $m$, and
    \cref{thm:closedimmersion}, we deduce first that $\H^1(\Spec \ZZ[j]; \R^1j_*\Gm)$ vanishes. From
    \cref{thm:PicTMF} it thus follows that $\H^1(\Spec \ZZ[j]; \pi_0j_*\Shpic_{\Oscr_\Mscr}) \cong
    \H^1(\Spec \ZZ[j];\F^3\pi_0j_*\Shpic_{\Oscr_\Mscr})$, where $F^3$ refers to the third
    filtration.  The sheaf $\F^3\pi_0j_*\Shpic_{\Oscr_\Mscr}$ sits in an extension
    \begin{equation}\label{eq:extsheaves} 0 \to\F^5\pi_0j_*\Shpic_{\Oscr_\Mscr}  \to\F^3\pi_0j_*\Shpic_{\Oscr_\Mscr} \to k_*v_!\ZZ/2 \to 0.\end{equation}
    The extension must be split since $(\F^5\pi_0j_*\Shpic_{\Oscr_\Mscr})_{(2)}$ is supported at
    $(2,j)$, while $k_*v_!\ZZ/2$ is only nonzero on \'etale maps $U \to \AA^1$ whose image does not
    contain $(2,j)$.

    To compute $\H^1(\AA^1; k_*v_!\ZZ/2)$, recall that we obtained $k_*v_!\ZZ/2$ as the kernel of a
    surjective differential $d_3\colon \Oscr/2 \to \Oscr/2$. Since $\Oscr/2$ is quasi-coherent, its
    first cohomology vanishes and we can thus identify $\H^1(\AA^1; k_*v_!\ZZ/2)$ with the cokernel
    of the map $d_3\colon \H^0(\AA^1; \Oscr/2) \cong \FF_2[j] \to \FF_2[j]\cong \H^0(\AA^1;
    \Oscr/2)$, which sends $f$ to $f + jf^2$ (cf.\ \cref{sec:Row3}). One checks that $j^2,
    j^4, j^6, \dots$ is a linearly independent subset in the cokernel and thus 
    \begin{equation}\label{eq:cohkv}\H^1(\AA^1; k_*v_!\ZZ/2)\cong \FF_2^{\infty}. \end{equation}
         
    Next, we turn to $\F^5\pi_0j_*\Shpic_{\Oscr_\Mscr}$. By \cref{thm:PicTMF}, this vanishes if
    localized at primes bigger than $3$, while $3$-locally it is isomorphic to $b_*\ZZ/3$ for
    $b\colon \Spec \FF_3 \to \AA^1$ the inclusion at $j=3=0$. Thus,
    \[\H^1(\AA^1;\F^5\pi_0j_*\Shpic_{\Oscr_\Mscr})_{(3)} \cong \H^1(\AA^1; b_*\ZZ/3) \cong
    \H^1(\FF_3; \ZZ/3) \cong \ZZ/3.\]
         
    For the $2$-local situation, recall from \cite[Corollary II.3.11]{milne-etale} that we can view
    sheaves supported at $(2,j)$ equivalently as \'etale sheaves on $\Spec \FF_2$,
    whose category is equivalent to (discrete) abelian groups with a continuous action by the absolute Galois
    group $\Gal(\FF_2) \cong \widehat{\ZZ}$; we refer to such as \emph{discrete $\widehat{\ZZ}$-modules}. 
    Let $\Fscr$ be the class of discrete $\widehat{\ZZ}$-modules where $\H^i(\widehat{\ZZ}, -)$ is
    finite for all $i$. From the fact that for discrete $\widehat{\ZZ}$-modules,
    $\H^i(\widehat{\ZZ}, -)$ vanishes for $i>1$,
    one deduces that $\Fscr$ is closed under kernels, cokernels and extensions. By \cref{prop:LongDiff},
    we know that only finitely many discrete $\widehat{\ZZ}$-modules can contribute to
    $(\F^5\pi_0j_*\Shpic_{\Oscr_\Mscr})_{(2)}$ and they all lie in $\Fscr$; moreover,
    there are only finitely many possible targets and they also lie in $\Fscr$ (cf.\
    \cref{sec:LongDiff}). Thus, $\H^1(\AA^1;\F^5\pi_0j_*\Shpic_{\Oscr_\Mscr})_{(2)}$ is finite.
         
    It remains to study the differential 
    \[d^{\TMF}\colon \H^1(\Spec \ZZ[j]; \pi_0j_*\Shpic_{\Oscr_\Mscr}) \to \H^3(\Spec \ZZ[j]; \Gm).\] 
    Let $U$ be the complement of the image of the closed immersion $\Spec \ZZ/6 \to \Spec \ZZ[j]$ corresponding to $j=0$. We obtain a commutative diagram 
    \[
    \xymatrix{
    \H^1(\AA^1;\F^5\pi_0\ShBPic_{\Oscr_{\TMF}}) \ar[r]\ar[d] &\H^1(\AA^1; \pi_0\ShBPic_{\Oscr_{\TMF}}) \ar[d]\ar[rr]^-{d^{
    \TMF}} &&\H^3(\Spec \ZZ[j]; \Gm)\ar[d] \\
    \H^1(\AA^1;\F^5\pi_0\ShBPic_{\Oscr_{\Mscr\times_{\AA^1}U}})\ar[r] &\H^1(U; \pi_0\ShBPic_{\Oscr_{\Mscr\times_{\AA^1}U}}) \ar[rr]^-{d^{(U, \Oscr_{\TMF}|_U)}} &&\H^3(U; \Gm)
    }\]
    The rightmost lower horizontal arrow is the differential in the descent spectral sequence for
    $\mathbf{BPic}$ on $(U, \Oscr_{\TMF}|_U)$. The rightmost vertical map is an injection by purity
    \cite[Th\'eor\`eme 6.1b]{grothendieck-brauer-3}. Moreover, the leftmost vertical map is zero
    since $\F^5\pi_0\ShBPic_{\Oscr_{\TMF}}$ is supported at $(2,j)$ and $(3,j)$. Thus, $d^{\TMF}$
    vanishes when restricted to $\H^1(\AA^1;\F^5\pi_0\ShBPic_{\Oscr_{\TMF}})$ and the differential
    factors over $\H^1(\AA^1; k_*v_!\ZZ/2)$.

    We can cover $U$ by $V = \Spec \ZZ[j^{\pm 1}, (j-1728)^{-1}]$ and $W = \Spec \ZZ[\frac16, j]$. We obtain an exact sequence
    \[\cdots \to \H^2(V\cap W; \GG_m) \to \H^3(U; \GG_m) \to \H^3(V; \GG_m) \oplus \H^3(W; \GG_m)\to \cdots .\]
    We claim that the image of $d^{\TMF}$ in $\H^3(V; \GG_m) \oplus \H^3(W; \GG_m)$ is zero.
    Assuming this claim for the moment, we know that the image of $d_3^{\TMF}$ lies in the image of
    $\H^2(V\cap W; \GG_m) \to \H^3(U; \GG_m)$. By \cref{thm:BrauerProperties}, we have $2$-locally
    an isomorphism $\H^2(V\cap W; \GG_m) \cong \Br(\ZZ[\frac16, j^{\pm 1}, (j-1728)^{-1}]) \cong
    \Br(\ZZ[\frac16]) \oplus \H^1(\ZZ[\frac16]; \QQ/\ZZ)^{\oplus 2})$. We use the following two
    computations:
    \begin{itemize}
        \item $\Br(\ZZ[\frac16]) \cong \QQ/\ZZ \oplus \ZZ/2$ by \cref{ex:BrauerComputations}; 
        \item $\H^1(\ZZ[\frac16]; \QQ/\ZZ) \cong \Hom(\Gal(K/\QQ), \QQ/\ZZ) \cong \Hom(\Z_2^{\times} \times \Z_3^{\times}, \QQ/\ZZ) \cong (\ZZ/2)^3 \oplus \QQ_2/\ZZ_2 \oplus \QQ_3/\ZZ_3$. Here, $K$ is the maximal abelian extension of $\QQ$, which is unramified at $2$ and $3$. We use the Kronecker--Weber theorem to identify $K$ with the field obtained by adjoining all $2^n$-th and $3^n$-th roots of unity to $\QQ$. 
    \end{itemize}
    Thus, the image of $\H^2(V\cap W; \GG_m) \to \H^3(U; \GG_m)$ is $2$-locally of the form (finite
    $\oplus$ divisible). Since the image of $d_3^{\TMF}$ must be an $\FF_2$-vector space (as the
    image of an $\FF_2$-vector space), the image of $d_3^{\TMF}$ must be finite. We deduce that the
    kernel of $d^{\TMF}$ consists of the finite group $\H^1(\AA^1;\F^5\pi_0j_*\Shpic_{\Oscr_\Mscr})$
    plus an infinite-dimensional subspace of $\H^1(\AA^1; k_*v_!\ZZ/2) \cong \FF_2^{\infty}$, as
    claimed.

    It remains to show that the restrictions of $d^{\TMF}$ to $V$ and $W$ are zero. The case of $W$
    is clear as $k_*v_!\ZZ/2$ is supported outside of $W$. For the case of $V$, recall from
    \cite[Lemma 3.2]{shin-brauer} that the base change $\Mscr\times_{\AA^1} V$ is equivalent to $V
    \times \B C_2$, i.e.\ the stack quotient of $V$ by the trivial $C_2$-action; this yields in
    particular an \'etale map $V\to \Mscr\times_{\AA^1} V \to \Mscr$. We obtain a diagram
         \[
         \xymatrix{
         \H^1(\AA^1;\F^3\pi_0\ShBPic_{\Oscr_{\Mscr}}) \ar[r]^{\cong}\ar[d] & \H^1(\AA^1; \pi_0\ShBPic_{\Oscr_{\Mscr}})\ar[rr]^{d^{\TMF}}\ar[d]&& \H^3(\AA^1; \Gm)\ar[d]\\
         \H^1(V;\F^3\pi_0\ShBPic_{\Oscr_{\Mscr\times_{\AA^1}V}}) \ar[r]\ar[d] &\H^1(\AA^1; \pi_0\ShBPic_{\Oscr_{\Oscr_{\Mscr\times_{\AA^1}V}}}) \ar[d]\ar[rr]^-{d^{
                (V,\Oscr_{\TMF})}} &&\H^3(V; \Gm)\ar[d]^{\mathrm{id}} \\
         \H^1(V;\F^3\pi_0\ShBPic_{\Oscr_{V}})\ar[r] &\H^1(V; \pi_0\ShBPic_{\Oscr_{V}}) \ar[rr]^-{d^{\Oscr(V)}} &&\H^3(V; \Gm).
     }\]
     Here, $d^{\Oscr(V)}$ refers to the boundary map in the long exact sequence from \cref{prop:omni}
     for the ring spectrum $\Oscr(V \to \Mscr)$, while $d^{(V, \Oscr_{\TMF})}$ uses the restriction
     of the spectral scheme structure of $\Spec \TMF$ to $V$; note that both affine spectral schemes
     here have underlying scheme $V$. In particular, the rightmost vertical map is an isomorphism.
     Note further that $\F^3\pi_0\ShBPic_{\Oscr_{V}} =0$ since all terms in the sheafy Picard
     spectral sequence of filtration $3$ and higher are of the form $\H^{2i+1}(V; \pi_{2i}\Oscr)$
     for $i\geq 1$, which all vanish since $V$ is an affine scheme. Thus, we see that $d^{\TMF}$ is
     indeed zero after restricting to $V$.
\end{proof}

Our next goal is to compare $\LBr(\Mscr, \Oscr)$ with $\LBr(\TMF)$. Clearly, we have maps
\[\LBr(\TMF) \to \LBr(\Mscr, \Oscr) \to \Br(\Mscr, \Oscr).\]
Since $\Br(\TMF) \to \Br(\Mscr, \Oscr)$
is an isomorphism, $\LBr(\TMF) \to \LBr(\Mscr, \Oscr)$ is an injection. We want to describe how to
obtain a computational handle on this injection. In conjunction with \cref{thm:LBrMO1},
this will also provide an alternative proof of \cref{thm:LBrTMF}.

Consider the sheaf $j_*\Shlbr_{\Oscr}$ on $\AA^1$. It assigns to every \'etale open $U$, the
spectrum $\Shlbr(U\times_{\AA^1}\Mscr, \Oscr)$. The relative descent spectral sequence (cf.
\cref{rem:relative-descent}) for  $j_*\Shlbr_{\Oscr}$ takes the form
\[\E_2^{s,t} = \R^sj_*\pi_t \Shlbr_{\Oscr} \quad \mathlarger{\Rightarrow} \quad \pi_{t-s} j_* \Shlbr_{\Oscr}   \underset{t-s\geq
    0}{\iso}\pi_{t-s}j_*\ShLBr_{\Oscr}\] 
and provides thus a method to compute $\pi_*j_*\Shlbr_{\Oscr}$. But since $\Shlbr$ is just a
suspension of $\Shpic$, this spectral sequence is up to a shift actually the same as the sheafy
Picard spectral sequence considered in \cref{sec:PicSheafTMF}. In particular, one observes that
$\pi_t j_*\Shlbr_{\Oscr} \cong \pi_{t-1} \ShPic_{\Oscr_{\TMF}}$ for $t\geq 1$, but have additionally
interesting sheaves $\pi_t$ for $t\leq 0$, which are computed by the $(t-1)$-column of the sheafy
Picard spectral sequence. We obtain a descent spectral sequence
\begin{equation}\label{eq:DSSforLBrMO}
    \E_2^{s,t} = \H^s(\AA^1; \pi_tj_*\Shlbr_{\Oscr}) \quad \mathlarger{\Rightarrow} \quad \pi_{t-s} \Gamma(j_* \Shlbr_{\Oscr})  \underset{t-s\geq
    0}{\iso}\pi_{t-s}\Shlbr(\Mscr, \Oscr).\end{equation}
In particular, \cref{prop:BrIdentifications} gives that $\pi_0\Shlbr(\Mscr, \Oscr) = \LBr'(\Mscr,
\Oscr) = \LBr(\Mscr, \Oscr)$. Thus, we are indeed computing the local Brauer group of $(\Mscr,
\Oscr)$.

\begin{figure}[H]
\begin{sseqpage}[Adams grading]
\foreach \x in {3, 4, 5}  {
\class[fill, blue](\x,0)
}
\foreach \x in {1,2} \foreach \y in {0,1, 2, 3, 4, 5} {
\class[fill, blue](-\y+\x,\y)
}
\foreach \x in {-2, -3, -4} \foreach \y in {0,1, 2} {
\class(\x,-\x+\y-2)
}
\class(0,0) \d2
\class(-1,0)
\class(-1,1) 
\end{sseqpage}
\caption{Schematic comparison of descent spectral sequences computing $\LBr(\TMF)$ (solid, in blue) and $\LBr(\Mscr, \Oscr)$}
\label{fig:comparison}
\end{figure}
    
Note that the map $\Shlbr_{\Oscr_{\TMF}} \to j_*\Shlbr_{\Oscr}$ induces a map of descent spectral
sequences, which is essentially the inclusion of the top two anti-diagonals.
\cref{fig:comparison} gives a schematic picture of part of this map, with
the image of the descent spectral sequence of $\Shlbr_{\Oscr_{\TMF}}$ colored in blue.

\begin{theorem}\label{thm:LBrMO}
    The injection $\LBr(\TMF) \to \LBr(\Mscr, \Oscr)$ has finite cokernel and is an isomorphism after inverting $2$. 
\end{theorem}

\begin{proof}
    In the spectral sequence \eqref{eq:DSSforLBrMO}, the only possible nonzero entries in the
    zeroth column are $\H^s(\AA^1, \pi_sj_*\Shlbr_{\Oscr})$ for $0\leq s\leq 2$. By the above
    discussion, every element in $\LBr(\Mscr, \Oscr)$ not coming from $\LBr(\TMF)$ must be
    detected in $\H^0(\AA^1, \pi_0j_*\Shlbr_{\Oscr})$.
    
    The charts \cref{fig:e5tmfss3} and \cref{fig:e7tmfss} show the possible contributions in the
    $(-1)$-column of the sheafy Picard spectral sequence to $\pi_0j_*\Shlbr_{\Oscr}$. We first
    note that the two question marks corresponding to $\R^1j_*\ZZ/2$ and $\R^2j_*\GG_m$ cannot
    contribute to $\H^0$. Indeed, by the Leray spectral sequence, $0 = \H^1(\Mscr; \ZZ/2)$
    surjects onto $\H^0(\AA^1; \R^1j_*\ZZ/2)$, which is thus zero as well. Likewise, from the
    Leray spectral sequence
    \[ \H^m(\AA^1; \R^nj_*\GG_m) \quad \Rightarrow \quad \H^{m+n}(\Mscr; \GG_m) \]
    we see that $\Br(\Mscr) \cong \H^2(\Mscr;\Gm)$ surjects onto the cokernel of the
    differential $\H^1(\AA^1, j_*\Gm) \to \H^0(\AA^1; \R^2j_*\Gm)$. But $j_*\Gm \cong \Gm$ (since
    every function $\Mscr \to \AA^1$ factors through $j$, even after \'etale base change) and
    thus $\H^1(\AA^1, j_*\Gm) \cong \Pic(\AA^1) =0$. Moreover, $\Br(\Mscr) = 0$ by one of the
    main results from \cite{antieau-meier}.  Thus, $\H^0(\AA^1; \R^2j_*\Gm)$ vanishes.
    
    Regarding the contributions in higher rows: even the $\E_2$-term vanishes $p$-locally for
    $p>3$. At $p=3$, the only potential contribution is in Row 14 (see \cref{fig:e5tmfss3} and
    \cref{prop:LongDiff}). As demonstrated in \cref{lem:3d9}, this contribution is hit by a
    surjective $d_9$. Thus, $\F^3 \pi_0j_*\Shlbr_{\Oscr}$ vanishes after inverting $2$. We deduce
    $\H^0(\AA^1, \pi_0j_*\Shlbr_{\Oscr})[\frac12] = 0 $ and hence that $\LBr(\TMF) \to
    \LBr(\Mscr, \Oscr)$ is an isomorphism after inverting $2$.
    
    Regarding the $2$-local picture, \cref{fig:e7tmfss} shows that the only possible
    contributions are in Rows $6$, $18$ and $30$ and each of them is an $\Oscr/(2,j)$ on the
    $\E_7$-page. The same argument as provided in the proof of \cref{thm:LBrTMF} for the
    finiteness of $\H^1(\AA^1;\F^5\pi_0j_*\Shpic_{\Oscr_\Mscr})_{(2)}$ shows also the finiteness
    of $\H^0(\AA^1, \pi_0j_*\Shlbr_{\Oscr})$. This in turn implies the finiteness of the
    cokernel of $\LBr(\TMF) \to \LBr(\Mscr, \Oscr)$.
\end{proof}

\bibliographystyle{amsplain}
\bibliography{bibl}
\end{document}